\documentclass[12pt]{amsart}
\usepackage{latexsym,amsmath,amsfonts,amsthm,amssymb,color}
\usepackage{fullpage}
   
\newtheorem{theorem}{Theorem}
\newtheorem{lemma}{Lemma}[section]

\newtheorem{proposition}[lemma]{Proposition}

\numberwithin{equation}{section}

\renewcommand{\P}{\tilde P}

\newcommand{\R}{{\mathbb R }}

\newcommand{\Z}{{\mathbb Z }}

\renewcommand{\H}{{\mathcal H }}
\renewcommand{\S}{{\mathbb S }}
\newcommand{\p}{\partial}
\renewcommand{\aa}{\alpha}
\newcommand{\bb}{\beta}

\newcommand{\la}{\langle}
\newcommand{\ra}{\rangle}

\renewcommand{\la}{\langle}
\renewcommand{\ra}{\rangle}

\newcommand{\step}[2]{\noindent\textbf{Step #1:}(\emph{#2})}

\newcommand{\tphi}{\widetilde{\phi}}

\newcommand{\lp}[2]{\Vert \, #1 \, \Vert_{#2}}

\newcommand{\snabla}{ {\slash\!\!\!\!\nabla} }

\newcommand{\g}{\mathfrak  g}
\newcommand{\G}{\mathbf{G}}
\renewcommand{\P}{\mathbf{P}}
\newcommand{\sgn}{\text{sgn}}

\begin{document}
\title[The MKG equation]{ Global well-posedness for the Yang-Mills
 equation in $4+1$ dimensions. Small energy.}

\author{Joachim\ Krieger} \address{B\^{a}timent des Math\'ematiques,
  EPFL, Station 8, CH-1015 Lausanne, Switzerland}
\email{joachim.krieger@epfl.ch}

\author{Daniel\ Tataru} \address{Department of Mathematics, The
  University of California at Berkeley, Evans Hall, Berkeley, CA
  94720, U.S.A.}  \email{tataru@math.berkeley.edu}

\thanks{The first author was partially supported by the Swiss National
  Science Foundation. The second author was supported in part
  by the NSF grant DMS-1266182 as well as by a Simons Investigator
  grant from the Simons Foundation.}

\begin{abstract}
  We consider the hyperbolic Yang-Mills equation on
the Minkowski space  $\R^{4+1}$. Our main result asserts
that this problem  is globally well-posed for all initial data whose energy 
is sufficiently small.  This solves a longstanding open problem.
\end{abstract}

\maketitle


\section{Introduction}

Let $\G$ be a semisimple Lie group and $\g$ its associated Lie
algebra. We denote by $ad(X) Y = [X,Y]$ the Lie bracket on $\g$ and by
$\la X,Y \ra = tr(ad(X) ad(Y)) $ its associated nondegenerate Killing
form.  The action of $\G$ on $\g$ by conjugation is denoted by $Ad(O)X
= O X O^{-1}$.  We recall that the Killing form is invariant, in the sense that
\begin{equation*}
\la [X,Y],Z \ra = \la X, [Y,Z] \ra, \qquad X,Y,Z \in \g, 
\end{equation*}
or equivalently 
\begin{equation*}
\la X,Y \ra = \la Ad(O) X,  Ad(O) Y  \ra, \qquad X,Y \in \g, \quad O \in \G. 
\end{equation*}

Let $\R^{4+1}$ be the five dimensional Minkowski space equipped with
the standard Lorentzian metric $m= \text{diag}(-1,1,1,1,1)$.
Denote by $A_{\alpha}: \R^{4+1}\rightarrow \g$, $\alpha = 1\ldots,4$,
a connection form taking values in the Lie algebra $\g$, and by $D_\alpha$
the associated covariant differentiation,
\[
D_{\alpha} B:= \partial_{\alpha} B + [A_{\alpha},B],
\]
acting on $\g$ valued functions $B$.
Introducing the curvature tensor
\[
F_{\alpha\beta}: = \partial_{\alpha}A_{\beta}
- \partial_{\beta}A_{\alpha} +[A_\alpha,A_\beta],
\]
the {\it{Yang-Mills}} equations are the Euler-Lagrange
equations associated with the formal Lagrangian action functional
\[
\mathcal{L}(A_{\alpha}, \phi): = \frac{1}{2}\int_{\R^{4+1}} \la F_{\alpha\beta}, F^{\alpha\beta}\ra \,dxdt.
\]
Here we are using the standard convention for raising indices. Thus,
the Yang-Mills equations take the form
\begin{equation}\label{ym}
D^\alpha F_{\alpha \beta} = 0.
\end{equation}
There is a natural energy-momentum tensor associated to the Yang-Mills
equations, namely 
\[
T_{\alpha \beta} =    \frac12  m^{\gamma \delta} \la F_{\alpha \gamma}, F_{\delta \beta}\ra
- \frac14 m_{\alpha \beta} \la F_{\gamma \delta}, F^{\gamma \delta}\ra.
\]
If $A$ solves the Yang-Mills equations \eqref{ym} then $T_{\alpha \beta}$ is divergence free,
\begin{equation}\label{divT}
\partial^\alpha T_{\alpha \beta} = 0.
\end{equation}
Integrating this for $\beta = 0$ yields a conserved energy
\begin{equation}\label{energy}
E(A) = \int_{\R^4} T_{00}\, dx \approx \|F\|_{L^2}^2.
\end{equation}
 The case $\beta \neq 0$ yields further conservation laws, i.e. the momentum,
which play no role in the present article.

 The Yang-Mills equations also have a scale invariance property,
\[
A(t,x) \to \lambda A(\lambda t,\lambda x).
\]
The energy functional $E$ is invariant with respect to scaling precisely in dimension
$4+1$. For this reason we call the $4+1$ problem energy critical;  this is one of the
motivations for our interest in this problem.

In order to study the Yang-Mills equations as well-defined evolutions
in time we first need to address its gauge invariance. Precisely, the
equations \eqref{ym} are invariant under the gauge transformations
   \[
A_\alpha \longrightarrow O A_\alpha O^{-1} -\partial_\alpha O  O^{-1},
\]
with $O$  elements  of the corresponding   group $G$. In order to uniquely 
determine the solutions to the Yang-Mills equations we need to add an additional 
set of constraint equations which uniquely determine the gauge. This 
procedure is known as {\em gauge fixing}.

To motivate our choice we introduce the covariant wave operator
\[
\Box_{A}: = D^{\alpha}D_{\alpha}.
\]
Then we can write the Yang-Mills system in the following form
\begin{equation}\label{eq:ym}
\begin{split}
 \Box_A A_\beta = D^\alpha \partial_\beta A_\alpha = \partial_\beta \partial^\alpha A_\alpha +
[A^\alpha,\partial_\beta A_\alpha].
\end{split}
\end{equation}
Expanded out, the equations take the form
\[
\Box A_\bb - \p_\bb \p^\aa A_\aa + \p^\aa  [A_\aa,A_\bb] +
 [A^\aa,  \p_\aa A_\bb - \p_\bb A_\aa + [A_\aa,A_\bb]] = 0,
\]
or 
\[
\Box A_\bb + 2 [ A_\aa,\p^\aa A_\bb] = 
 \p_\bb \p^\aa A_\aa -[\p^\aa  A_\aa,A_\bb] 
 + [A^\aa,  \p_\bb A_\aa]  - [A^\aa,[A_\aa,A_\bb]]. 
\]

A natural condition which insures that the above system is strictly
hyperbolic is the Lorenz gauge, $\partial^\alpha A_\alpha =
0$. Unfortunately there are multiple technical difficulties if one
tries to implement such a gauge in the low regularity setting, see
e.g. \cite{2013arXiv1309.1977S}.  For this reason we will instead
impose the {\it{Coulomb Gauge condition}} which requires
\begin{equation}\label{eq:Coulomb}
  \sum_{j=1}^4\partial_j A_j = 0.
\end{equation}
We remark that a somewhat similar gauge is the {\em temporal gauge}, namely $A_0 = 0$.
Another choice which is likely better but more involved technically is the caloric gauge,
see e.g. \cite{Oh12}.

Returning to the Coulomb gauge, we can use it to view the equations as a 
nonlocal hyperbolic system for the spatial components $A_j$; precisely, they 
solve the system
\begin{equation*}
\Box_A A_j  = \p_j \p_t A_0  +  [A^\aa,  \p_j A_\aa].  
\end{equation*}
In order to eliminate the first term on the right and also to restrict the evolution 
to divergence free fields $A_j$ we apply the Leray projection $\P$, and rewrite  
the equation in the form
\begin{equation}\label{ym-cg}
\Box A_j  = \P\left(  [A^\aa,  \p_j A_\aa] - 2 [ A^\alpha, \partial_\alpha A_j] 
- [ \partial_0 A_0, A_j]  - [A^\alpha,[A_\alpha,A_j]] \right).  
\end{equation}
The nonlocality is due to the $A_0$ component, which solves an elliptic 
equation at fixed time, namely
\begin{equation}\label{a0}
\Delta_A A_0 = [A_j,\partial_0 A_j].
\end{equation}
The time derivative of $A_0$ also appears in the $A_j$ system,
so it is useful to derive an equation for it as well. This has the form
\begin{equation}\label{dta0}
\Delta \partial_0 A_0 = \partial_0 \partial_j [ A_0,A_j].
\end{equation}

To summarize, in the Coulomb gauge, the Yang-Mills system can be cast in the following expanded out form:
\[
\begin{split}
\Box A_i +   2 [ A_\aa,\p^\aa A_i]  = & \ \p_i \p_t A_0 + [\p_0 A_0, A_i]
+  [A^\aa,  \p_i A_\aa] -  [A^\aa,[A_\aa,A_i]],
\\
\Delta A_0 + 2 [A_i,\p_i A_0]  =& \  [A_i,\p_0 A_i] - [A_i,[A_i,A_0]]. 
\end{split}
\]

We will consider the solvability question for the system \eqref{ym-cg}
in the class of divergence free vector fields, with initial data at time $t = 0$,
\begin{equation}\label{data}
(A_j(0), \partial_0 A_j(0)) = (A_{0j}, A_{1j}) \in \H:= \dot H^1(\R^4) \times L^2(\R^4).
\end{equation}
We will also consider higher regularity properties of the solutions, using the
spaces 
\[
 \H^N:= (\dot H^N(\R^4)  \cap \dot H^1(\R^4) ) \times  H^{N-1}(\R^4) , \qquad N \geq 1
\]

Here the dependent variables $A_0$, $\partial_0 A_0$ are determined by the linear equations
\eqref{a0}, \eqref{dta0}. We remark that the solvability for these
equations in various spaces, including $\dot H^1 \times L^2$ at fixed
time, is considered in Section~\ref{s:elliptic}.

In order to study the dependence of the solutions on the 
initial data we will also need  the linearized Yang-Mills equation,
\begin{equation}\label{ym-cg-lin} \begin{split}
\Box B_j  = \ &  \P\left(  [A^\aa,  \p_j B_\aa] - 2 [ A^\alpha, \partial_\alpha B_j] 
- 2 [ B^\alpha, \partial_\alpha A_j]
- [ \partial_0 A_0, B_j] - [ \partial_0 B_0, A_j] \right. \\ \ & \left. - 2[B^\alpha,[A_\alpha,A_j]] - 
 [A^\alpha,[A_\alpha,B_j]] \right)  
\end{split}
\end{equation}
with appropriate linear elliptic equations for $B_0$, $\partial_0 B_0$,
\begin{align}\label{b0}
\Delta_A B_0 = & \ [B_j,\partial_0 A_j] + [A_j,\partial_0 B_j] + 2 [B_j,\partial_j A_0] + 
2 [B_j,[B_j,A_0]],
\\
\label{dtb0}
\Delta \partial_0 B_0 = & \  \partial_0 \partial_j( [ B_0,A_j]+[ A_0,B_j]).
\end{align}
For the linearized equation we will go below scaling in regularity, and use the spaces 
\[
\dot \H^s = \dot H^s(\R^4)  \times \dot H^{s-1}(\R^4) ,
\]
with $s < 1$ but close to $1$. Now we can state our main result:

\begin{theorem}\label{t:main}
The Yang-Mills system in Coulomb gauge \eqref{ym-cg}-\eqref{a0}-\eqref{dta0}
is globally well-posed in $\H$  for  initial data which is small 
in $\H$, in the following sense:

(i) (Regular data) If in addition the data $(A_{0j}, A_{1j})$ is more
regular, $(A_{0j}, A_{1j}) \in \H^N$, then there exists a unique
global regular solution $(A_j,\partial_0 A_j) \in C(\R, \H^N)$, which
has a Lipschitz dependence on the initial data locally in time in the
$\H^N$ topology.

(ii) (Rough data) The flow map admits an extension
\[
\H \ni (A_{j0}, A_{j1}) \to (A_j, \partial_t A_j) \in C(\R, \H)
\]
within the class of initial data which is small in $\H$, and which is continuous
in the $\H \cap \dot \H^s$ topology for $s <1$ and close to $1$.

(iii) (Weak Lipschitz dependence) The flow map is globally Lipschitz in the 
$\dot \H^s$ topology for $s < 1$, close to $1$.
\end{theorem}

To clarify, in part (ii) the $\H \cap \dot \H^s$ norm is applied to
differences of solutions.  In particular, we remark that $\H^N$ is
dense in $\H$ in this topology, so this extension yields solutions for all small
data in $\H$. The $\dot \H^s$ norm plays an essential role here, as
this is the norm where we have Lipschitz dependence of the solutions
on the initial data.  If we limit ourselves to just the $\H$ topology,
then the best we can prove is a local in time continuous dependence on
data; thus, the scattering information is lost.

We remark that in effect the proof of the theorem provides a stronger
statement, where the regularity of the solutions is described in terms
of function spaces $S^1$, $S^N$ which incorporate both Strichartz norms,
$X^{s,b}$ norms and null frame spaces.  For convenience, the stronger
result is stated later in Theorem~\ref{t:main-sn}. 

Implicit in Theorem ~\ref{t:main-sn} is also a scattering result;
however, this is not so easy to state as it is a modified rather than
linear scattering. In a weaker sense, one can think of scattering as
simply the fact that the $S^1$ norm is finite.

\subsection{Brief historical remarks}
The Yang-Mills equation belongs to the larger class of geometric 
nonlinear wave equations, which includes other problems such as 
Wave-Maps and the (mass-less) Maxwell-Klein-Gordon system. These problems
have a number of shared features, including the gauge structure,
and the null condition. Also, in all these problems the nonlinearity 
is nonperturbative at critical scaling, though only mildly so, more 
precisely in a way which can be addressed via renormalization. 
For these reasons, the understanding of these problems has evolved in a 
related fashion, and, as we describe below, our work on Yang-Mills was strongly
influenced by prior developments for  both Wave-Maps and Maxwell-Klein-Gordon. 

For the Yang-Mills equation, a first global regularity result on a
Minkowski background in the physical dimension $n = 3$ was first
established for large data in classical work by Eardley-Moncrief,
\cite{MR649158}, \cite{MR649159}, after earlier work by Choquet-Bruhat
and Christodoulou had proved a small data global existence result in
\cite{ChoChr}. The physical $n = 3$ case is energy subcritical, which
makes this problem easier from the point of view of global existence
than the critical case $n = 4$, but harder from the point of view of
understanding scattering.  

The Eardley-Moncrief result was revisited and significantly
strengthened by Klainerman-Machedon \cite{Klainerman:1995}. In fact,
these authors showed local (and thence global) well-posedness in
$H^1$. This work proved important for future developments on account
of the fact that it identified the null-structure and its use via
bilinear null-form estimates, which is also of paramount importance in
this work.  The energy critical case $n = 4$ of the Yang-Mills system
was first attacked in Klainerman-Tataru~\cite{Klainerman:1999do}; more
precisely, a model system with similar null-structures was considered
there, and almost optimal local well-posedness (in light of the
scaling of the system) was shown.  Somewhat later, Machedon-Sterbenz
\cite{Machedon:2004cu} revisited the closely related subcritical
Maxwell-Klein Gordon system in $3+1$ dimensions, and exploiting a deep
trilinear null-structure in the system, managed to push local
well-posedness all the way to an almost optimal
$H^{\frac{1}{2}+\epsilon}$-result (optimal in light of scaling). The
new null-structure used there will also be of fundamental importance
for our work.

Further work on the
Maxwell-Klein-Gordon and Yang-Mills equation followed in the wake of
important progress on the Wave Maps equation by the second author in
\cite{MR1827277,wm} as well as by Tao in \cite{Tao:2001gb}. These  works
introduced the functional framework that will be crucial for the
present paper.  In \cite{MR2100060}, Rodnianski and Tao established an
optimal small data global existence result at the scaling invariant
level for high-dimensional Maxwell-Klein-Gordon in the Coulomb
Gauge. The important innovation there was the use of an approximate
parametrix for a magnetic potential wave equation to deal with certain
bad interaction terms which could not be handled perturbatively. By
refining this and working with more sophisticated Banach spaces coming
from the theory of Wave Maps, the authors jointly with J. Sterbenz
pushed this to the energy critical case in $n = 4$ dimensions in
\cite{MKG}.

The present paper will borrow quite heavily from \cite{MKG}, and in
fact be built directly on the spaces and null-form estimates
established there.  However, the geometry for the Yang-Mills system is
significantly more complicated than for the Maxwell-Klein-Gordon
system, as the field $A$ no longer 'essentially behaves like a free
wave'. An adaptation of the method of \cite{MR2100060} to global
regularity for small critical data of high dimensional ($n\geq 6$)
Yang-Mills was accomplished in
Krieger-Sterbenz~\cite{Krieger:2005wh}. In the present paper we use an
approximate parametrix of the same type as in \cite{Krieger:2005wh}.
However, in its construction we take advantage of  the better
functional framework in \cite{MKG}, as well as of better connection integration
techniques borrowed from Wave-Maps~\cite{wm}.

The small data result in the present paper can also be viewed as a
stepping stone toward the corresponding large data problem, which is
still open. The large data problem is better understood for the
Wave-Map equation, where the so-called Threshold Conjecture was
recently proved by Sterbenz-Tataru~\cite{MR2657817,MR2657818}, \cite{}
and also, independently, by Krieger-Schlag~\cite{Krieger:2009uy} and
Tao~\cite{Tao:2008wn,Tao:2008tz,Tao:2008wo,Tao:2009ta,Tao:2009ua} for
special target manifolds. More recently, large data well-posedness was
also established for the Maxwell-Klein-Gordon system, independently in
Oh-Tataru~\cite{OT1,OT2,OT3} and Krieger-Luhrman~\cite{KL}.
 
In related developments, one should also note the work of
Bejenaru-Herr~\cite{BH},\cite{BH1} on the closely related cubic Dirac
equation, as well as the massive Dirac-Klein-Gordon system.

\subsection{Ingredients of the proof}

The present paper is built directly on the predecessor paper
\cite{MKG}.  The nonlinearity is split into two parts, a perturbative
one and a non-perturbative paradifferential type component. As in
\cite{MKG}, even the ``perturbative'' part cannot directly estimated
in full. Instead, there is a portion of it that requires reiteration
of the equation and the use of the second null condition.  The
nonperturbative part is then eliminated via a paradifferential gauge
renormalization.

The main novelty here then concerns the approximate parametrix
construction for the magnetic potential wave equation
\eqref{parab-noP}, which is considerably more difficult in the present
noncommutative setting. We use an ansatz \eqref{eq:Bparametrix} as in
\cite{Krieger:2005wh}, but construct the phase shift $O(t, x, \xi)$
via a continuous version of the 'discretized' (over frequency blocks)
Gauge construction in \cite{Tao:2001gb}, see \eqref{psi-def}. Such a
construction was first introduced in \cite{wm}, and proved its
usefulness further in \cite{MR2657817}. The fact that the angular
separation in the definition of the $\Psi_k$ can be chosen as
$2^{\delta k}$ with $\delta>0$ arbitrarily small simplifies the
arguments for the control of the parametrix in section 7 compared to
the arguments in \cite{Krieger:2005wh}.


\subsection{Notation and Conventions}
We use the notation $A \lesssim B$ to mean $A \leq CB$ for some
universal constant $C > 0$.  We write $A \ll B$ if the implicit
constant should be regarded as small.

Our convention regarding indices is as follows. The greek indices
$\alpha, \beta$ run over $0, \ldots, 4$, whereas the latin
indices $i, j$ only run over the spatial indices $1, \ldots,
4$. We raise and lower indices using the Minkowski metric, and sum
over repeated upper and lower indices. The indices $k,h,l$ are reserved 
for dyadic frequencies.

For the space-time Fourier variables we will use $(\tau,\xi)$ or
$(\sigma,\eta)$. On occasion we set $\tau = \xi_0$, or $\sigma = \eta_0$;
we will do this only to keep the notations simple where there is covariant
summation with respect to indices $\alpha$, $\beta$.

\subsubsection*{Littlewood-Paley projections}

We denote by $P_k = P_k(D_x)$ the standard spatial Littlewood Paley projections,
where $k$ is a dyadic index. We allow $k$ to be either discrete (integer)
or continuous. We also use the notations $P_{<k}$, $P_{>k}$ for projections
selecting lower or higher frequencies. 

On occasion we will also need space-time  Littlewood Paley projections.
These are denoted by $S_k:= S_k(D_{x,t})$, $S_{<k}$, $S_{>k}$.

We also define modulation Littlewood Paley projections,
$Q_j:=Q_j(|D_t|-|D_x|)$. Sometimes we will restrict these to positive
or negative time frequencies, $Q^{\pm}_{j} := Q^{\pm} Q_{j}$, where
$Q_{\pm} := \mathcal{F}^{-1}[ 1_{[0, \infty)}(\pm \tau) \mathcal{F}[\varphi]]$
restricts to the $\pm$ frequency half-space.

\subsubsection*{Frequency envelopes}
For some more accurate bounds at various places we need to keep better
track of the dyadic frequency distribution of norms. This is done using the
language of frequency envelopes. An \emph{admissible frequency
  envelope} will be any sequence $\{c_k\}_{k \in \mathbb Z}$ of
positive numbers which is slowly varying upwards,
\[
2^{-C_0(j-k)} \leq c_j /c_k \leq 2^{\delta_{0} (j-k)}, \qquad j > k,
\]
with a large universal constant $C_0$ and a small universal constant
$\delta_{0}$. Given such a sequence and a norm $X$, we define the norm
\[
\|\phi\|_{X_c} = \sup_k c_k^{-1} \| P_{k} \phi\|_{X}.
\]
We say that $c$ is a frequency envelope for the data $A_{x}[0]$ if for
every $k \in \mathbb{Z}$, we have
\begin{equation*}
  \|(P_{k}A_{x}[0], P_{k}\phi[0])\|_{\H} \leq c_{k}.
\end{equation*}
Given any $A_{x}[0], \phi[0] \in \dot{H}^{1} \times L^{2}$, we may
construct such a $c$ by
\begin{equation*}
  c_{k} := \sum_{k'>k} 2^{-\delta_{0}|k - k'|} \| P_{k'}A_{x}\|_{\H} \ 
+  \sum_{k'\leq k} 2^{-C_{0}|k - k'|} \| P_{k'}A_{x}\|_{\H}.
\end{equation*}
By Young's inequality, we have $\|c\|_{\ell^{2}} \lesssim
\|A_{x}[0]\|_{\dot{H}^{1} \times L^{2}} $.

\subsubsection*{Lie group and algebra notations}

We use the notation $ad(A)B = [A,B]$ for the Lie bracket on $\g$, and its
interpretation as a representation of $\g$ as a subspace of $Aut(\g)$.
The Killing form
\[
\la A,B\ra = tr(ad(A)ad(B))  
\]
is nondegenerate if $\G$ is semisimple, and (with a possible sign adjustment)
can be used as an invariant inner product on $\g$. It also has the invariance property
\[ 
\la [A,B],C\ra = \la A,[B,C]\ra.
\]
The action of $\G$ on $g$ is denoted by $Ad(O) A = OAO^{-1}$. This
preserves Lie brackets and the Killing form. 

We also need to work with $\G$ valued functions and symbols $O(t,x,\xi)$.
To differentiate $O$ we introduce the notations  
\[
O_{;x} = \partial_xO O^{-1}, \qquad O_{;\xi} = \p_{\xi} O O^{-1}, etc.
\]
These are all well defined elements of the Lie algebra $\g$. Furthermore,
for any two such derivatives we have the commutation relation
\begin{equation}\label{comm}
\partial_k O_{;l} - \partial_l O_{;k} = [O_{;k},O_{;l}].
\end{equation}

Now we introduce the corresponding classes of pseudodifferential operators acting 
on Lie algebra valued functions. We begin with Lie algebra valued symbols 
$\Psi(x,\xi)$, where for $\g$ valued functions $B$ we use the Lie bracket to define 
using the left calculus
\begin{equation}\label{opad}
Op (ad(\Psi))(x,D) B(x) =  \int e^{i(x-y)\xi} [\Psi(x,\xi),B(y)]  dy d\xi.
\end{equation}
We note that its  $L^2$ adjoint (with respect to the Killing form duality) is 
$-  Op (ad(\Psi))(D,y)$, 

Similarly for a $\G$ valued symbol $O$ we define
 \begin{equation}\label{opAd}
Op(Ad (O))(x,D) B(x) =  
\int e^{i(x-y)\xi} O(x,\xi) B(y) 
 O^{-1}(x,\xi) dy d\xi.
\end{equation}
Its  $L^2$ adjoint (with respect to the Killing form duality) is 
$  Op (Ad(O^{-1}))(D,y)$.

\subsection{Structure of the paper}

Our paper is organized as follows:
 
In Section~\ref{s:elliptic}, we begin with some elliptic gauge related
fixed time estimates. In particular these will help us relate the full
nonlinear gauge independent energy with the linear energy associated
to the MKG-CG system. We also consider similar issues 
for the linearized equation.

In the following section we switch to space-time analysis, and define
the function spaces $S^1$ and $N$; with minor changes this follows
\cite{MKG}. We also recall some useful estimates from
\cite{MKG}, and add to that some additional properties from \cite{OT2},
related to the interval decomposition of the $S^1$ and $N$ spaces.

In Section~\ref{s:proof} we use the $S^1$ norms to provide a stronger
form of our main theorem, and we show that this follows from three
estimates in Propositions~\ref{p:pert}, ~\ref{p:para} and
~\ref{p:linearize}.

Section~\ref{s:pert} contains the perturbative part of our analysis,
which primarily consists of bilinear estimates in $S^1$ and $N$
spaces.  There we prove Proposition~\ref{p:pert}, as well as
Proposition~\ref{p:linearize} (the latter modulo Lemma~\ref{l:tri},
which captures the trilinear structure governed by the second null
condition, and whose proof is relegated to the next to last section).

The bulk of the paper is devoted to the construction of a  parametrix for the 
paradifferential equation \eqref{para}, which is the main step in the proof of the 
remaining Proposition~\ref{p:para}.

We begin in Section~\ref{s:para} with some heuristic considerations, followed by
the rigorous definition of the parametrix and by Theorem~\ref{t:para}, which summarizes
its properties. This suffices for the proof of  Proposition~\ref{p:para}.
In Section~\ref{s:decomp} we review the notion of decomposability, and establish
a number of bounds for the symbols $\Psi$ and $O$ arising in the definition of the 
parametrix.  The symbol bounds are then used in Section~\ref{s:l2} to derive kernel
bounds, and a number of $L^2$ estimates, concluding with the proof of the 
first three parametrix bounds in Theorem~\ref{t:para}, as well as  the Strichartz and null frame bounds 
for the renormalization operators in our parametrix.   Section~\ref{s:error} contains the proof of the 
error estimates in Theorem~\ref{t:para},  modulo Lemma~\ref{l:ttri}.
The two estimates that require a fine trilinear analysis, namely  Lemma~\ref{l:ttri} 
and Lemma~\ref{l:tri}, are proved in Section~\ref{s:tri}.

\section{ Elliptic $L^2$ bounds}
\label{s:elliptic}

Here for convenience we show that any small energy data admits a
Coulomb representation which is small in $\H$. We also show that the
equations \eqref{a0}-\eqref{dta0} are well-posed; this justifies the
fact that the initial data in the Coulomb gauge is fully determined by
$(A_j(0),\partial_t A_j(0))$ (at least at small energies).

\begin{proposition}\label{p:data-cg}
 a)  Let $(\tilde A_\alpha(0), \partial_t \tilde A_j(0)) \in \dot H^1
  \times L^2$ be an initial data for the Yang-Mills equation with
  energy $E$. If $E$ is small enough then there exists an unique gauge
  equivalent Coulomb data with 
\begin{equation}
\| (A_j(0),\partial_t A_j(0))\|_{\H}^2 \approx E 
\end{equation}

b) For any Coulomb data $ (A_j(0),\partial_t A_j(0))$ which is small in $\H$
there exists a unique solution $(A_0(0),\partial_t A_0(0)) \in \H$ to \eqref{a0}-\eqref{dta0}
so that 
\begin{equation}\label{en-a0}
 \| (A_0(0),\partial_t A_0(0))\|_{\H}^2 \lesssim E^2
\end{equation}

c) If in addition we have $(A_j(0),\partial_t A_j(0)) \in \H^N$ then we also have 
$(A_0(0),\partial_t A_0(0)) \in \H^N$ and 
\begin{equation}
 \| (A_0(0),\partial_t A_0(0))\|_{\H^N}^2 \lesssim E \| (A_j(0),\partial_t A_j(0))\|_{\H^N}^2
\end{equation}
\end{proposition}

\begin{proof} 

  The first part is proved (in $n\geq 6$ dimensions, but equally valid
  in lower ones) for example in \cite{Krieger:2005wh}. The second part
  is a consequence of Sobolev embeddings and a simple fixed point
  argument.

\end{proof}

We also consider the counterpart of part (b) for the linearized equation
\eqref{ym-cg-lin}. We have:

\begin{proposition}\label{p:data-cg-lin}
  Let $(A_j(0), \partial_t A_j(0)) \in \H$ be a
  Coulomb initial data for the Yang-Mills equation with small energy
  $E$.  Let $\frac12 < s < 1$ and $(B_j(0), \partial_t B_j(0)) \in
  \dot \H^s$ be a Coulomb initial data for the
  linearized Yang-Mills equation \eqref{ym-cg-lin}.  Then there exists
  a unique solution $(B_0(0),\partial_t B_0(0)) \in \dot \H^s$ to \eqref{b0}-\eqref{dtb0} so that
\begin{equation}
 \| (B_0(0),\partial_t B_0(0))\|_{ \dot \H^s}^2 
\lesssim  E  \| (B_j(0),\partial_t B_j(0))\|_{ \dot \H^s}^2
\end{equation}
\end{proposition}

\begin{proof}
This is also a simple fixed point argument which is based on the Sobolev embeddings. The 
details are left for the reader.

\end{proof}
\section{The $S$ and $N$ spaces}

With minor modifications, we will use the function spaces introduced in 
\cite{MKG} in the whole of $\R^{4+1}$. We also need to work
on bounded time intervals, for which we use the set-up of \cite{OT2}.

\subsection{The $S^1$, $N$, $Z$ and $Y^{1}$ spaces} 
We begin our discussion with the function spaces introduced in \cite{MKG},
namely $S^1$ for the MKG waves $(A,\phi)$ and $N$ for the inhomogeneous terms 
in both the $\Box$ and the $\Box_A$ equation. In addition to these we also recall
the $Z$ norm, which plays a key role in the reiteration of the equation in connection to
trilinear estimates and the second null structure.

 These are spaces of functions defined 
over all of $\R^{n+1}$, together with the related spaces $S$ and $N^*$.
They are all defined via their dyadic subspaces, with norms 
\[
\| \phi\|_{X}^2 = \sum_{k \in \Z} \|\phi_k\|_{X_k}^2, \qquad X \in \{ S,S^1,N,Z\}.
\]
Here we use the $\ell^2$ Besov structure. On occasion we will also need
$\ell^1$ and $\ell^\infty$ type Besov norms, which are denoted by $\ell^1 X$, 
respectively $\ell^\infty X$, with norms
\[
\| \phi\|_{\ell^1 X} = \sum_{k \in \Z} \|\phi_k\|_{X_k},\qquad 
\| \phi\|_{\ell^\infty X} = \sup_{k\in \Z} \|\phi_k\|_{X_k},
 \qquad X \in \{ S,S^1,N,Z\}. 
\]

We recall the definition of their norms. With minor modifications at high modulations, we follow
\cite{MKG}. For $N_k$ we set
\begin{equation}
  N_k \ = \ {L^1 L^2} +  X_1^{0,-\frac{1}{2}},
\label{n}
\end{equation}
where
\begin{equation*}
\|\phi\|_{X^{s, b}_{r}} := \big( \sum_{k} \big( \sum_{j} (2^{sk} 2^{bj} \|P_{k} Q_{j} \phi\|_{L^{2} L^{2}})^{r} \big)^{\frac{2}{r}}\big)^{\frac{1}{2}} .
\end{equation*}
The $N_{k}$ norm is the same as in \cite{MKG}.

The $S_k$ space is a strengthened version of $N_k^*$,
\begin{equation} \label{s-vs-n}
 X_1^{0,\frac{1}{2}}   \subseteq S_k\subseteq L^\infty L^2  \cap X_\infty^{0,\frac{1}{2}} = N_{k}^{\ast}, 
\end{equation}
while $S_k^1$ is defined as 
\begin{equation}\label{s1}
  \| \phi\|_{S_k^1} = \| \nabla \phi\|_{S_k} 
  + 2^{-\frac{k}2} \|\Box \phi\|_{L^2 L^{2}}
 + 2^{-\frac{4k}9} \|\Box \phi\|_{L^{\frac95} L^2} .
\end{equation}
As in \cite{OT2}, compared to \cite{MKG} we have loosened
the $\ell^{1}$ summability of the $\Box^{-1} L^{2}L^{2}$ norm and
added the $\Box^{-1} L^{\frac95} L^2$ norm above. Both of these
modifications are of interest only at high modulations.  The exact
exponent $9/5$ is not really important, for our purposes it only
matters that it is less than two and greater than $5/3$.

We now recall the definition of the space $S_{k}$ from
\cite{MKG}. The space $S_k$ scales like free waves with
$L^{2} \times \dot{H}^{-1}$ initial data, and is defined by
\begin{equation}
  \| \phi\|_{S_k}^2 \ = \ \| \phi\|_{S^{str}_k}^2 + \|\phi\|_{S^{ang}_k}^2
  +  \|\phi\|_{X_\infty^{0,\frac{1}{2}}}^2  \ , \notag
\end{equation}
where:
\begin{equation}\label{str_and_defn}
\begin{aligned}
  \|\phi\|_{S^{str}_{k}} \ =& \sup_{2 \leq q, r, \leq \infty, \
    \frac{1}{q}+ \frac{3/2}{r}\leq \frac{3}{4}}
  2^{(\frac{1}{q}+\frac{4}{r}-2)k} \|(\phi, 2^{-k} \partial_{t}
  \phi)\|_{L^q L^r} \ , \quad
  \|\phi\|_{S^{ang}_{k}} = \sup_{l < 0} \| \phi\|_{S^{ang}_{k, k+2l}} \ , \\
  \| \phi\|_{S^{ang}_{k, j}}^{2} =& \sum_{\omega}\| P^\omega_l
  Q_{<k+2l}\phi\|_{S_k^\omega(l)}^2 \qquad \hbox{ with } l = \lceil
  \frac{j-k}{2} \rceil.
\end{aligned}
\end{equation}
The $S^{str}_{k}$ norm controls all admissible Strichartz norms on
$\R^{1+4}$. The $\omega$-sum in the definition of $S^{ang}_{k, j}$ is
over a covering of $\S^{3}$ by caps $\omega$ of diameter $2^{l}$
with uniformly finite overlaps, and the symbols of $P^{\omega}_{l}$ form
a smooth partition of unity associated to this covering.  The angular
sector norm $S_{k}^{\omega}(l)$ combines the null frame space as in wave
maps \cite{Tao:2001gb, MR1827277} with additional square-summed norms
over smaller radially directed blocks $\mathcal{C}_{k'}(l')$ of dimensions
$2^{k'} \times (2^{k'+l'})^{3}$. We first define

\begin{align}
  \| \phi\|_{P\!W^\pm_\omega(l)} \ &=\ \inf_{\phi=\int \!\!
    \phi^{\omega'} } \int_{|\omega-\omega'|\leqslant 2^{l}}
  \| \phi^{\omega'} \|_{L^2_{\pm\omega' }(L^\infty_{(\pm\omega')^\perp}
    )} d\omega' \ , \notag\\
  \| \phi\|_{N\!E} \ &= \ \sup_\omega \|\snabla_\omega
    \phi\|_{L^\infty_{ \omega} (L^2_{\omega^\perp})} \ , \notag
\end{align}
where the norms are with respect to $\ell_\omega^\pm = t\pm
\omega\cdot x$ and the transverse variable in the
$(\ell^{\pm}_{\omega})^{\perp}$ hyperplane (i.e., constant
$\ell^{\pm}_{\omega}$ hyperplanes). Moreover, $\snabla_\omega$ denotes
tangential derivatives on the $(\ell^+_\omega)^\perp$ hyperplane. As in
\cite{MKG}, we set:
\begin{multline}
  \| \phi\|_{S_k^\omega(l)}^2 \ = \ \| \phi\|_{S_k^{str}}^2 +
  2^{-2k}\|\phi\|_{N\!E}^2 + 2^{-3k}\sum_\pm
  \| Q^\pm  \phi\|_{P\!W^\mp_\omega(l) }^2 \\
  + \sup_{\substack{k'\leqslant k ,   l'\leqslant 0\\
      k+2l\leqslant k'+l'\leqslant k+l }} \sum_{\mathcal{C}_{k'}(l') }
  \Big( \|P_{\mathcal{C}_{k'}(l')} \phi\|_{S_k^{str}}^2
  + 2^{-2k}\| P_{\mathcal{C}_{k'}(l')} \phi\|_{N\!E}^2\\
  + 2^{-2k'-k}\|P_{\mathcal{C}_{k'}(l')} \phi\|_{L^2(L^\infty)}^2 +
  2^{-3(k'+l')}\sum_\pm \| Q^\pm P_{\mathcal{C}_{k'}(l')}
    \phi\|_{P\!W^\mp_\omega(l) }^2 \Big) \ , \label{Sl_def}
\end{multline}
where the $\mathcal{C}_{k'}(l')$ sum runs over a covering of $\R^{4}$
by the blocks $\mathcal{C}_{k'}(l')$ with uniformly finite overlaps,
and the symbols of $P_{\mathcal{C}_{k'}(l')}$ form an associated
partition of unity.  We emphasize the role played by the next to last
term in the above expression, which captures the gain in Strichartz
estimates on blocks which are shorter radially. This gain was first
discovered in \cite{Klainerman:1999do}, and plays a key role in
getting some of the sharper bilinear bounds which are needed in the
present paper. We remark that there is a similar gain at the level of
the $L^2 L^6$ Strichartz norm, which could be easily added to the
$S^1$ structure; this would improve some of the intermediate estimates
in this paper, but would not affect the final result in a significant
way.

We also define the smaller space $S_k^\sharp
\subset S_k$ (see the bound \eqref{lin-S} below) by
\[
\| u \|_{S_k^\sharp} = \| \Box u\|_{N_k} + \|\nabla u\|_{L^\infty L^2}.
\]
On occasion we need to separate the two characteristic cones
$\{ \tau = \pm |\xi|\}$. Thus we define the spaces $N_{k, \pm}$, $S^{\sharp}_{k, \pm}$ and $N^{\ast}_{k, \pm}$ in an obvious fashion, so that
\begin{equation*}
N_k = N_{k,+} \cap  N_{k,-}, \quad
S_k^\sharp = S_{k,+}^\sharp + S_{k,-}^\sharp, \quad
N^*_k = N^*_{k,+} + N_{k,-}^* \ .
\end{equation*}

Next we describe an auxiliary space of the type $L^1(L^\infty)$ which
will be useful for decomposing the nonlinearity:
\begin{equation}
  \|\phi\|_{Z}^2 \ =\  \sum_k  \| P_k\phi \|^2_{Z_k} \ , \ \	
  \| \phi\|_{Z_k}^2 \ =\  \sup_{l<C} 
  \sum_{\omega }2^{l}\|P^\omega_lQ_{k+2l} \phi \|_{L^1(L^\infty)}^2
  \ . \notag
\end{equation}
Note that as defined this space already scales like $\dot{H}^1$ free
waves. In addition, note the following useful embedding which is a direct
consequence of Bernstein's inequality:
\begin{equation}
  \Box^{-1}  L^1(L^2) \ \subseteq \ Z \ . \label{B_embed}
\end{equation}
Finally, the function space $Y_1$  for $A_0$ is simple to describe,
since the $A_0$ equation is elliptic:
\begin{equation}
\|A_0\|_{Y^1}^2 = \| \nabla A_0\|_{L^\infty L^2}^2 + \| \nabla A_0\|_{L^2 
\dot H^{\frac12}}^2  \ , \notag
\end{equation}

One of the results in \cite{MKG} asserts that we have linear solvability for the 
d'Alembertian in our setting.  

\begin{proposition}
 We have  the linear estimates
 \begin{align}
 \| \nabla \phi\|_{S}
 &\lesssim \  \| \phi[0]\|_{\H} + \|\Box \phi\|_{N}  \ , 	\label{lin-S} \\
 \|\phi\|_{S^1} \ 
 &\lesssim \  \| \phi[0]\|_{\H} + \|\Box \phi\|_{N \cap L^{2} \dot{H}^{-\frac{1}{2}} \cap L^{\frac{9}{5}} \dot{H}^{-\frac{4}{9}}}  \ . 	\label{lin}
  \end{align}
\end{proposition}
Here \eqref{lin-S} is the embedding $S^{\sharp} \subset S$, whereas \eqref{lin} follows immediately from \eqref{lin-S}.

\subsection{Interval localization} \label{subsec:intervals} So far, we
have described the global setting in \cite{MKG}.  However,
in this article we need to work on compact time intervals, therefore we also
need suitable interval localized function spaces. For this we borrow the set-up
of \cite{OT2}. 

We start by defining
\begin{equation}\label{sn(i)-def}
  \| \phi\|_{S^1[I]} = \inf_{\phi = \tphi_{|I}}  \|\tphi\|_{S^1},  \qquad
  \| f \|_{N[I]} = \inf_{f = \tilde f_{|I}}  \|\tilde f\|_{N}
\end{equation}
The next result from \cite{OT2} provides an alternate take on these
definitions:

\begin{proposition} \label{p:intervals}
  Consider a time interval $I=[0,T]$, and its
    characteristic function $\chi_I$. Then we have the bounds
    \begin{equation}\label{sn(i)}
      \|\chi_I \phi\|_{S} \lesssim \|\phi\|_{S}, \qquad \|\chi_I f\|_{N} \lesssim \| f\|_{N},
    \end{equation}
    The latter norm is also continuous as a function of $I$. We also
    have the linear estimates
    \begin{align}
      \|\nabla \phi\|_{S[I]} \ &\lesssim \| \phi[0]\|_{\H} + \|\Box \phi\|_{N[I]} , \label{lin-S(i)} \\
      \|\phi\|_{S^1[I]} \ &\lesssim \| \phi[0]\|_{\H} + \|\Box
      \phi\|_{(N \cap L^{2} \dot{H}^{-\frac{1}{2}} \cap
        L^{\frac{9}{5}} \dot{H}^{-\frac{4}{9}}) [I]} . \label{lin(i)}
    \end{align}
  
\end{proposition}

Note that a consequence of the above proposition is that,
up to equivalent norms, we can replace the arbitrary extensions in
\eqref{sn(i)-def} by the zero extension in the $N$ case, respectively
by homogeneous waves with $(\phi, \partial_{t} \phi)$ as the data at each
endpoint outside $I$ in the $S^1$ case.

\section{The proof of the main result}
\label{s:proof}
In this section we provide the main intermediate results used in the
proof, and we use them in order to complete the proof of the
Theorem~\ref{t:main}. For convenience, we restate the theorem here
in a more precise form:

\begin{theorem}\label{t:main-sn}
The Yang-Mills system in Coulomb gauge \eqref{ym-cg}-\eqref{a0}-\eqref{dta0}
is globally well-posed in $\dot H^1 \times L^2$ for  initial data which is small 
in $\H = \dot H^1 \times L^2$,
\begin{equation}\label{small-data}
\| A_x(0),\partial_t A_x(0)\|_{\H} \leq \epsilon,
\end{equation}
in the following sense:

(i) (Regular data) If in addition the data $(A_{0j}, A_{1j})$ is more regular, 
$(A_{0j}, A_{1j}) \in \H^N$, then there exists a unique global in time regular solution
$(A_j,\partial_0 A_j) \in S^N$, which has a Lipschitz dependence 
on the initial data locally in time in the $\H^N$ topology.

(ii) (Rough data) The initial data to solution map admits an extension
\[
\H \ni (A_{j0}, A_{j1}) \to (A_j, \partial_t A_j) \in S^1,
\]
globally in time, for all small data as above, and which is 
continuous in the $\H \cap \dot H^s \to S^1 \cap \dot S^s$ topology 
(applied to differences of solutions) for $s < 1$ but close to $1$.
\end{theorem}

To set the stage for the proof of the theorem, we assume that we have a solution 
$A_j$ for the Yang mills equation \eqref{ym-cg} in a time interval $I$ containing $0$,
and further that this solution satisfies 
\begin{equation}\label{small-s}
\|A_j\|_{S^1[I]} \leq  \epsilon \ll 1.
\end{equation}

We begin by rewriting the equation in a paradifferential fashion,
\begin{equation}\label{para}
\Box A_{j,k}  + 2 \P [A_{\alpha,<k}, \partial^\alpha A_{j,k}] = F_k,
\end{equation}
where $F_k$ contains only terms that will be treated in a perturbative fashion,
\begin{equation}
F_k = \P \left( P_k \left(   [A^\aa,  \p_j A_\aa] - 2 [ A^\alpha_{\geq k}, \partial_\alpha A_j] 
- [ \partial_0 A_0, A_j]  - [A^\alpha,[A_\alpha,A_j]] \right)  - 2
 [ [P_k, A^\alpha_{< k}], \partial_\alpha A_j] \right).
\end{equation}
To estimate $F$ we use the following:

\begin{proposition} \label{p:pert} Assume that $A$ is a solution to
  the Yang-Mills equation in Coulomb gauge in an interval $I$, which
  satisfies \eqref{small-s}.  Then  for any admissible frequency envelope
$c$ we have 
\begin{equation}\label{FinN}
\|F\|_{N_{c^2}[I]} \lesssim  \|A_j\|_{S^1_c[I]}^2.
\end{equation} 
\end{proposition}
This proposition is proved in the next section. We remark that 
by appropriately choosing the envelope $c$, this implies that
\begin{equation}\label{FinN1}
\|F\|_{N_c[I]} \lesssim  \epsilon \|A_j\|_{S^1_c[I]},
\end{equation} 
as well as 
\begin{equation}\label{FinN2}
\|F\|_{\ell^1 N[I]} \lesssim  \epsilon \|A_j\|_{S^1[I]},
\end{equation} 

We now turn our attention to the linear equation \eqref{para}.  In order to
uncouple variables it will be useful to also consider the more general
frequency localized equation:
\begin{equation}\label{parab}
\Box B_{j,k}  + 2\P [A_{\alpha,<k}, \partial^\alpha B_{j,k}] = G_{j,k}.
\end{equation}

\begin{proposition} \label{p:para} Assume that $A$ is a solution to
  the Yang-Mills equation in Coulomb gauge in an interval $I$, which
  satisfies \eqref{small-s}.  Then for the equation
 \eqref{parab} we  have the following linear estimate:
\begin{equation}
\|B_{j,k}\|_{S^1[I]} \lesssim ( \|G_{j,k}\|_{N[I]} + \| B_{j,k}[0]\|_{\H}).
\end{equation} 
\end{proposition}
This result is the key point of the paper. Its proof is closed in Section~\ref{s:para},
using the paradifferential parametrix in Theorem~\ref{t:para}. However, the proof 
of Theorem~\ref{t:para} requires all the subsequent sections of the paper.

The two bounds above suffice in order to close the a-priori bounds in $S^1$ and $S^N$,
including frequency envelope bounds. In order to compare different solutions, 
we need to work with the linearized equation  \eqref{ym-cg-lin}-\eqref{b0}-\eqref{dtb0}.

\begin{proposition}\label{p:linearize}
  Suppose that $A$ is a solution to the Yang-Mills equation in Coulomb
  gauge in an interval $I$, which satisfies \eqref{small-s}.  Then the
  equation \eqref{ym-cg-lin} is well-posed in $\H^s$ for $s < 1$,
  close to $1$, in the time interval $I$.
\end{proposition}

To further clarify this last result,  we rewrite the equation \eqref{ym-cg-lin}
 in a paradifferential form,
\begin{equation}\label{para-lin}
\Box B_k + \P [A_{\alpha, <k} \partial^\alpha B_k] = \P [B_{\alpha,<k}, \partial^\alpha A_k] + G_k .
\end{equation}
The term $G_k$ plays the same role as $F_k$ in the original equation. Precisely,
we have:

\begin{proposition} \label{p:pert-lin} Assume that $A$ is a solution to
  the Yang-Mills equation in Coulomb gauge in an interval $I$, which
  satisfies \eqref{small-s}.  Then for $s \leq 1$, close to $1$ we have
\begin{equation}
\|G_j\|_{N^{s-1}} \lesssim \epsilon \|B_j\|_{S^s}.
\end{equation} 
\end{proposition}
This result is proved in the next section. We remark that the range of $s$ 
depends on the constant $\delta$ in the estimate \eqref{null-main}
in the next section, which is from \cite{MKG}. We expect the correct range here to be $s > \frac12$.

The new term $[B_{\alpha,<k}, \partial^\alpha A_k] $ in
\eqref{para-lin} does not have a counterpart in the previous
argument. This is the term which is responsible for disallowing case
$s = 1$ in Proposition~\ref{p:para}, and ultimately for the failure of
the Lipschitz dependence of the solution on the initial data in the
strong topology $\H$. We will estimate this in a more roundabout
fashion, proving the following statement:

\begin{proposition}\label{p:tri}
  Suppose that $B \in S^s$ solves the linearized equation
  \eqref{ym-cg-lin} in a time interval $I$, around a YM-CG solution
  $A$ which satisfies \eqref{small-s}.  Then for $s < 1$, close to $1$
  we have the estimate
\begin{equation}
\|[B_{\alpha,<k}, \partial^\alpha C_k]\|_{N^{s-1}} \lesssim \epsilon \|B\|_{S^s}    \|C_k\|_{S^1}.
\end{equation}
\end{proposition}
This Proposition is more delicate than the previous proposition, as it 
requires a fine trilinear analysis based on reiterating the linearized equation. Its proof
is also in the next section, modulo  the most difficult
case in Lemma~\ref{l:tri}, which is relegated to Section~\ref{s:tri}. 

The result in Proposition~\ref{p:linearize} is a direct consequence of
Proposition~\ref{p:para}, Proposition~\ref{p:pert-lin} and
Proposition~\ref{p:tri}. We now turn our attention to
Theorem~\ref{t:main-sn}.

\begin{proof}[Proof of Theorem~\ref{t:main-sn}]
  Here we show that Theorem~\ref{t:main-sn} follows from
  Propositions~\ref{p:pert}\ref{p:para},\ref{p:linearize}. In addition
  to these Propositions, we will also take it for granted that for large
  $N$ (e.g. $N \geq 3$) the Yang-Mills equation is locally well-posed
  in $\H^N$, with smooth dependence on the initial data; at least at
  small energies this is a straightforward perturbative result, based
  purely on energy estimates. We carry this out in several steps.
\medskip 

\step{1}{A-priori bounds for regular data} 
 Here we consider regular $\H^N$ solutions in a time interval $I = [0,T]$, and which 
satisfy the smallness condition  
\begin{equation}\label{apriori-smallS}
\| A_x\|_{S^1[I]} \leq \epsilon_0 \ll 1.
\end{equation}
Let $c$ be an admissible $\H$ frequency envelope for the initial data. Then we claim 
that $c$ is also an $S^1$ frequency envelope for the solution, 
and, in addition, we have the bound
\begin{equation}\label{Sc} 
\| A_x\|_{S^1_c} \lesssim \| A_x[0]\|_{\H_c}.
\end{equation}
We remark that, as a consequence of this, we have in particular the bounds
\begin{equation}\label{global-bounds}
\|  A_x\|_{S^1} \lesssim \| A_x[0]\|_{\H}, \qquad \|  A_x\|_{S^N} \lesssim \| A_x[0]\|_{\H^N}.
\end{equation}

Assume first that we already know that $A_x \in S_c$. Then \eqref{Sc} is obtained by successively 
applying Propositions~\ref{p:pert}, \ref{p:para} in the equation \eqref{para}. Without knowing that 
$A_x \in S_c$, let $d$ be an admissible frequency envelope for $A_x$ in $S^1$. Then for 
$\delta > 0$ we have $A_x \in S^1_{c+\delta d}$. Then we have \eqref{Sc} with $c$ replaced by $c + \delta d$,
and it suffices to let $\delta$ to zero to obtain again \eqref{Sc}.

\medskip

\step{2}{Global solutions for regular data}  Here we start with regular data
$(A_j(0)\, \partial_t A_j(0)) \in \H^N$ which is small in the energy
norm, i.e. it satisfies \eqref{small-data}.  Then the solution exists in $\H^N$ on some nonempty 
time interval $[0,T)$. We claim that the solution is global, $T= \infty$, 
and that it satisfies the bound
\begin{equation}\label{small-sln}
\| A_j\|_{S^1} \leq C \epsilon,
\end{equation}
with a fixed universal constant $C$.

This is done using a  time continuity argument. Let $\mathcal T$
denote the set of all times $T$ for which a classical (i.e. $\H^N$ solution )
exists in $[0,T]$ which satisfies \eqref{small-sln}. We will prove that $\mathcal T$
is both open and closed, and thus must be equal to $\R^+$.

\medskip 

a) $\mathcal T$ is closed.  Indeed, suppose that $[0,T_0) \subset
\mathcal T$. By \eqref{global-bounds} we have a uniform bound
\[
\| A_j\|_{S^N[0,T]} \lesssim \| A_j[0]\|_{\H^N}.
\]
Then, in view of  the Lipschitz dependence for classical solutions,
the solution $A_j$ extends to time $T_0$ (and indeed, past it) as a
classical solution. By a scaling argument, see e.g. \cite{wm}, the
$S^1[I]$ norm of classical solutions depends continuously on the
interval $I$.  Thus the bound \eqref{small-sln} at time $T_0$ follows,
so $T_0 \in \mathcal T$.

\medskip

b) $\mathcal T$ is open. Let $T \in \mathcal T$. Then $A_j[T] \in
H^N$, so we can continue the solution beyond time $T$.  It remains to
show that the bound \eqref{small-sln} persists.  Using again the
continuous dependence of the $S^1[I]$ norm of classical solutions on
the interval $I$, it suffices to prove \eqref{small-sln} this under
a bootstrap assumption
\begin{equation}\label{sln-boot}
\| A_j\|_{S^1} \leq 2 C \epsilon,
\end{equation}
with a large universal constant $C$. But this again follows from \eqref{global-bounds}
in Step 1.

\bigskip
 
\step{3}{Weak Lipschitz dependence for regular solutions} 
Here we assert that for any two small data global regular solutions we have the  
bound
\begin{equation} \label{lip}
\| A_j -\tilde A_j\|_{S^s} \lesssim \|A_j[0]-\tilde A_j[0]\|_{\H^{s}}.
\end{equation}
provided $s< 1$ is close to $1$. This is a direct consequence of the 
result in Proposition~\ref{p:linearize}.

\bigskip \step{4}{Rough data solutions} The continuous extension of
the flow map to rough data for solutions which satisfy
\eqref{small-sln}, using the $\H \cap \dot H^s$ topology, follows in a
standard manner from two properties of small data solutions: 

\begin{itemize}
\item  The frequency envelope bounds \eqref{Sc}.
\item The Lipschitz dependence in a weaker topology \eqref{lip}.
\end{itemize}

Indeed, consider some small energy data $A_x[0] \in \H$. Then for any
$A_x^{(n)}$ are regular solutions, whose data $A_x^{(n)}[0]$ converge
to $A_x[0] \in \H$ in the sense that
\[
\| A_x^{(n)}[0] - A_x[0]\|_{\H \cap \dot H^s} \to 0.
\]
By \eqref{lip} the limit  $A_x$  of $A_x^{(n)}$ exists in $\dot S^s$.  Further, the relation \eqref{lip}
extends to all solutions constructed in this way.

Favorably choosing $A_x^{(n)}[0]$ so that they have the same $\H$  frequency envelope as 
$A_x[0]$ (e,g. as $A_x^{(n)}[0]=P_{<n}A_x[0]$ )  and applying \eqref{Sc}, it follows that $A_x \in S^1$, and further that \eqref{Sc}
holds for $A_x$.

Finally, to establish the continuity of the data to solution map from
$\H \cap \dot \H^s$ to $S \cap \dot S^s$ we use the previously
established $\dot H^s$ Lipschitz bound for low frequencies, combined
with the uniform smallness of high frequency tails, which is in turn
derived from the frequency envelope bound.

\end{proof}

\section{ Bilinear estimates and perturbative analysis }
\label{s:pert}

The first goal of this section is to review the bilinear null form bounds from \cite{MKG},
which will be repeatedly used in our analysis. Then we use these bounds to 
provide some preliminary characterization of YM solutions which satisfy an a-priori
$S^1$ bound. Finally, we  conclude with a proof of Propositions~\ref{p:pert} and 
~\ref{p:linearize}.

\subsection{Bilinear null form bounds}

We begin with the main bilinear null form estimate, where $\mathcal
N(u, v)$ refers to any expression of the form $\partial_i u\partial_j
v - \partial_j u\partial_i v$. It comes from \cite{MKG}, and
specifically from (131) in Theorem 12.1 there:
\begin{proposition}\label{prop:nullform1}(\cite{MKG})
For any null form $\mathcal N$ we have the following null form estimates:
\begin{equation}\label{null-main}
\| P_k \mathcal N(u_{k_1},v_{k_2})\|_{N} \lesssim 2^k 2^{\delta (k_{min}- k_{max})}\|u_{k_1}\|_{S} \|v_{k_2}\|_{S}
\end{equation} 
\end{proposition}
We remark that, in view of Proposition~\ref{p:intervals}, the same bound holds in any time interval $I$.

Ideally, we would like to improve this bound in the case of
low-high frequency interactions $k_1 < k_2 = k$, and have a $2^{k_1}$
factor instead. Unfortunately that 
does not work in general. However, it does work for the most part.
To describe that we isolate the bad component, namely  $\H^* \mathcal N(u_{k_1},v_{k_2})$.
Here, following \cite{MKG}, if $\mathcal{M}(D_{t,x},D_{t,y})$ is any 
bilinear translation invariant operator then we set:
\begin{align}
  \mathcal{H}^* \mathcal{M}(\phi_{k_1},\psi_{k_2}) \ = \ \sum_{j<k_1} Q_{<j-C}
  \mathcal{M}(Q_{j} \phi_{k_1}, Q_{<j-C}\psi_{k_2} ) \ , \qquad k_1 < k_2-C
\end{align}
We observe that the map $\mathcal H^{*}$ selects the portion of the
bilinear interaction where both the high frequency input and the
output have high modulation.  This case is unfavorable in the high
frequency limit; this is most easily seen using duality to rewrite the
above bound in a trilinear fashion.  We also remark that the
frequency/modulation localization in $\mathcal H^{*}$ fixes the angle
$\theta$ between the two input functions to
\[
\theta \approx 2^{(j-k_1)/2}
\]

A benefit of the null form structure of the nonlinearity is that it provides
an additional gain at small angles in bilinear estimates, which is roughly proportional 
to the angle. We will also need to take advantage of this gain  in our estimates.
For this we introduce a second selection device for bilinear interactions.
Precisely, given two spatial frequencies $\xi$ and $\eta$ we define a partition of unity
\[
1 = \sum_{\theta \ \text{dyadic}} \chi_\theta(\xi,\eta)
\]
where $\chi_\theta(\xi,\eta)$ is a smooth homogeneous cutoff which
selects the region where $\angle(\xi,\eta) \approx \theta$. Then, given
bilinear translation invariant operator $\mathcal{M}(D_{t,x},D_{t,y})$
with symbol ${m}(\tau,\xi,\sigma,\eta)$, we define
$\mathcal{M}^\theta$ as the bilinear translation invariant operator
with symbol ${m}(\tau,\xi,\sigma,\eta) \chi_\theta(\xi,\eta)$. We will
similarly used the notations $\mathcal{M}^{<\theta}$,
$\mathcal{M}^{>\theta}$ with the obvious meanings.

We now return to the promised decomposition of the null form into a
good and a bad part.  For the complement $(I-\H^*) \mathcal N(u_{k_1},v_{k_2})$
we have a good $S$ bound; for $\H^* \mathcal N(u_{k_1},v_{k_2})$, instead, we
use the $Z$ norm as a proxy. The following estimates are contained in
Theorem 12.1, Theorem 12.2 in \cite{MKG}:

\begin{proposition}(\cite{MKG})\label{p:ssn}
For $k_1 < k_2-C$  and any null form $\mathcal N$ we have the following bilinear estimates:

a) $S^1 \times S^1 \to N$ bound:
\begin{equation}\label{ssn}
\| (I-  \H^*) \mathcal N(u_{k_1},v_{k_2})\|_{N} \lesssim 2^{k_1} \|u_{k_1}\|_{S^1} \|u_{k_2}\|_{S^1}.
\end{equation} 
We also have the small angle improvement
\begin{equation}\label{ssn-theta}
\| (I-  \H^*)^{< \theta} \mathcal N(u_{k_1},v_{k_2})\|_{N} \lesssim 2^{k_1} \theta^\frac14 
\|u_{k_1}\|_{S^1} \|u_{k_2}\|_{S^1}.
\end{equation} 

b) $Z \times S^1 \to N$ bound:
\begin{equation}\label{zsn}
\| \H^* \mathcal N(u_{k_1},v_{k_2})\|_{N} \lesssim  2^{k_1} \|u_{k_1}\|_{Z} \|v_{k_2}\|_{S^1}, \qquad k_1 < k_2
\end{equation} 

c) $L^2 \dot H^\frac32 \times S \to N$ bound:
\begin{equation}\label{hsn}
\| (I-  \H^*) (u_{k_1} \cdot \nabla v_{k_2})\|_{N} \lesssim 
 \|u_{k_1}\|_{L^2 \dot H^\frac32} \|u_{k_2}\|_{S^1}.
\end{equation} 

d) $\Box^\frac12 \Delta^{-\frac12} Z \times S \to N$ bound:
\begin{equation}\label{zsn-ell}
\| \H^* (u_{k_1} \cdot \nabla v_{k_2})\|_{N} \lesssim 
 \|u_{k_1}\|_{\Box^\frac12 \Delta^{-\frac12}Z} \|u_{k_2}\|_{S}.
\end{equation} 

\end{proposition}
In order to be able to take advantage of the bilinear bounds which use
the $Z$ norm we need to have an additional estimate allowing us to
bound $Z$ norms appropriately.

 To describe the result we need a second
operator $\H_k$, which, following \cite{MKG} is defined as
\begin{align}
  \mathcal{H}_k \mathcal{M}(\phi_{k_1},\psi_{k_2}) \ &= \ \sum_{j<k+C}
  Q_{j} P_k \mathcal{M}(Q_{<j-C}\phi_{k_1}, Q_{<j-C}\psi_{k_2} ) , \qquad k < k_1 = k_2
  \end{align}
Then the $Z$ bounds are as follows, also contained in \cite{MKG}:

\begin{proposition}
For any null form $\mathcal N$ have the following $Z$ bounds:

a) Bound for classical solutions:
\begin{equation}\label{l12-to-z}
 \| \phi_k\|_{Z} \lesssim \|\Box \phi_k\|_{L^1L^2}
\end{equation}

b) High-low interactions:
\begin{equation}\label{hl-to-z}
\| P_k  \mathcal N(u_{k_1},v_{k_2})\|_{\Box Z} \lesssim 2^{k}
2^{-\delta |k_1-k_2|} \|u_{k_1}\|_{S} \|u_{k_2}\|_{S}, 
\qquad k >  k_{max}-C
\end{equation} 
\begin{equation}\label{hl-to-z0}
\| P_k (u_{k_1} \cdot \nabla v_{k_2})\|_{\Delta^\frac12 \Box^\frac12 Z} \lesssim 2^{k_1+k_2}
2^{-\delta |k_1-k_2|} \|u_{k_1}\|_{S} \|u_{k_2}\|_{S}, 
\qquad k >  k_{max}-C
\end{equation} 

c)  High-high-low interactions:
\begin{equation}\label{hh-to-z}
\| (I-\H_k) \mathcal N(u_{k_1},v_{k_2})\|_{\Box Z} \lesssim 2^{k_1} 2^{-\delta|k-k_1|} 
\|u_{k_1}\|_{S} \|u_{k_2}\|_{S}, \qquad k < k_1 = k_2
\end{equation} 
\begin{equation}\label{hh-to-z0}
\| (I-\H_k) (u_{k_1} \cdot \nabla v_{k_2})\|_{\Delta^\frac12 \Box^\frac12 Z} \lesssim 2^{k_1+k_2} 2^{-\delta|k-k_1|} 
\|u_{k_1}\|_{S} \|u_{k_2}\|_{S}, \qquad k < k_1 = k_2
\end{equation}

\end{proposition}
 
To better understand how the last two propositions fit together, we
remark that the in the bounds in Proposition~\ref{p:ssn} there is no
off-diagonal decay with respect to the frequency gap $k_1-k_2$. Hence, 
we can only apply it for portions of  $A_x$ which we control in 
$\ell^1 Z$. This is why the off-diagonal decay in \eqref{hl-to-z},  \eqref{hl-to-z0} and 
\eqref{hh-to-z}, \eqref{hh-to-z0} is important.

We further remark that the same estimates in \cite{MKG}
also yield a bound for the remaining bad component of $\mathcal N(u_{k_1}, v_{k_2})$,
namely 
\begin{equation}\label{hh-to-z-bad}
\| \H_k \mathcal N(u_{k_1},v_{k_2})\|_{\Box Z} \lesssim 2^{k_1}  \|u_{k_1}\|_{S} \|u_{k_2}\|_{S}, \qquad k < k_1 = k_2
\end{equation} 
and similarly
\begin{equation}\label{hh-to-z-bad0}
\| \H_k (u_{k_1} \cdot \nabla v_{k_2})\|_{\Delta^\frac12 \Box^\frac12 Z} \lesssim 2^{k_1+k_2}  
\|u_{k_1}\|_{S} \|u_{k_2}\|_{S}, \qquad k < k_1 = k_2
\end{equation}

Unfortunately, these bounds have no off-diagonal decay, so they only lead
to an $\ell^\infty Z$ bound for the corresponding ``bad''part of
$A$. If one attempts to combine this with Proposition~\ref{p:ssn}, we
are left with an unresolved logarithmic divergence. Addressing this issue 
requires the finer trilinear analysis in the last section of the paper, and the use of the second null form.

\subsection{ Characterization of $S^1$ solutions for YM-CG}

While the $S^1$ envelope of a Yang-Mills wave $A$ naturally inherits
the $\ell^2$ dyadic structure from the initial data, one might expect
that the inhomogeneous part of $A$, arising from bilinear or cubic
interactions, might carry a better, $\ell^1$ dyadic summation. This
was indeed the case for the Maxwell-Klein Gordon system in \cite{MKG},
and it allowed us to treat the inhomogeneous part of $A$ in a
perturbative fashion, as well as to use free wave magnetic potentials in the 
parametrix construction. Unfortunately, it is no longer the case here, as
the bilinear self-interactions of $A$ are not perturbative. However,
we are still able to prove $\ell^1$ dyadic summation fully for $A_0$,
and in a partial manner only for the inhomogeneous part of $A_x$. This
will allow us to treat not all but the bulk of the nonlinearity in a
perturbative fashion.  Precisely, we prove the following:

\begin{proposition}\label{prop:refinedbox}
Let $A$ be a solution for the YM-CG in an interval $I$ so that $\|A \|_{S^1} \leq \epsilon$.
Then the following property holds:
\begin{equation}\label{l1-a0}
\| \nabla A_0\|_{\ell^1 L^2 \dot H^\frac12} \lesssim \epsilon.
\end{equation}
Also, for each $0 \leq  b <  \frac{1}{2}$ we have 
\begin{equation}\label{l1-ax}
\|\Box  A_x\|_{\ell^1 X^{b - \frac12, -b}} \lesssim_{\delta} \epsilon
\end{equation}
\end{proposition}

A related result holds for the linearized equation. There, the dyadic summation is 
not an issue because the bounds for the linearized problem are no longer at scaling
(though they are scale invariant).  Also, the bounds we need for the linearized equation are 
not as refined as those we need for the original equation. We have:

\begin{proposition}\label{p:box-lin}
Let $A$ be a solution for the YM-CG in an interval $I$ so that $\|A \|_{S^1} \leq \epsilon$,
and $B \in S^s$ a solution to the linearized equation, with $\frac12 < s \leq 1$.  
Then the following properties hold:
\begin{equation}\label{b0-hs}
\| \nabla B_0\|_{ L^2 \dot H^{s-\frac12}} \lesssim \|B\|_{S^s}
\end{equation}
\begin{equation}\label{box-bx}
\|\Box  A_x\|_{L^2 \dot H^{s-\frac32}} \lesssim \epsilon  \|B\|_{S^s}
\end{equation}
\end{proposition}

Next we prove Proposition~\ref{prop:refinedbox} with $b = 0$, as well as
Proposition~\ref{prop:refinedbox} \ref{p:box-lin}.  The proof of the case $b > 0 $ of 
Proposition~\ref{prop:refinedbox} is postponed for later in this  section. We remark that
while the case $b=0$ is frequently used, the  stronger bound for $b > 0$
is used just once, later in the paper, in estimating the error term $E_{1,out}$
in Section~\ref{s:error}.

\begin{proof}[Proof of Proposition~\ref{prop:refinedbox} for $b=0$] 
a) We begin with the $A_0$ bound, where we first estimate the right hand side in the equation \eqref{a0}.
Using Sobolev embeddings we have the dyadic estimate with off-diagonal decay
\begin{equation}\label{da0-lp}
\begin{split}
\| P_k [A_{j,k_1},  \partial_0 A_{j,k_2}]\|_{L^2 \dot H^{-\frac12}} \lesssim & \ 2^{-\frac16 (k_{max} - k_{min})} 
\| |D_x|^\frac16 A_{j,k_1}\|_{L^2 L^6} \|  \partial_0 A_{j,k_2} \|_{L^\infty L^2} 
\\ \lesssim & \ 2^{-\frac16 (k_{max} - k_{min})} \|A_{j,k_1}\|_{S^1} 
\| A_{j,k_2} \|_{S^1}
\end{split}
\end{equation}
After dyadic summation this gives
\[
\|   [A_{j,k_1} \partial_0 A_{j,k_2}]\|_{\ell^1 L^2 \dot H^{-\frac12}} \lesssim  \|A_{j}\|_{S^1} 
\| A_{j} \|_{S^1} \lesssim \epsilon^2
\]
Now we solve the equation \eqref{a0} perturbatively in $\ell^1 L^2 \dot H^{\frac32}$, estimating the 
terms $[A_j,[A_j, A_0]]$ and $[A_j,\partial_j A_0]$  in the same manner as above, appropriately 
using Sobolev embeddings to gain off-diagonal decay in frequency.

We need to separately prove the $\partial_t A_0$ bound, for which we use the equation~\eqref{dta0}.
Then it suffices to prove estimates of the form 
\[
\begin{split}
\|[\partial_0 A_0, A_j] \|_{\ell^1 L^2 \dot H^{-\frac12}} \lesssim \| \partial_0 A_0\|_{L^2 \dot H^\frac12} 
\| A_j\|_{\ell^2 L^\infty \dot H^1} 
\\
\|[A_0, \partial_0 A_j] \|_{\ell^1 L^2 \dot H^{-\frac12}} \lesssim \| A_0\|_{L^2 \dot H^\frac32} 
\| \partial_0 A_j\|_{\ell^2 L^\infty L^2} 
\end{split}
\]
These are also easily proved via dyadic estimates with off-diagonal
decay, which in turn are obtained using Sobolev embeddings.

b) We separately consider each of the terms on the right in the
equation \eqref{ym-cg} for $A_x$.  exactly as in case (a), using the
bound \eqref{l1-a0} for the terms containing $A_0$. Then the $b = 0$
case of \eqref{l1-ax} follows exactly as in case (a). We remark that
the null condition is not used at all here.

\end{proof}

\begin{proof}[Proof of Proposition~\ref{p:box-lin}]
  a) This is similar to the proof of the previous proposition. One
  only needs to combine the bound \eqref{l1-a0} and the energy bound \eqref{en-a0} for
  $A_0$ with Strichartz estimates for $A_x$ and $B_x$ and Sobolev
  embeddings in order to solve the equations \eqref{b0} and
  \eqref{dtb0} perturbatively in $L^2 \dot H^{s+\frac12}$,
  respectively $L^2 \dot H^{s-\frac12}$.

b) This is  similar to the corresponding  bound in the $b=0$ case of the previous
proposition. The terms on the right in \eqref{ym-cg-lin} are
similar to those in \eqref{b0}, so exactly the same estimates apply.
\end{proof}

\subsection{ The perturbative bounds in 
Proposition~\ref{p:pert}, ~\ref{p:pert-lin}
}
We primarily discuss Proposition~\ref{p:pert} here, as the numerology
is simpler. As the terms in $G_k$ are similar to those in $F_k$, the
proof of Proposition~\ref{p:pert-lin} is completely similar. However,
we remark that, since we work with $F_k$ and $G_k$ term by term, one
can view Proposition~\ref{p:pert} as a special case of
Proposition~\ref{p:pert-lin}, for $s = 1$.

\begin{proof}[Proof of  Proposition~\ref{p:pert}]
We will successively consider all terms in $F$, taking into account 
the following observations:
\smallskip

(a)  All estimates below are
consequences of the corresponding dyadic estimates.  Hence, in order
to gain the control of the frequency envelope for the output $F$ it
suffices to obtain an off-diagonal gain in each of the expressions we
consider.
\smallskip

(b) The estimates in the proposition are restricted to a time interval $I$. 
However, this does not cause any difficulties since both the Strichartz 
bounds and the estimate \eqref{null-main} are equally valid in $I$.
Further, we recall that by Proposition~\ref{p:intervals} we can readily 
restrict $S$ and $N$ functions to time intervals.
\smallskip

(c) Due to the Leray projector and the identity 
\[
F_j = \sum_k \triangle^{-1}\partial_k\big(\partial_k F_j - \partial_j F_k\big)
\]
valid for divergence free vector fields $F$,  it suffices to estimate the curl of $F_k$. 
This observation will be used for the first  term below, but not for the rest.

\bigskip
{ \bf 1. The term $ [A_i,\p_j A_i]$. } Its curl is a null form $\mathcal N(A_i,A_i)$, therefore it remains
to produce an $N$ bound for the expression $|D_x|^{-1} \mathcal N(A_i,A_i)$.
But this is a direct consequence of Proposition~\ref{prop:nullform1}, with a suitable off-diagonal gain.
\medskip

{\bf 2. The term $[A_j, \p_j A_i]$}, high-high and  and high-low interactions. 
Here we use the Coulomb condition $\partial_j A_j = 0$ to write 
\[
[A_j, \p_j A_i] = [\p_k(\triangle^{-1}\p_k A_j), \p_j A_i] =  [\p_k(\triangle^{-1}\p_k A_j), \p_j A_i] -  [\p_j(\triangle^{-1}\p_k A_j), \p_k A_i]
\]
which is of the form $\mathcal N(|D_x|^{-1}A, A)$ where the high
frequency term is hit by $|D_x|^{-1}$.  Then the desired bound is
again a consequence of Proposition~\ref{prop:nullform1}, with off-diagonal gain.

\medskip

{\bf 3. The term $ [\p_0 A_0, A_i]$.}
This is a Strichartz term. Precisely, we can use  $\p_0 A_0 \in L^2 \dot H^{\frac12}$ as in \eqref{l1-a0}
together with the $L^2 L^6$ Strichartz bound for $A_i$ and Sobolev embeddings to place it in $L^1L^2$,
with off-diagonal gain.

\medskip 
{\bf 4. The term $[A_0, \p_t A_i]$}, high-high and  and high-low interactions. 
Here one uses $A_0\in L^2\dot{H}^{-\frac{1}{2}}$ and $\nabla^{-\frac{3}{2}}\partial_t A_i\in L^2 L^\infty$ 
to place the output into $L^1 L^2$. 

\medskip

{\bf 5. The commutator term $[P_k, ad(A_{<k}^\alpha)] \partial_{\alpha}A_j$.} 
This is equivalent to an expression of the form 
\[
2^{-k} [|D_x| A^\alpha_{<k}, \partial_\alpha A_k]
\]
For $\alpha \neq 0$ this gives, as in Case 2, a null form of the type $2^{-k} \mathcal N(A_{<k},A_k)$ 
which is handled via Proposition~\ref{prop:nullform1}. For $\alpha = 0$ it is equivalent to 
\[
2^{-k} [\nabla A_{0, <k}, \partial_t A_k]
\]
which is a Strichartz term as in Case 3. Both cases have some off-diagonal gain.

\medskip

{ \bf 6. The cubic term $ [A_j,[A_j,A_i]]$. }
This is placed in $L^1 L^2$ via Strichartz estimates and Sobolev embeddings. The off-diagonal gain
is a consequence of the fact that there is a range of Strichartz estimates which can be used in order 
to obtain the $L^1L^2$ bound.
\end{proof}

\subsection{The proof of Proposition~\ref{prop:refinedbox} 
for $b > 0$} 

We consider the paradifferential decomposition of the nonlinearity in the wave equation for $A_j$
as in \eqref{para}. For the $F$ component we already have the bound in Proposition~\ref{p:pert},
more precisely \eqref{FinN2}, which suffices for all $0 \leq b < \frac12$. Hence it remains to bound the 
expression 
\begin{equation}
\| \sum_{k_1 < k-C} [A_{k_1,\alpha}, \partial^\alpha A_{k_2}]\|_{\ell^1 X^{b-\frac12,-b}} \lesssim  \epsilon^2
\end{equation}
We first dispense with some good portions of this expression. First, by using an $L^2 L^\infty$ bound 
for the first factor we obtain
\[
\|  [A_{k_1,\alpha}, \partial^\alpha A_{k_2}]\|_{L^2} \lesssim 2^{\frac{k_1}2} (
\| A_{k_1,0}\|_{L^2 \dot H^{\frac32}} + \| A_{k_1,x}\|_{S^1}) \|A_{k_2}\|_{S^1}
\]
which has off-diagonal decay when measured in $X^{b-\frac12,-b}$ at
modulations $j \geq k_1-C$ in the output. It remains to consider low
modulations in the output, namely 
\[
Q_{<k_1-C}[A_{k_1,\alpha}, \partial^\alpha A_{k_2}].
\]
We can peel off some further part of this, using the estimate
\[
\| (I -\mathcal H^*) Q_{<k_1-C}[A_{k_1,\alpha}, \partial^\alpha A_{k_2}]\|_{N} \lesssim 
(\| A_{k_1,0}\|_{L^2 \dot H^{\frac32}} + \| A_{k_1,x}\|_{S^1}) \|A_{k_2}\|_{S^1}
\]
which is a consequence of \eqref{ssn} and \eqref{hsn}, and again suffices for all $b < \frac12$.
Thus we have reduced the problem to an estimate for 
\[
 H^* Q_{<k_1-C}[A_{k_1,\alpha}, \partial^\alpha A_{k_2}] = 
\sum_{j < k_1} Q_{<j-C} [Q_j A_{k_1,\alpha}, Q_{<j-C}\partial^\alpha A_{k_2}]
\]
For each $j$, this fixes the angle $\theta$ between the two factors to $\theta \approx 2^{-(k_1-j)/2}$,
so we can localize to angles of this size. Note carefully that these angles will be essentially disjoint
on the high frequency side, but they will be overlapping on the low frequency side.

From here on we can no longer view this as a bilinear estimate for two
$S^1$ functions.  This is not just a technical difficulty; the direct
bilinear null form estimate for two $S^1$ functions will in effect be
false for $\delta < \frac14$, which is exactly the threshold we need
to cross.

To bypass this difficulty we need to use the fact that $(A_0,A_x)$ are
not arbitrary $L^2 \dot H^\frac32$, respectively $S^1$ functions, but
are solution for the Yang-Mills equation. Thus we can reiterate, and
use again the equation \eqref{ym-cg} specifically for the low
frequency factor $A_{k_1}$. Here we can take advantage of the $Z$
norm.  We will consider $A_0$ and $A_x$ separately: \medskip

{\em a) The contribution of $A_0$.} 
The analysis is simpler in this case.  We simply observe that, once \eqref{l1-a0} is proved, we can use it to 
expand it to a range of mixed norm spaces as follows:
\begin{equation}\label{l1-a0+}
\| |D_x|^{\frac{3}p} A_0\|_{\ell^1 L^{p'} L^p} \lesssim \epsilon^2, \qquad 2 \leq p < \infty 
\end{equation}
We remark that this bound fails when $p = \infty$ (precisely, we can
only control the $\ell^\infty$ norm in that case).  This is why in the
study of the Yang-Mills equation we cannot simply think of $A_0$ as
directly perturbative, and is closely related to the coupling of $A_0$
with $A_x$ in the second null condition leading to the trilinear
estimates in the last section of the paper.

To prove \eqref{l1-a0+}, we only discuss the inhomogeneous term in the
$A_0$ equation, as the terms involving $A_0$ are similar but
simpler. For this, it suffices to prove the $L^1 L^\infty$ counterpart
of \eqref{da0-lp} without off-diagonal decay; then by interpolation we
gain the off-diagonal decay for all intermediate $p$'s, and conclude
as above.  Precisely, we claim that
\begin{equation}
\| |D|^{-2} P_k [A_{j,k_1},  \partial_0 A_{j,k_2}]\|_{L^1 L^\infty}  \lesssim    \|A_{j,k_1}\|_{S^1} 
\| A_{j,k_2} \|_{S^1}
\end{equation}
The case of unbalanced frequency interactions is easy, just by using
$L^2 L^\infty$ Strichartz bounds for both factors. The more delicate
case is that of $high \times high \to low$ interactions, where $k <
k_1 = k_2$. There simply using $L^2 L^\infty$ for both factors would
yield a bad $2^{2(k_1-k)}$ bound.  To remedy this, we partition both
$A_{k_1}$ and $A_{k_2}$ in spatial frequency with respect to a lattice
of cubes $\mathcal C_k$ of size $2^k$, so that only opposite cubes
will contribute to the output. Then by Cauchy-Schwarz we have
\[
\begin{split}
\| |D|^{-2} P_k [A_{j,k_1},  \partial_0 A_{j,k_2}]\|_{L^1 L^\infty}^2 \lesssim \ & 2^{-4k} 2^{2k_1} \left( \sum_{ \mathcal C_k}
\| P_{ \mathcal C_k} A_{j,k_1} \|_{L^2 L^\infty}^2  \right) \left( \sum_{ \mathcal C_k}
\| P_{ \mathcal C_k} A_{j,k_2} \|_{L^2 L^\infty}^2  \right) 
\\
\lesssim \ & \|  A_{j,k_1} \|_{S^1}^2 \| A_{j,k_2} \|_{S^1}^2
\end{split}
\]
where we have used the next to last component of the $S_k^\omega(l)$
norm in \eqref{Sl_def} with $k = k_{1,2}$, $k'=k$ and $l'=0$.

We can now use \eqref{l1-a0+} to bound directly all 
 $ low\times high$ frequency interactions in the expression $[A_{0,k_1},\partial_0
A_{x,k_2}]$.  Indeed, by Sobolev embeddings we have 
\[
\| |D_x|^{-\frac{1}p} A_0 \|_{\ell^1 L^{p'} L^\infty} \lesssim \epsilon^2.
\]
Using this we can estimate 
\[
\||D_x|^{-\frac1p}  [A_0,\partial_0 A_x]\|_{L^{p'} L^2} \lesssim \epsilon^2 \|\partial_0 A_x\|_{L^\infty L^2}
\]
which gives the desired bound as in \eqref{l1-ax} with $b= \frac12- \frac1p$ in view of the embedding
\[
L^{p'} L^2 \subset X^{0,\frac1p-\frac12}
\]
Since $p$ is arbitrarily large, we obtain the desired bound for all $0\leq b <\frac12$.
\medskip

{\em b) The contribution of $A_x$.}  Here we begin with the bounds
\eqref{l12-to-z} and \eqref{hl-to-z}, which allow us to split $A_x$ into two
components,
\[
A_x = A_x^{good} + A_x^{bad}
\]
where $A_x^{good}$ satisfies a favorable $Z$ bound,
\[
\|A_x^{good}\|_{\ell^1 Z} \lesssim \epsilon^2
\]
and $A_x^{bad}$ is the remainder, namely
\[
A_x^{bad} = \Box^{-1} |D_x|^{-1} \sum_{k < k_1 = k_2} \mathcal H_k N(A_{k_1},A_{k_2}) 
\]
We can use the $\ell^1 Z$ bound directly for $A_x^{good}$ due to
\eqref{zsn}, which yields off-diagonal decay for all $\delta > 0$.

For $A_x^{bad}$, on the other hand, we have a favorable $S^1$ bound
with off-diagonal decay, due to \eqref{null-main},  and a $Z$ bound without off-diagonal
decay. Hence interpolating the $X_{\infty}^{1,\frac12}$ component of the $S^1$ norm
with the $Z$ norm we obtain all intermediate bounds for $A_x^{bad}$ with
off-diagonal decay. Then we can conclude as in the $A_0$ case.
This suffices for all $b < \frac12$.

\bigskip

\subsection{Proof of Proposition~\ref{p:tri}, the bulk part}

Here we consider most of the proof of Proposition~\ref{p:tri}, modulo
the more delicate trilinear part in Lemma~\ref{l:tri}. We extend $B_j$
outside the interval $I$ as free waves, and $B_0$ by zero. Then we
seek to prove the bound in the proposition on the full real line. This
allows us to consider modulation localizations.
We decompose the bilinear form
\[
[B_{\alpha,<k}, \partial^\alpha C_k] = (I - \H^*) [B_{\alpha,<k}, \partial^\alpha C_k] 
+ \H^*  [B_{\alpha,<k}, \partial^\alpha C_k] 
\]
In the first term we separate the $B_j$ and $B_0$ components.  For
$B_j$ we use the $S$ norm bound, together with the null condition and
the estimate \eqref{ssn}. For $B_0$ we use the $L^2 \dot
H^{s+\frac12}$ bound in \eqref{hsn}. It remains to consider the second term,
for which the $B_j$ and $B_0$ terms can no longer be separated:

\begin{lemma}\label{l:tri}
  Suppose that $B \in S^s$ solves the linearized equation
  \eqref{ym-cg-lin} in a time interval $I$. Extend $B_j$ outside $I$ as free waves,
and $B_0$ by zero. Then for $s < 1$, close to $1$ we have the global estimate
\begin{equation}
\|\H^* [B_{\alpha,<k}, \partial^\alpha C_k]\|_{N^{s-1}} \lesssim \epsilon \|B\|_{S^s}    \|C_k\|_{S^1}
\end{equation}
\end{lemma}
This remaining lemma is proved in Section~\ref{s:tri}.

\section{ The gauge transformation }
\label{s:para}
This section is devoted to the proof of Proposition~\ref{p:para}.

\subsection{Equivalent formulations}

A first difficulty we encounter in the proof of the proposition is that the 
equations for $B_j$ are coupled via the Leray projection. 
Fortunately, it turns out that the coupling is  perturbative, and we can discard 
the projector and work with the uncoupled equations:

\begin{proposition} \label{p:para-uncouple}
Assume that $A$ is a solution to the Yang-Mills equation in Coulomb gauge which 
satisfies
\[
\|A_j\|_{S^1} \leq  \epsilon \ll 1.
\]
 Then for the equation 
\begin{equation}\label{parab-noP}
\Box B_{j,k}  + 2 [A_{\alpha,<k}, \partial^\alpha B_{j,k}] = F_{j,k}
\end{equation}
 we have the following linear estimate:
\begin{equation}
\|B_k\|_{S^1} \lesssim ( \|F_k\|_{N} + \| B_k[0]\|_{\H})
\end{equation} 
\end{proposition}

To transition from this to Proposition~\ref{para} it suffices to estimate the difference,
namely 
\[
\| \Delta^{-1} \partial_j  [\partial_l A_{\alpha,<k}, \partial^\alpha B_{l,k}] \|_{N} \lesssim
\| A_{\alpha,<k}\|_{S^1} \|  B_{l,k}\|_{S^1}
\]
(using the null condition via $\nabla \cdot B = 0$). This is a pure
$S$ bound as we have an extra derivative on the low frequency, and follows by \eqref{null-main}..

In view of the estimates  in Proposition~\ref{p:pert}, the frequency localized result in 
Proposition~\ref{p:para-uncouple}
is equivalent to the following nonlocalized version:

\begin{proposition} \label{p:para-noloc}
Assume that $A$ is a solution to the Yang-Mills equation in Coulomb gauge which 
satisfies
\[
\|A_j\|_{S^1} \leq  \epsilon \ll 1.
\]
 Then for the equation
\[
\Box_A B = F
\]
we have the following linear estimate:
\begin{equation}
\|B\|_{S^1} \lesssim ( \|F\|_{N} + \| B_k[0]\|_{\H})
\end{equation} 
\end{proposition}

Further, in view of the same estimates in Proposition~\ref{p:pert}, the last proposition
is equivalent to the existence of a good parametrix for the corresponding 
paradifferential problem, see the proof of Theorem 5 in \cite{MKG}:

\begin{proposition} \label{p:para-app}
Assume that $A$ is a solution to the Yang-Mills equation in Coulomb gauge which 
satisfies
\[
\|A_j\|_{S^1} \leq  \epsilon \ll 1.
\]
Then for each frequency localized initial data $(B_{0k}, B_{1k}) \in \H$ and
inhomogeneous term $F_k \in N$ there exists an approximate solution
$B_k$ for the equation \eqref{parab}, in the sense that:

(i) We have the following linear estimate:
\begin{equation}\label{eq:Sbound}
\|B_k\|_{S^1} \lesssim ( \|F_k\|_{N} + \| (B_{0k},B_{1k})\|_{\H})
\end{equation}

(ii) We have the small error estimates:
\begin{equation}\label{eq:errorbound}
  \|B_k[0] -(B_{0k},B_{1k})\|_{\H} +
 \|\Box B_{k}  + 2  [A_{\alpha,<k}, \partial^\alpha B_{k}] -F_k\|_{N}
  \lesssim \epsilon ( \|F_k\|_{N} + \| (B_{0k},B_{1k})\|_{\H})
\end{equation} 
\end{proposition}

\subsection{Heuristic considerations}
Naively, our goal is to ``gauge out'' the magnetic potential, i.e. to
find a suitable transformation, which we call {\em the renormalization operator},
which, up to small errors,  interchanges the magnetic wave equation with the 
flat d'Alembertian.  We now outline several considerations which eventually lead to our 
renormalization operators.
\bigskip

{\em 1. Scalar conjugations.} We would like to make a gauge transformation
\[
C_k = O^{-1}_{<k} B_k O_{<k}
\]
where $O_{<k}$ is a $\G$ valued map which is also localized at lower
frequency, in order to turn the above equation into
\[
\Box C_k  = error 
\]
A direct computation gives 
\[
\Box C_k = O^{-1}_{<k} ( \Box B_k -  [\partial_\alpha O_{<k}O^{-1}_{<k} , \partial^\alpha B_k]  
+ l.o.t.)     O_{<k}  
\]
where in ``lower order terms'' we have included expressions where both
derivatives apply to the lower frequency term $O_{<k}$.  To insure
cancellation here we would need to require that
\begin{equation}\label{hr1}
\partial_\alpha O_{<k}O^{-1}_{<k} = - A_{<k,\alpha}
\end{equation}
Solving this exactly would require  the connection $A$ to have zero curvature,
which is obviously unacceptable. 

\bigskip

{\em 2. Pseudodifferential renormalizations.}  The first remedy to the
above failure of complete integrability is then to allow the conjugation
by $O$ to be a pseudodifferential operator, whose symbol $O(t,x,\xi)$
would then have to satisfy
\begin{equation}\label{hr2}
  \p_\aa O_{<k}  O_{<k}^{-1}    \xi^\aa \approx  - A_{<k,\aa} \xi^\aa
\end{equation}
Algebraically this means that for each $\xi$ we renormalize $A_\alpha$ in a single 
direction, which is now possible. 

However, from an analytic perspective this implies that the symbol of
$O$ will have singularities associated to space-time frequencies
$\eta$ so that $\eta^\alpha \xi_\alpha = 0$.  To bypass this second difficulty
we observe that solutions to the linear wave equation are localized
in frequency on the null cone $\xi_\alpha \xi^\alpha = 0$, while the leading part 
of $A_\alpha$ are also primarily localized on the cone $\eta_\alpha \eta^\alpha = 0$.
This is useful because when both $\xi$ and $\eta$ are on the cone, the expression
$\eta^\alpha \xi_\alpha $ cannot vanish unless $\xi$ and $\eta$ are collinear.

To take advantage of the above observation,  we first note that we 
are in a paradifferential situation where $|\eta| \ll |\xi|$, therefore
the two cones $\xi_0 = \pm |\xi|$ are completely uncoupled, and will be renormalized
separately using different parametrices $O_{\pm}$. In particular this will allow us to 
work with symbols $O_{\pm}(t,x,\xi)$ which do not depend on $\xi_0$, therefore 
they act separately on time slices. Thus we replace \eqref{hr2} by
\begin{equation}\label{hr3}
 (\omega_j   \p_j \pm \p_0)   O_{<k,\pm}  O^{-1}_{<k, \pm}  \approx  
- (\omega_j A_{<k,j} \pm  A_{<k,0}), \qquad 
\omega = \xi' |\xi'|^{-1} 
\end{equation}

\bigskip

{\em 3. Pseudodifferential vs. nonlinear: divide and conquer.}
Above it was easy to replace $\xi_0$ by $\pm |\xi|$, but, due to the
nonlinear nature of the expression on the left, it is far less
straightforward to do the same for $\eta$. In order to uncouple the
pseudodifferential and nonlinear aspects of the analysis, we introduce
an intermediate step, namely
\begin{equation}\label{hr4}
(\omega_j   \p_j \pm \p_0)   O_{<k,\pm}  O^{-1}_{<k, \pm}    \approx (\omega_j   \p_j \pm \p_0) 
 \Psi_{<k,\pm}  \approx - (\omega_j A_{<k,j} \xi^j \pm  A_{<k,0})
\end{equation}

The transition from $A$ to $\Psi$ is pseudodifferential but linear,
therefore appropriately (so that only differential operators in time
are used) replacing $\eta_0$ by $|\eta'|$ we can rewrite the second
part of the above relation as
\begin{equation}\label{hr5}
(\p_j^2-  (\omega_j   \p_j)^2)   \Psi_{<k,\pm}  \approx (\p_0 \pm p_j \omega_j)
 (\omega_j A_{<k,j} \xi^j \pm  A_{<k,0})
\end{equation}
This transition is similar to the related step in the previous Maxwell-Klein Gordon result
\cite{MKG}.

The step from $\Psi$ to $O$, on the other hand, is more algebraic in
nature, and resembles the similar step in the study of wave maps, see
\cite{wm}.  Precisely, for fixed $\omega$ we seek to have the more
general approximate relation 
\[
\nabla O_{<k,\pm} O^{-1}_{<k,\pm} \approx \nabla \Psi_{<k,\pm}
\]
 Differentiating with respect to the frequency  parameter $h < k$ we obtain
\[
 \nabla (\partial_h O_{<h,\pm} O^{-1}_{<h,\pm}) + [\partial_h  O_{<h,\pm} O^{-1}_{<h,\pm},\nabla 
O_{<h,\pm} O^{-1}_{<h,\pm}]
\approx \nabla \Psi_h 
\]
The second term on the left is quadratic, and has the added feature
that the derivative applies to the lower frequency factor. Hence it is
natural to discard it. Then it is natural to obtain $O$ by integrating
$\Psi_h$ with respect to the frequency parameter $h$, i.e.
\begin{equation}\label{hr7}
\partial_h O_{<h} O_{<h}^{-1} = \Psi_h
\end{equation}
which is a well defined $\G$ valued evolution.

\bigskip {\em 4. Perturbative vs. renormalizable.} The last question
we need to address is whether we need to feed all or only part of $A$
into the construction of the renormalization operators. For simplicity 
one might attempt first the former, but, as it turns out, there are two distinct 
obstructions for this strategy. Of course, the downside of choosing the latter 
is that the remaining part of $A$ needs to be treated perturbatively.

The first issue is related to the symbol regularity for $O$. We
observe that even with $\xi$ and $\eta$ restricted to the null cones,
the expression $\eta^\alpha \xi_\alpha = 0$ can still vanish but only
when $\xi$ and $\eta$ are collinear. This is the well-known difficulty
of {\em small angle interactions}.  To avoid the corresponding symbol
singularities, we will excise the small angle interactions from the
linear flow \eqref{para} and treat them perturbatively; this is where
the null condition comes in handy. Unfortunately, it is too much to
ask to uniformly excise the small angle interactions, and instead we
do this in a frequency dependent fashion.  Precisely, we will treat
perturbatively only the interactions at angles
\[
|\angle(\xi,\eta) | \lesssim (|\eta|/|\xi|)^\delta  
\]
where $\delta$ is a universal small parameter.  This considerations will affect the 
linear step in the above construction, i.e. the transition from $A$ to $\Psi$.

The second issue is related to the fact that the expression
$\partial^\alpha \Psi \xi^\alpha$ vanishes in frequency on the
hyperplane $\eta_\alpha \xi_\alpha = 0$.  Thus, it cannot at all
cancel $A$ in the region near this hyperplane. It follows that, in
order for our strategy to work, the portion of $A$ near this
hyperplane must be perturbative. But then it is pointless (and indeed
counterproductive) to allow it to participate in the construction of
the renormalization operator. Further, $A_0$'s leading contribution
lies in this region. Thus it is natural to place $A_0$ fully on the perturbative side.

\subsection{The parametrix}

Here we define the parametrix for $\Box_A$ that yields the proof of
Proposition~\ref{p:para-app}.  By scaling we can assume that $k = 0$
in the Proposition, and drop it from the notations. For the rest of
the section we will use $k < 0$ to denote dyadic frequencies for $A$,
$\Psi$ and $O$.

Following the above heuristics, we begin with $\xi$ of size $O(1)$ and 
$\omega = \xi/|\xi|$. Then we decompose $A_{j,<0}$ into a leading part $A^{main,\pm}_{j,<0}$
and a perturbative part $A^{pert,\pm}_{j,<0}$ in a fashion which depends on $\omega$.
Here the choice of $\pm$ sign corresponds to the two cones $\tau \pm |\xi| = 0$.

The first difficulty we face is that $A_j$ are a-priori only defined in a fixed time interval $I$,
while our analysis uses many modulation localizations, which are nonlocal in time.
To address this issue, we start with $A_j$ in $I$, and extend them in time outside $I$ 
as free waves. By Proposition~\ref{p:intervals}, such an extension does not increase significantly
the $S^1$ norm of $A$.

Denoting the Fourier variables for $A$ by $(\sigma,\eta)$, the two relevant 
geometric objects are the null cone $|\sigma| = |\eta|$ and the null plane 
$\sigma \pm \eta \cdot \omega = 0$. 

It is natural to consider the two components of $\eta$, namely $\eta \cdot \omega$
and $\eta_{\perp} = \eta - \omega \eta \cdot \omega$.
We first define a partition of the Fourier space 
\[
\R^{4+1} = D^{\omega,\pm}_{cone} \cup D^{\omega,\pm}_{null} \cup D^{\omega,\pm}_{out}
\]
where the three regions are homogeneous, symmetric with respect to the origin and
\[
\begin{split}
D^{\omega,\pm}_{cone}  =& \ \{ \sgn(\sigma)(\sigma \pm \eta \cdot \omega) < - \frac{1}{16}  |\eta|^{-1} (|\eta_{\perp}|^2+ |\sigma \pm \eta \cdot \omega|^2)\} \cap \{ |\sigma| < 4 |\xi|\},
\\
  D^{\omega,\pm}_{null} = & \ \{ |\sigma \pm \eta \cdot \omega| 
\lesssim \frac{1}{8} |\eta|^{-1} (|\eta_{\perp}|^2+ |\sigma \pm \eta \cdot \omega|^2)\},
\\
 D^{\omega,\pm}_{out}= & \ \{ \sgn(\sigma)(\sigma \pm \eta \cdot \omega) > 
\frac{1}{16}  |\eta|^{-1} (|\eta_{\perp}|^2+ |\sigma \pm \eta \cdot \omega|^2)\} \cup \{ |\sigma > 2 |\xi|\}
\end{split}
\]
Correspondingly we consider a partition of unit 
\[
1 = \Pi^{\omega,\pm}_{cone}+ \Pi^{\omega,\pm}_{null} +  \Pi^{\omega,\pm}_{out}
\]
where the regularity of these symbols degenerates where $(\sigma,\eta)$ and $(\mp 1, \omega)$ are collinear,
\[
\partial_{\sigma,\eta_{perp}}^\alpha\partial_{\eta_{||}}^\beta |\Pi^{\omega,\pm}_{*}| \lesssim
\left(\frac{|\eta|}{|\eta_{\perp}| + (|\eta||\sigma \pm  \eta \cdot \omega|)^\frac12 }\right)^{2|\alpha|+ |\beta|}
\]

Our second partition is with respect to angles. Given an angle $0 < \theta < \Pi/2$ we 
 partition the Fourier space as 
\[
\R^{4+1} = D^{\omega,\pm}_{ < 2 \theta} \cup D^{\omega,\pm}_{ > \theta/2} 
\]
where 
\[
D^{\omega,\pm}_{ < 2 \theta} = \{ \angle(\omega, - \eta\, \sgn(\sigma))  < 2 \theta \}, \qquad 
 D^{\omega,\pm}_{ < 2 \theta} = \{ \angle(\omega, - \eta\, \sgn(\sigma))  >  \theta/2 \}
\]
Correspondingly we define a partition of unit 
\[
1 = \Pi^{\omega,\pm}_{<\theta}+ \Pi^{\omega,\pm}_{> \theta}
\]
with the obvious symbol regularity.

Now we are ready to define the decomposition of $A_{j,<0}$, namely 
\[
A_{j,<0}(t,x) = A^{main,\pm}_{j,<0}(t,x,\xi) +  A^{pert,\pm}_{j,<0}(t,x,\xi) 
\]
where 
\[
\begin{split}
A^{main,\pm}_{j,<0}(t,x,\xi) = & \Pi^{\omega,\pm}_{> |\eta|^\delta} \Pi^{\omega,\pm}_{cone} A_{j,<0}\\ 
A^{pert,\pm}_{j,<0}(t,x,\xi) = & (\Pi^{\omega,\pm}_{< |\eta|^\delta} \Pi^{\omega,\pm}_{cone}+ 
\Pi^{\omega,\pm}_{null} +  \Pi^{\omega,\pm}_{out})A_{j,<0}
\end{split}
\]
Here we make two observations. First, the size of the excised angle
decreases with the size of the frequency $|\eta|$. This is needed in
order to guarantee decay of the perturbative errors as $|\eta| \to
0$. Secondly, even though $ \Pi^{\omega,\pm}_{> |\eta|^\delta}$ has a
jump discontinuity at $\sigma = 0$, the symbol
$\Pi^{\omega,\pm}_{cone}$ vanishes at $\sigma = 0$ so the
discontinuity disappears.

Next we use the symbols $A^{main,\pm}_{j,<1}$ to define the 
$\g$ valued zero homogeneous symbols $\Psi_{\pm} = \Psi_{< 0, \pm} $,
by
\begin{equation}\label{defpsi}
  \Psi_{\pm}(t,x,\xi) \ = \  - L^\omega_{\mp} \Delta^{-1}_{\omega^\perp}  A^{main,\pm}_{j,<1}
\end{equation}
where
\[
L^\omega_{\pm}=\partial_t\pm\omega\cdot\nabla_x, \qquad
\Delta_{\omega^\perp} =\Delta-(\omega\cdot\nabla_x)^2,
\]

Later in the analysis we will also use the frequency localized
functions $A^{main,\pm}_{j,k}$ and $ \Psi_{\pm,k}$ defined for a
continuous dyadic parameter $h < 0$ so that
\begin{equation}\label{eq:intident}
A^{main,\pm}_{j,<k} = \int_{-\infty}^k A^{main,\pm}_{j,<h} dh,
\qquad  
  \Psi_{\pm,<k} = \int_{-\infty}^k  \Psi_{\pm,h} dh
\end{equation}

Once we have the $\g$ valued symbols $\Psi_{\pm,k}$, we define the
zero homogeneous $\G$ valued symbols $O_{\pm,<k} (t,x,\xi)$ by
solving the following differential equation on the Lie group $\G$,
\begin{equation}\label{psi-def}
\frac{d}{dk} O_{<k,\pm} O_{<k,\pm}^{-1} =  \Psi_{\pm,k} , \qquad O_{-\infty,\pm} = const
\end{equation}
Here the ode is solved separately for each $(x,\xi)$, and the solution is uniquely 
determined up to multiplication $O \to O U$ with $U= U(x,\xi)$ 
an arbitrary $\G$-valued function.  While a-priori $U$ may depend on 
$x$ and $\xi$, we can partially eliminate this dependence by requiring 
that 
\begin{equation}
\lim_{k \to - \infty} \|\partial_x O_{<k,\pm}(t,x,\xi)\|_{L^\infty} = 0,
\end{equation}
This uniquely determines $O_{\pm}$ up to multiplication with respect a field
$U(\xi)$.  We will allow this ambiguity to remain; all of our
results will be invariant with respect to such a conjugation.

To construct the parametrix for the equation \eqref{parab} we fix a
large universal constant $\kappa$ (e.g. $\kappa= 10$), and use the symbols
\[
O_{\pm}(x,D) :=  O_{\pm, < -\kappa }(x,D)
\]
and the associated operators $Op(Ad(O_{\pm}))(x,D)$. To do this we
conjugate the constant coefficient wave flow with respect to the pair
$Op(Ad( O_{\pm}))(x,D)$ on the left, respectively their adjoints
$Op(Ad( O_{\pm}^{-1}))(D,y)$ on the right. The $\pm$ operators apply to
the $\pm$ waves.  

It is important to remark here on a minor technical point that will
affect the exact definition of the parametrix. Precisely, our
parametrix should take frequency one functions to frequency one
functions. However, even though the symbols $\Psi_{\pm,k}$ have sharp
frequency localization, the symbols $O_{\pm,<k}$ are defined in a nonlinear fashion 
and do not fully inherit this property. Thus, instead of using directly the 
operators $Op(Ad(O_{\pm}))(x,D)$ in our parametrix, we need to relocalize 
these symbols at frequencies much smaller than $1$; for this we use the notation
\[
(Ad(O_{\pm}(x,\xi)))_{< 0} = P(|D_x| \ll 1) Ad(O(x,\xi)),
\]
which is nothing but a localized average of $O_{\pm}(x,\xi)$ on the unit spatial scale.
We further remark that this truncation is largely harmless, because the symbols $O_{\pm}$
exhibit rapid decay with favorable bounds at all frequencies much larger than $2^{-\kappa}$.
This issue is discussed in detail in \cite{MKG}, and we will only go over it lightly in here.

The approximate solution $B$ will have the form
\begin{equation}\label{eq:Bparametrix}
\begin{split}
B(t) =\sum_{\pm}   & \  \frac12 Op(Ad( O_{\pm})_{<0}) (t,x,D) e^{\pm it|D|}  Op(Ad( 
O^{-1}_{\pm})_{<0})(D,0,y)
( B_0 \pm i|D|^{-1}B _1)    \\ & \ +
Op(Ad( O_{\pm})_{<0})(t,x,D) \frac{1}{|D|} K^{\pm} Op(Ad( 
O^{-1}_{\pm})_{<0})(D,s,y) F
\end{split}
\end{equation}
where
\[
K^\pm f(t) = \int_{0}^t e^{\pm i(t-s)|D|} f(s) ds
\]
represents the solution to
\[
(\partial_t \mp i|D|) u = f, \qquad u(0) = 0
\]
By analogy with the MKG problem, we need to prove the following 
bounds:

\begin{theorem}\label{t:para}
  The frequency localized renormalization operators $Op(Ad(
  O_{\pm})_{<0})(t,x,D)$ have the following mapping properties with $Z
  \in \{ N_0,L^2, N^*_0\}$:
  \begin{align}
   Op(Ad(
  O_{\pm})_{<0})(t,x,D) : \quad & Z \to Z \ , \label{N_mapping}\\
    \partial_t  Op(Ad( O_{\pm})_{<0})(t,x,D): \quad & Z
    \to \epsilon Z \ , \label{dtS_mapping} \\
   Op(Ad( O_{\pm})_{<0})(t,x,D)Op(Ad( O^{-1}_{\pm})_{<0})(D,y,s) -I: \quad &
    Z \to\epsilon Z,
    \  \label{S_mapping}\\
   Op(Ad( O_{\pm})_{<0}) \Box - \Box^p_{A_{<0}}Op(Ad( O_{\pm})_{<0})
    : \quad& S^\sharp_{0,\pm}\to \epsilon N_{0,\pm} \ . \label{error_ests}\\
  Op(Ad( O_{\pm})_{<0})  : \quad &S_0^\sharp \to S_0 \ ,
    \label{parametrix_bound}
  \end{align}
where 
\[
\Box^p_{A_{<0}} = \Box + 2 ad(A_{\alpha,<0}) \partial^\alpha.
\]
\end{theorem}

We remark that, as we have constructed it above, $O$ is defined
globally in time, and is based on the free wave extension of $A_j$
outside the interval $I$.  All the bounds in the above theorem will
also be proved globally in time; indeed, with the exception of the
error estimate \eqref{error_ests}, only the $S^1$ norm of $A_x$ and
the Coulomb Gauge condition are used. However, in order to prove the
bound \eqref{error_ests} we will need to use the Yang-Mills equation 
for $A_x$ in $I$, as well as the definition of $A_0$ in terms of $A_x$, 
also in $I$.

The rest of the paper are devoted to the proof of the theorem.
For the remainder of this section we use the Theorem to conclude the proof 
of Proposition~\ref{p:para-app}: 
\begin{proof}[Proof of Proposition~\ref{p:para-app}.] This is completely
  analogous to the proof of Theorem 4 in \cite{MKG}. We define the
  approximate solution via \eqref{eq:Bparametrix}. Then the bound
  \eqref{eq:Sbound} follows from \eqref{N_mapping},
  \eqref{parametrix_bound}. 

  Next, we prove \eqref{eq:errorbound}. For the homogeneous part of
  the parametrix at time $t = 0$, we have
\begin{align*}
B(0) - B_0 = &\frac{1}{2}\sum_{\pm}Op(Ad( O_{\pm})_{<0})(0,x,D)  Op(Ad(O^{-1}_{\pm})_{<0})(D,0,y)
( B_0 \pm i|D|^{-1}B _1)  - B_0\\
& = \big[\frac{1}{2}\sum_{\pm} Op(Ad( O_{\pm})_{<0})(0,x,D)  Op(Ad(O^{-1}_{\pm})_{<0}(D,0,y))- I\big]( B_0 \pm i|D|^{-1}B _1)
\end{align*}
Thus the bound 
\[
\big\|B(0) - B_0\big\|_{\dot{H}^1}\lesssim \epsilon\big\|(B_0, B_1)\big\|_{\mathcal{H}}
\]
is a consequence of \eqref{S_mapping} applied to $Z = L^2$. Further, the inequality 
\[
\big\|\partial_t B(0) - B_1\big\|_{L^2}\lesssim \epsilon\big(\big\|(B_0, B_1)\big\|_{\mathcal{H}} + \big\|F\big\|_N\big)
\]
is a consequence of \eqref{dtS_mapping}, \eqref{S_mapping}, see the
proof of Theorem 5 in \cite{MKG}. Finally, for the the inhomogeneous
term, we have the following
\begin{align*}
\Box B  + 2[A_{\alpha, <0}, \partial^{\alpha}B] &= 
 \sum_{\pm}\big[\Box^p_{A_{<0}} Op(Ad( O_{\pm})_{<0})(t,x,D)-  Op(Ad( O_{\pm})_{<0})(t,x,D) \Box\big]B_{\pm}\\
& + \frac{1}{2}\sum_{\pm}\big[Op(Ad( O_{\pm})_{<0})(t,x,D)  Op(Ad(O^{-1}_{\pm})_{<0})(D,t,y) - 1\big]F\\
& +  \frac{1}{2}\sum_{\pm}\pm\big[Op(Ad( O_{\pm})_{<0})(t,x,D)|D|^{-1}Op(Ad(O^{-1}_{\pm})_{<0})(D,t,y)- |D|^{-1}\big]\partial_tF\\
&+\sum_{\pm}Op(Ad( O_{\pm})_{<0})(t,x,D)|D|^{-1}\partial_tOp(Ad(O^{-1}_{\pm})_{<0})(D,t,y) F 
\end{align*}
where we set 
\[
B_{\pm} = e^{\pm it|D|} Op(Ad(O^{-1}_{\pm})_{<0})(D,0,y)
( B_0 \pm i|D|^{-1}B _1) + |D|^{-1}K^{\pm} Op(Ad(O^{-1}_{\pm})_{<0})(D,s,y)F
\]
The first term on the right is handled by combining \eqref{error_ests}
with \eqref{N_mapping}, and the last three terms are controlled using
\eqref{S_mapping} and \eqref{dtS_mapping}.

\end{proof}


\section{ Decomposability and symbol 
bounds for $\Psi$ and $O$} \label{s:decomp}

In this section we review the notion of disposability, which is a
convenient technical tool allowing us to easily deal with issues
related to symbol calculus, which would otherwise be quite technical
in the context of our function spaces. Then we provide bounds for 
$\Psi$ and $O$, first pointwise and then in disposable spaces.

 This  section uses only the spatial components $A_{j,<0}$ at low frequency.
We assume throughout that this is divergence free,  with $\| A\|_S \leq \epsilon$ and
frequency envelope $c_k$. We fix the $\pm$ sign to $+$ and drop it from the 
notations.

\subsection{A review of the Decomposable Calculus}

First we discuss the notion of decomposable function spaces and estimates.  
This has originated in \cite{MR2100060}, \cite{Krieger:2005wh}.
 
A  zero homogeneous symbol $c(t,x;\xi)$ is said
to be in ``decomposable $L^q(L^r)$'' if $c=\sum_\theta c^{(\theta)}$,
$\theta\in 2^{-\mathbb{N}}$, and:
\begin{equation}
  \sum_\theta \lp{c^{(\theta)}}{D_\theta \big(L_t^{q}(L_x^{r})\big)} 
  \ < \ \infty \ , \label{decomp_sum}
\end{equation}
where, adhering to the definition in  \cite{MR2100060} and with $n = 4$ throughout, we put:
\begin{equation}
  \lp{c^{(\theta)}}{D_\theta \big(L_t^{q}(L_x^{r})\big)}  \ = \ 
  \big\Vert \Big( \sum_{k=0}^{10n} \
  \sum_\phi \ \sup_\omega\ \lp{b^{\phi}_\theta\
    (\theta \nabla_\xi )^k \ c^{(\theta)} }
  {L_x^{r}}^2\Big)^\frac{1}{2}\big\Vert_{L_t^{q}}
  \ . \label{theta_decomp_norm}
\end{equation}
Here $b^{\phi}_\theta(\xi)$ denotes a cutoff on a solid angular sector
${\big|\xi|\xi|^{-1} - \phi\big|\leqslant \theta}$ for a fixed
$\phi\in\mathbb{S}^{n-1}$, and the sum is taken over a uniformly
finitely overlapping collection.  We define $\lp{b}{DL^q(L^r)}$ as the
infimum over all sums \eqref{decomp_sum}.  In \cite{Krieger:2005wh} it is shown
that the following H\"older type inequality holds:
\begin{equation}
  \lp{\prod_{i=1}^m b_i}{DL^q(L^r)} \ \lesssim \ 
  \prod_{i=1}^m \lp{b_i}{DL^{q_i}(L^{r_i})} \ , \ \ 
  (q^{-1},r^{-1})=\sum_i (q_i^{-1},r_i^{-1}) 
  \ . \label{decomp_holder}
\end{equation}
In the sequel we only need a special case of decompositions provided
in terms of these norms:

\begin{lemma}[Decomposability Lemma](\cite{MKG}, Lemma 7.1)
  Let $A(t,x;D)$ be any pseudodifferential operator with symbol
  $a(t,x;\xi)$.  Suppose $A$ satisfies the fixed time bound:
  \begin{equation}
    \sup_t \lp{A(t,x;D)}{L^2\to L^2} \ \lesssim \ 1 \ . \label{ab}
  \end{equation}
  Then for any symbol $c(t,x;\xi)\in DL^{q}(L^r)$ one has the
  space-time bounds:
  \begin{equation}\label{decomp1}
    \begin{split}
      \lp{(ac)(t,x;D)}{L^{q_1} L^2\to L^{q_2}(L^{r_2})} \lesssim &
      \lp{c}{DL^{q}(L^r)} \ , \\
      \frac{1}{q_1} + \frac{1}{q} = \frac{1}{q_2}, \qquad \frac12 +
      \frac1r = \frac{1}{r_2}, \qquad & 1 \leq q_1,q_2,q,r,r_2 \leq
      \infty
    \end{split}
  \end{equation}
\end{lemma}

In the sequel it will also be useful for us to treat estimates for
products of operators in a modular way. Recall that if $a(x,\xi)$ and
$b(x,\xi)$ are symbols, then $a^rb^r- (ab)^r\approx i(\partial_x
a \partial_\xi b)^r$. This formula is not exact, but it leads to an
estimate, which is a simple variant of Lemma 7. 2 in \cite{MKG}:

\begin{lemma}[Decomposable product calculus]
  Let $a(x,\xi)$ and $b(x,\xi)$ be smooth symbols, and $\lambda > 0$. Then:
  \begin{equation}
    \lp{a^rb^r- (ab)^r}{L^r(L^2)
      \to L^q(L^2)}  \lesssim  \sup_{1 \leq |\alpha| < N}   
\lambda^{|\alpha|} \lp{(\nabla_x a)^r}{L^r(L^2)\to L^{p_1}(L^2)}
 \sup_{1 \leq |\alpha| < N}  \lambda^{|\alpha|}     \lp{\nabla^\alpha_\xi b}{L^{p_1}L^2 \to L^q L^2}  \label{spec_decomp1}
  \end{equation}
\end{lemma}

\subsection{Bounds for $A$}

Here we state the decomposability bounds for $A$, see  \cite{MKG}, Lemma 7.3:

\begin{lemma}
The functions $A_x \cdot \omega$, $A_0$ satisfy the following decomposability bounds:
\begin{equation}\label{A_decomp2}
\| P_k (A_x^{(\theta)} \cdot \omega, A_0)\|_{DL^p L^\infty} \lesssim 2^{(1 - \frac1p)k} \theta^{\frac52 - \frac2p}
\end{equation}
\end{lemma}

\subsection{ Bounds for $\Psi$}

For the purpose of our first step we use the frame determined by 
$\omega = \xi |\xi|^{-1}$ and its orthogonal complement $\omega^{\perp}$ to describe the 
regularity of $\Psi$. We have 

\begin{lemma}
The functions $\Psi_k(t,x,\xi)$ satisfy the  following bounds for fixed $t$ and $\xi$:
\begin{equation}\label{psi-2-ft}
 \| \nabla_{\omega^\perp} \nabla \Psi_k\|_{L^2} \lesssim c_k 
\end{equation}
\begin{equation}\label{psi-2-fta}
  \|  \nabla^2 \Psi_k\|_{L^2}\lesssim 2^{-\delta k}c_k 
\end{equation}
We also get the bounds 
\begin{equation}
\| \nabla_{\xi}^N\nabla^2 \Psi_k\|_{L^2}\lesssim 2^{-(N+1)\delta k}c_k.
\end{equation}

\end{lemma}
We remark that, as a consequence of Bernstein's inequality, the bound 
\eqref{psi-2-ft} implies the pointwise bounds
\begin{equation}\label{psi-inf-ft}
 \| \Psi_k\|_{L^\infty} \lesssim c_k 
\end{equation}
Also we consider $L^p$ norms at fixed time. Fixing $\xi $ we use the
orthonormal frame associated to $\xi$, and the mixed norms
$L^2_{\omega} L^6$ and $L^\infty_{\omega} L^3$.  By Bernstein's inequality,
from \eqref{psi-2-ft} we obtain
\begin{equation}\label{psi-mixed-ft}
\| \nabla_x \Psi_k\|_{L_\omega^2 L^6} + \|\nabla_{\perp} \Psi_k\|_{L^\infty_{\omega} L^3}   \lesssim  c_k
\end{equation}

\begin{proof}
We first note that simply using the $L^2$ fixed time bound for $\nabla A$ does not suffice
due to the presence of two inverse derivatives in \eqref{defpsi}. We will use 
the Coulomb gauge condition to cancel one of these
two derivatives.  Precisely, using the Coulomb gauge condition to write
\[
A_{j,k} \omega_j = (\omega_j -  |\Delta|^{-1}  \nabla \otimes \nabla )    A_{j,k} =  
\Delta^{-1} \nabla \nabla_{\omega^\perp} A_{j,k}
\]
which is exactly what we need.

\end{proof}

Next, we consider a number of decomposable
estimates for the phase $\Psi(t,x;\xi)$ used to define our microlocal
gauge transformations:

\begin{lemma}[Decomposable estimates for $\Psi$]
  Let the phase $\Psi(t,x;\xi)$ be defined as in \eqref{defpsi}, and
  its angular components $\Psi^{(\theta)}=\Pi_\theta^\omega
  \psi(t,x;\xi)$, where $\omega=|\xi|^{-1}\xi$.  Then for $q \geq 2$
  and $2/q+3/r \leq \frac32$ one has:
  \begin{equation}
    \lp{(\Psi_k^{(\theta)},2^{-k}\nabla_{t,x}\Psi_k^{(\theta)})}{DL^q(L^r)} \ 
    \lesssim \ 2^{-(\frac{1}{q}+\frac{4}{r})k}\theta^{\frac{1}{2}-\frac{2}{q}-\frac{3}r}\epsilon
    , \label{psi_decomp2}
  \end{equation}
In addition, suppose that $\theta \lesssim  2^{j} \lesssim 1$.
Then for $q,r \geq 2$ we also have 
\begin{equation}
    \lp{Q_{k+2j} (\Psi_k^{(\theta)},2^{-k}\nabla_{t,x}\Psi_k^{(\theta)})}{DL^q(L^r)} \ 
    \lesssim \ 2^{-(\frac{1}{q}+\frac{4}{r})k} 2^{-(1-\frac{2}q)j}\theta^{\frac{1}{2}-\frac{3}r}\epsilon
    , \label{psi_decomp2-mod}
  \end{equation}
Further,
\begin{equation}
 \lp{\Box \Psi_k^{(\theta)}} {DL^2(L^\infty)} \
    \lesssim \ \theta^\frac32 2^{\frac{3}{2}k}, \label{boxpsi-disp}
\end{equation}

  In particular
  \begin{align}
    \lp{(\Psi_k,2^{-k}\nabla_{t,x}\Psi_k)} {DL^q(L^\infty)} \
    &\lesssim \ 2^{-\frac{1}{q}k}\epsilon
    \ ,  & q>4 \ , \label{psi_decomp1}\\
 \lp{Q_{k+2j}(\Psi_k,2^{-k}\nabla_{t,x}\Psi_k)} {DL^q(L^\infty)} \
    &\lesssim \ 2^{-\frac{1}{q}k}2^{(\frac12-\frac2q)j}\epsilon
    \ ,  & 2 \leq q <4 \ , \label{psi_decomp1-mod}\\
    \lp{\nabla_{t,x}\Psi_k} {DL^2(L^r)} \ &\lesssim \ 2^{(\frac12-
      \frac{4}{r} - \delta(\frac12+\frac3r))k}\epsilon \ ,  & r\geq 6 \
    , \label{psi_decomp3}
  \end{align}
\end{lemma}

\begin{proof}
  Notice that the last three estimates follow from the first by summing
  over dyadic $2^{-\delta k} \leq \theta\lesssim 1$. For the first
  two bounds we interchange the $t$ integration and the $\omega$ summation
  to obtain:
  \[
  \begin{split}
    \lp{(\psi_k^{(\theta)},2^{-k}\nabla_{t,x}\psi_k^{(\theta)})}{DL^q(L^r)}
    \ \lesssim \ & \theta^{-2}2^{-k} \big(\sum_{\omega}
    \lp{\Pi^\omega_\theta(D)A \cdot
      \omega}{L^q(L^r)}^2\big)^\frac{1}{2} \ \notag \\ \ \lesssim \ &
    \theta^{-1}2^{-k} \big(\sum_{\omega} \lp{\Pi^\omega_\theta(D)A
    }{L^q(L^r)}^2\big)^\frac{1}{2},
  \end{split}
  \]
  where at the second step we have used the Coulomb gauge to gain
  another factor of $\theta$.

  Now we conclude the proof of \eqref{psi_decomp1}
using the Strichartz estimate component of the $S_k$ norms.  In four space
  dimensions the Strichartz sharp range is given by
  $\frac{2}{q}+\frac{3}{r_0}=\frac{3}{2}$.  Moreover, on an angular
  sector of size $\theta$ Bernstein's inequality gives the embedding
  $\Pi^\omega_\theta(D)P_k L^{r_0}\subseteq \theta^{3(\frac{1}{r_0} -
    \frac{1}{r})} 2^{4(\frac{1}{r_0}-\frac{1}{r})k}L^{r}$. Thus:
  \begin{equation}
    \big(\sum_{\omega}
    \lp{\Pi^\omega_\theta(D)A_k}{L^q(L^r)}^2\big)^\frac{1}{2} \ \lesssim \
    \theta^{\frac{3}{2}-\frac{2}{q}-\frac{3}r}2^{(1-\frac{1}{q}-\frac{4}r)k}\lp{A_k}{S_k} \ , \notag
  \end{equation}
  and  \eqref{psi_decomp1}  follows.

The argument for \eqref{psi_decomp1-mod} is simpler. The case $q=r=2$ is 
immediate using the $L^2$ bound coming from the $X^{1,\frac12}_{\infty}$ component
of the $S^1$ norm, and the transition to larger $q,r$ is done using Bernstein's inequality. 
\end{proof}

We wrap this section up by proving some additional symbol type bounds
for the phases $\Psi$. These involve the variation over the physical
space variables:

\begin{lemma}[Additional symbol bounds for $\Psi$]
  Let $\psi$ be as above. Then one has:
  \begin{align}
    |\Psi_{<k}(t,x;\xi) -\Psi_{<k}(s,y;\xi)|
    &\lesssim   \epsilon \log( 1+2^k (|t-s| + |x-y|)), \label{symbol1} \\
    |\Psi(t,x;\xi) -\Psi(s,y;\xi)|
    &\lesssim   \epsilon \log ( 1+ |t-s| + |x-y|) \label{symbol2} \\
    |\partial_\xi^\alpha( \Psi(t,x;\xi) -\Psi(s,y;\xi))| &\lesssim \
    \epsilon \la (t-s,x-y)\ra^{|\alpha -\frac12| \sigma}, 1 \leqslant
    \alpha \leqslant \sigma^{-1} . \label{symbol3}
  \end{align}
\end{lemma}

\begin{proof}
  We decompose as before
  \[
  \psi_{<k}(t,x;\xi) = \sum_{j < k} \sum_{\theta > 2^{\sigma j}}
  \psi^{(\theta)}_j(t,x,\xi)
  \]
  For each fixed $\theta$ and $j$ we have by the definition of $\psi$
  and the Coulomb gauge condition
  \[
  |\psi^{(\theta)}_j(t,x,\xi)| \lesssim \theta^{-1} 2^{-j}
  \sup_{\omega} \|\Pi_\theta^\omega A_j\|_{L^\infty}
  \]
  Then by energy estimates for $A$ and Bernstein's inequality we
  obtain
  \begin{equation} \label{psij} |\psi^{(\theta)}_j(t,x,\xi)| \lesssim
    \theta^\frac12 \| A_j[0]\|_{H^1 \times L^2}, \qquad
    |\psi_j(t,x,\xi)| \lesssim \| A_j[0]\|_{H^1 \times L^2}
  \end{equation}
  A similar argument leads to
  \begin{equation}\label{psijx}
    |\partial_{t,x} 
    \psi^{(\theta)}_j(t,x,\xi)| \lesssim  2^j \theta^\frac12 \| A_j[0]\|_{H^1 \times L^2},
    \qquad 
    |\partial_{t,x}  \psi_j(t,x,\xi)| \lesssim    2^j  \| A_j[0]\|_{H^1 \times L^2}
  \end{equation}
  Differentiating with respect to $\xi$ yields $\theta^{-1}$ factors,
  \[
  |\partial_\xi^\alpha \psi^{(\theta)}_j(t,x,\xi)| \lesssim
  \theta^{\frac12-|\alpha|} \| A_j[0]\|_{H^1 \times L^2}, \qquad
  |\partial_{x,t}
  \partial_\xi^\alpha \psi^{(\theta)}_j(t,x,\xi)| \lesssim 2^j
  \theta^{\frac12-|\alpha|} \| A_j[0]\|_{H^1 \times L^2}.
  \]
  For the bound \eqref{symbol1} we use both \eqref{psij} and
  \eqref{psijx} to write for $j \leq k$
  \[
  |\psi_{<k}(t,x;\xi) -\psi_{<k}(s,y;\xi)| \lesssim 2^j(|t-s| + |x-y|)
  + |k-j|
  \]
  and then optimize the choice of $j$.

  The proof of \eqref{symbol3} is similar.

\end{proof}

\subsection{ Fixed time bounds for $O$}

Here we transfer the above bounds from $\Psi$ to $O$. 
Precisely, we have the following 

\begin{lemma}
The following  estimates hold for $O$, where $\perp$ below refers to derivatives in the plane $\omega^{\perp}$:
\begin{equation}\label{l2forO}
\|P_{k'}  O_{;\perp}\|_{L^2} \lesssim 2^{-k'} c_{k'} 2^{-N (k'-k)_+}
\end{equation}
\begin{equation}\label{l2forOa}
\|P_{k'}  O_{;x,t}\|_{L^2} \lesssim 2^{(-\delta-1)k'} c_{k'} 2^{-N (k'-k)_+}
\end{equation}
 
Estimates with one derivative less  hold for $\partial_k O_{;x,t}$.
\end{lemma}
\begin{proof}
We treat the case of spatial derivatives, time derivatives being handled similarly. 
Our strategy will be to use integration in $h$ and reiteration in the commutation relation 
\begin{equation}\label{com-oh}
\frac{d}{dh} O_{<h;x} = \Psi_{h,x} + [\Psi_h,O_{<h;x}]
\end{equation}
as well as differentiated forms of it, in order to build up successively stronger bounds 
for the derivatives of $O_{<h}$. In this section, mixed Lebesgue spaces $L^p L^q$ refer to the coordinates $\omega, \omega^{\perp}$ for the $x$-plane. 

\bigskip

{\em 1. $L^\infty$ bounds.}  A-priori we have 
\[
\| \Psi_{k}\|_{L^\infty} \lesssim c_k.
\]
Then integration from $-\infty$  with respect to $h$ in \eqref{com-oh} gives
\[
\|  O_{<k;x}\|_{L^\infty} \lesssim 2^{k} c_k
\]
Repeated differentiation similarly leads to a better high frequency bound
\[
\|\partial_x^{m-1} O_{<k;x}\|_{L^\infty} \lesssim 2^{mk} c_k 
\]

\bigskip

{\em 2. $L^2 L^{12}$ bounds.} Here we start with
\[
\| \Psi_k\|_{L^2 L^{12}} \lesssim 2^{-\frac{3k}4} c_k
\]
The same argument as above using \eqref{com-oh}
leads to 
\[
\|  O_{<k;x}\|_{ L^2 L^{12} } \lesssim 2^{\frac{k}4} c_k
\qquad
\|\partial_x^{m-1} O_{<k;x}\|_{L^2 L^{12}} \lesssim 2^{(m-\frac34)k} c_k 
\]
\bigskip

{\em 3. $L^\infty L^{6}$ bounds.} Here we start with
\[
\| \partial_{\perp} \Psi_k\|_{L^\infty L^{6}} \lesssim 2^{\frac12 k} c_k
\]
As above, using \eqref{com-oh} but only for $\partial_{\perp}$ derivatives
we obtain
\[
\|  O_{<k;\perp}\|_{ L^\infty L^{6} } \lesssim 2^{\frac{k}2} c_k
\qquad
\|\partial_x^{m-1} O_{<k;\perp}\|_{L^\infty L^{6}} \lesssim 2^{(m-\frac12)k} c_k 
\]
\bigskip

{\em 4. $L^2 L^6$ bounds.} For this we use the bound
\[
\| \Psi_k\|_{L^2 L^{6}} \lesssim 2^{-k} c_k
\]
We apply a Littlewood-Paley projector $P_{k'}$ in \eqref{com-oh}
and integrate in $h$, 
\[
\| P_{k'} O_{<k;x} \|_{L^2 L^6} \lesssim \int_{-\infty}^k \| P_{k'}  \Psi_{h,x}\|_{L^2 L^6}
 + \| P_{k'} [\Psi_h,O_{<h;x}]\|_{L^2 L^6} \, dh
\]
The first term on the right contributes only when $h = k+O(1)$. Thus we consider 
two scenarios. If $k < k'$ then we combine directly the high 
frequency $L^\infty$ bound for $O_{;x}$ with the $L^2 L^6$ bound for $\Psi_{h,x}$
to obtain the rapid decay
\[
\|  P_{k'} O_{<k;x} \|_{L^2 L^6} \lesssim c_{k'} 2^{-N(k'-k)}, \qquad k < k'
\]
If $k \geq k'$ then we retain the contribution of the first term when
$h = k'+O(1)$, and in addition we bound the second term for larger $h
> k'$ using Bernstein's inequality as follows:
\[
\|  P_{k'} [\Psi_h,O_{<h;x}]\|_{L^2 L^6} \lesssim 2^{\frac34 k'} \| \Psi_h\|_{L^2 L^{6}}
\|O_{<h;x}]\|_{L^2 L^{12}} \lesssim c_h^2 2^{\frac34 (k'-h)} 
\]
Taking advantage of the decay in $h$, we obtain the desired bound
\[
\|  P_{k'} O_{<k;x} \|_{L^2 L^6} \lesssim c_{k'}, \qquad k \geq k'.
\]

\bigskip

{\em 5. $L^\infty L^3$ bounds.}
For this we use the bound
\[
\| \partial_{\perp} \Psi_k\|_{L^\infty L^{3}} \lesssim 2^{-k} c_k
\]
and argue as in the $L^2 L^6$ case. The only difference arises in the 
treatment of the bilinear term for $h \geq k'$, namely
\[
\|  P_{k'} [\Psi_h,O_{<h;\perp}]\|_{L^\infty L^3} \lesssim 2^{\frac12 k'} \| \Psi_h\|_{L^2 L^{6}}
\|O_{<h;\perp}]\|_{L^\infty L^6} \lesssim c_h^2 2^{\frac12 (k'-h)} 
\]
We obtain
\[
\|  P_{k'} O_{<k;\perp} \|_{L^\infty L^3} \lesssim c_{k'} 2^{-N(k'-k)_+}
\]



\bigskip

{\em 6. $L^2$ bounds.}
In this final step we use the equation
\[
\frac{d}{dh} P_{k'}  O_{<h;\perp} = P_{k'} \Psi_{h,\perp}
+ P_{k'} [\Psi_{h},O_{;\perp}]  
\]
take $L^2$ norms and integrate with respect to $h$.
For $h < k'$ the first term on the right vanishes, while for the second
we have
\[
\| P_{k'} [\Psi_{h},O_{<h;\perp}] \|_{L^2} \lesssim \| \Psi_h\|_{L^2 L^6} 
\| P_{k'}O_{;\perp} \|_{L^\infty L^3} 
\lesssim c_h^2 2^{-h} 2^{-N(k'-h)}
\]
Integrating we obtain 
\[
\| P_{k'} O_{<k;\perp} \|_{L^2} \lesssim c_{k'} 2^{-k'}  2^{-N(k'-k)}, \qquad k < k'. 
\]

It remains to consider the case $k > k'$. The first term $ P_{k'}
\Psi_{h,\perp}$ is nonzero only if $h = k'+O(1)$, in which case it is
easily estimated using \eqref{psi-2-ft}. For the second term, on the
other hand, we have
\[
\| P_{k'} [\Psi_{h},O_{<h;\perp}] \|_{L^2} \lesssim \| \Psi_h\|_{L^2 L^6} \| O_{;\perp} \|_{L^\infty L^3} 
\lesssim c_h^2 2^{-h}
\]
which is easily integrated for $h > k'$. Thus the proof of \eqref{l2forO} is complete.

\bigskip

{\em 7. Proof of the bound \eqref{l2forOa}.} This proof is largely similar,
so we outline the change. In fact, in  Step 5, a $2^{-\frac12 \delta k}$ loss in the $\partial_x \psi$ bound generates a similar loss
for $O_{<k;x}$.
The same loss propagates directly to Step 6.

\subsection{ Fixed time bounds for $O_{;\xi}$}

Differentiating the functions $\Psi_k$ with respect to $\xi$ looses 
a factor of $\angle(\xi,\eta)$. Due to the angular separation, this 
factor is at most $2^{\delta k}$. Thus, the bounds for $O_{<k;\xi}$ 
are similarly related to the bounds for $O_{<k}$:

\begin{lemma}
We have the pointwise bounds
\begin{equation}
|\partial_\xi^n O_{<k;x}| \lesssim 2^{k(1-n\delta)},\,n\delta < 1,
\end{equation}
as well as the $L^2$ bounds
\begin{equation}\label{pdxio2}
\|P_{k'}\partial_\xi^n O_{<k;x}\|_{L^2} \lesssim 2^{k(-1-n\delta)} 2^{-N(k'-k)_+}
\end{equation}
\end{lemma}

The evolution equation for $O_{;\xi} = O_\xi O^{-1}$ is 
\[
\frac{d}{dh}O_{<h;\xi} = \Psi_{h,\xi} + [\Psi_h,O_{<h;\xi}]
\]
We have a similar relation for $O_x$,
\[
\frac{d}{dh} O_{<h;x} = \Psi_{h,x} + [\Psi_h,O_{<h;x}]
\]
Differentiating the latter with respect to $\xi$ yields
\[
\frac{d}{dh} \partial_{\xi} O_{<h;x} = \partial_\xi \Psi_{h,x}  + [\partial_\xi \Psi_h,O_{<h;x}]
+ [\Psi_h,\partial_\xi O_{<h;x}]
\]
Since 
\[
 |\Psi_{h,x}| + |O_{<h;x}|\lesssim 2^h, \qquad |\partial_\xi \Psi_{h,x}| \lesssim 2^{h(1-\delta)}
\]
we can integrate to obtain  
\[
|\partial_\xi O_{<h;x}| \lesssim  2^{h(1-\delta)}
\]
We further have $L^2$ bounds
\[
\| \partial_\xi \Psi_{h,x}\|_{L^2}  \lesssim  2^{h(-1-2\delta)}, \qquad
 \| \partial_\xi \Psi_{h}\|_{L^2}  \lesssim  2^{h(-2-2\delta)}, \qquad
\| \Psi_{h}\|_{L^2}  \lesssim  2^{h(-2-\delta)}
\]
with extra gain for further $x$ derivatives.
We can transfer these bounds to $\partial_\xi O_{<h;x}$ by using Littlewood-Paley
projectors in $x$ in the above evolution, to obtain \eqref{pdxio2}.

We also have the commutation relation
\begin{equation}\label{oxi:eq}
\partial_\xi O_{<h;x} - \partial_x O_{<h;\xi} = [ O_{<h;x},O_{<h;\xi}]
\end{equation}
Up to this point $O$ is only uniquely determined up to a $\xi$ dependent 
conjugation, 
\[
O_{<h}(x,\xi)  \to O_{<h}(x,\xi)P(\xi)
\]
At the level of $O_{<h;\xi}$ this translates to the gauge freedom
\[
 O_{<h;\xi}(x,\xi) \to O_{<h;\xi}(x,\xi) + O_{<h}(x,\xi) P_\xi P^{-1} O_{<h}^{-1}(x,\xi)
\]
Fixing a choice of $P$ is not necessary, as all estimates we need are invariant under such a change.

\subsection{Decomposable bounds for $O_{;x}$, $O_{;t}$}  

Our goal here is to transfer decomposability bounds from $\Psi$ to $O$.
Precisely, we have

\begin{lemma}\label{lem:Obounds1}
We have the following estimates:
 \begin{align}
    \lp{O_{<k;x},O_{<k;t}} {DL^q(L^\infty)} \
    &\lesssim \ 2^{(1-\frac{1}{q})k}\epsilon
    \ ,  & q>4 \ , \label{Ox_decomp1}\\
    \lp{O_{<k;x},O_{<k;t}} {DL^2(L^\infty)} \ &\lesssim \ 2^{\frac12(1-\delta) k}\epsilon \ 
    , \label{Ox_decomp3}
  \end{align}
\end{lemma}

\begin{proof}
  We prove the bounds for $O_{<k;x}$; those for $O_{<k;t}$ are
  identical. We use the evolution for $O_{k;x}$, namely 
\[
\partial_k O_{<k;x} = \Psi_{k,x} + [\Psi_k, O_{<k;x}]
\]
We proceed in several stages:
\medskip

\step{1}{A weaker $DL^\infty L^\infty$ bound}
Using the pointwise bounds on $O_{;x}$ and its $\xi$ derivatives, we directly
conclude that (for $\delta>0$ small enough)
\[
\| O_{<k;x}\|_{DL^\infty L^\infty} \lesssim 2^{(1-n\delta) k},\,n = 40. 
\]
\medskip

\step{2}{The full $DL^\infty L^\infty$ bound}
Using the above evolution we obtain the integral bound
\[
\| O_{<k;x}\|_{DL^\infty L^\infty} \lesssim \| O_{<l;x}\|_{DL^\infty L^\infty}+ 
\int_{l}^k  \| \Psi_{h,x}\|_{DL^\infty L^\infty} +  \| \Psi_{h}\|_{DL^\infty L^\infty} 
\|  O_{<h;x}\|_{DL^\infty L^\infty} dh
\]
By Gronwall's inequality this gives
\begin{equation}\label{gron(Ox)}
\| O_{<k;x}\|_{DL^\infty L^\infty} \lesssim \| O_{<l;x}\|_{DL^\infty L^\infty}
e^{\int_l^k  c_h  dh} + \int_{l}^k 2^{h} c_h e^{\int_{h}^k c_{h_1} d h_1} dh
\end{equation}
But by Cauchy-Schwarz we have
\[
\int_l^k  c_h  dh \lesssim |k-l|^\frac12
\]
Thus, using the weaker $DL^\infty L^\infty$ bound, the  first term in \eqref{gron(Ox)}
decays to zero as $l \to -\infty$. On the other hand, the leading contribution 
in the second term in \eqref{gron(Ox)} comes from $h = k-O(1)$. 
Hence we obtain the desired bound.
\[
\| O_{<k;x}\|_{DL^\infty L^\infty} \lesssim 2^k c_k
\]
\medskip

\step{3}{The $DL^q L^\infty$ bound} Using again the above evolution
and the fact that, by construction, $\lim_{k \to 0} O_{<k;x} = 0$ we
write
\[
O_{<k;x} = \int_{-\infty}^k  \Psi_{h,x} + [\Psi_h, O_{<h;x}] dh
\]
Then we combine the $DL^q L^\infty$ decomposability  bound \eqref{psi_decomp1} for   
$\Psi_h$ with the previously established $DL^\infty L^\infty$ bound for $O_{<h;x}$.

\medskip

\step{4}{The $DL^2 L^\infty$ bound} We proceed as in the previous
step, but using the bound \eqref{psi_decomp1} for $\Psi_h$ instead.

\end{proof}

\subsection{ Difference bounds for $O$}

Here we seek to compare $O_{<k}(t,x,\xi) $ with $O_{<k}(s,y,\xi)$. Since both are 
elements of the Lie group $\G$, it is natural (and most useful in the sequel)
to look at the product $O_{<k}(t,x,\xi) O^{-1}_{<k}(s,y,\xi)$. We have

\begin{lemma}[Difference bounds for $O$]
  Let $O$ be as above. Then one has:
  \begin{align}
    d(O_{<k}(t,x,\xi) O^{-1}_{<k}(s,y,\xi), Id)
    &\lesssim   \epsilon \log( 1+2^k (|t-s| + |x-y|)), \label{symbol1O} \\
    d(O(t,x,\xi) O^{-1}(s,y,\xi),Id)
    &\lesssim   \epsilon \log ( 1+ |t-s| + |x-y|) \label{symbol2O} \\
  |\partial_\xi^n (O(t,x,\xi) O^{-1}(s,y,\xi))_{;\xi}| & \lesssim  \langle (t-s,x-y) \rangle^{n \delta}
. \label{symbol3O}
  \end{align}
\end{lemma}
\begin{proof}
For the first two bounds we use the $Ad(O^{-1}(t,x,\xi))$ to interchange the order
and estimate instead the distance $ d( O^{-1}_{<k}(s,y,\xi) O_{<k}(t,x,\xi), Id)$.
This vanishes as $k \to -\infty$, therefore we can write
\[
d( O^{-1}_{<k}(s,y,\xi) O_{<k}(t,x,\xi), Id) \lesssim \int_{-\infty}^k 
|(O^{-1}_{<h}(s,y,\xi) O_{<h}(t,x,\xi))_{;h}| dh
\]
But we have
\[
 (O^{-1}_{<h}(s,y,\xi) O_{<h}(t,x,\xi))_{;h} = O^{-1}_{<h}(t,x,\xi) 
(\Psi_h(t,x,\xi) - \Psi_h(s,y,\xi)) O_{<h}(s,y,\xi) 
\]
so we obtain
\[
d( O^{-1}_{<k}(s,y,\xi) O_{<k}(t,x,\xi), Id) \lesssim \int_{-\infty}^k
| \Psi_h(t,x,\xi) - \Psi_h(s,y,\xi)| dh
\]
For $\Psi_h$ we have the bound 
\[
| \Psi_h(t,x,\xi) - \Psi_h(s,y,\xi)| \lesssim \epsilon \min\{1, 2^h(|x-y|+|t-s|)\}
\]
Thus the bounds \eqref{symbol1O} and \eqref{symbol2O} follow after 
dyadic integration with respect to $h$.

\bigskip

 For the third bound \eqref{symbol3O} we denote 
$V_{<k} = O^{-1}_{<k}(t,x,\xi) O_{<k}(s,y,\xi)$,  and proceed in two steps. For the first step
we fix $k$, and show that
\begin{equation}\label{get-xiO}
  |\partial_\xi^n V_{<k;\xi}|  \lesssim  (|t-s|+|x-y|) 2^{k(1- n \delta)}, \qquad 2^k |x-y| \lesssim 1, 
\quad n \geq 0.
\end{equation}
This bound is favorable provided that $k$ is small enough.
In the second step, we extent the range of $k$ for which \eqref{symbol3O} holds
by evaluating the $k$ derivative of $\partial_\xi^n V_{<k;\xi}$.

We now proceed with the first step, where we will crucially use the bound
\begin{equation}\label{use-xiO}
|\partial_\xi^{n} O_{<k;x} (t,x,\xi)|  \lesssim  2^{k(1- \delta(\frac12+n))},
\end{equation}
see \eqref{pdxio2}.  
The expression $V_{<k;\xi}$ vanishes if $x=y, t = s$ so it suffices to estimate its 
$x, t$ derivatives; below we do so for the $x$-derivatives, with similar estimates applying to the $t$-derivatives: 
\[
\partial_x \partial_\xi^{n} V_{<k;\xi} =  \partial_\xi^{n+1} V_{<k;x} +   \partial_\xi^{n} [V_{<k;x},
 V_{<k;\xi}]
\]
for which we use $V_{<k;x}= O_{<k;x}(t,x,\xi)$ to  rewrite it as 
\[
\partial_y \partial_\xi^{n} V_{<k;\xi} -  [V_{<k;y}, \partial_\xi^{n} V_{<k;\xi}]
=  \partial_\xi^{n+1} O_{<k;x}  + \sum_{j=1}^{n}[\partial_\xi^j O_{<k;x},
  \partial_\xi^{n-j} V_{<k;\xi}].
\]
The last term is absent if $n=0$, so the bound \eqref{get-xiO} follows
directly from \eqref{use-xiO} by integration.  Finally we close by
induction integrating over $x$, estimating
\[
|[\partial_\xi^j O_{<k;x}, \partial_\xi^{n-j} V_{<k;\xi}]| \lesssim 
 2^{k(1+ \delta(\frac12-j))}  |x-y| 2^{k(1 - \delta(\frac12+(n-j)))} 
= 2^{k}|x-y| 2^{k(1 - \delta n)} 
\]
So far the bound \eqref{symbol3O} is established in the range $ 2^{k}|x-y| \lesssim 1$.
To extend it we forget about the distance between $x$ and $y$ and integrate 
instead with respect to $l$ (the new $k$). First write
\[
V_{<l} = W_{<l}(x) V_{<k} W^{-1}_{<l}(y)
\]
where 
\[
 W_{<l} = O_{<l}^{-1} O_{<k} 
\]
We have 
\[
V_{<l;\xi} = - W_{<l}(x)  V_{<k}^{-1}  W_{<l}^{-1}(y)W_{<l; \xi }(y) W_{<l}(y) V_{<k} W^{-1}_{<l}(x) + W_{<l}(x) V_{<k;\xi} W^{-1}_{<l}(y) +
W_{<l;\xi} (x)
\]
so repeated differentiation shows that it suffices 
to bound 
\begin{equation}\label{W_l}
|\partial_\xi^{n} W_{<l; \xi }(x)| \lesssim   2^{- k \delta(\frac12+n)}, \qquad  l > k, \ \  n \geq 0
\end{equation}
 For this we follow the previous strategy, writing
\[
\partial_l \partial_\xi^{n} W_{<l;\xi} =  \partial_\xi^{n+1} O_{<l;l} +   \partial_\xi^{n} [O_{<l;l},
 W_{<l;\xi}]
\]
which leads to
\[
\partial_l \partial_\xi^{n} W_{<l;\xi} -  [\Psi_{l}, \partial_\xi^{n} W_{<l;\xi}]
=  \partial_\xi^{n+1} \Psi_l +   \sum_{j=1}^{n}[\partial_\xi^j \Psi_l,
  \partial_\xi^{n-j} W_{<l;\xi}]
\]
Using the bounds for $\Psi$ we can inductively close \eqref{W_l}.

\end{proof}

\section{$L^2$ bounds for the parametrix}
\label{s:l2}
In this section we establish a number of $L^2$ bounds for the renormalization operators
and the parametrix.    In the last part we prove the bounds \eqref{N_mapping},
\eqref{dtS_mapping}, \eqref{S_mapping} and \eqref{parametrix_bound}.  Throughout the section we
assume that $A$ is a Yang-Mills wave with $\| A\|_S \ll 1$ and
frequency envelope $c_k$. We fix the $\pm$ sign to $+$ and drop it from the 
notations. Also, we shall consider unit frequencies, and put $O$ instead of $O_{<0}$. We split the argument across several subsections.


\subsection{Oscillatory integral estimates}

We first observe that on one hand our parametrix involves operators of the form 
\[
T^a = Op(Ad(O^{\pm}))(t,x,D) e^{\pm i(t-s) |D|} a(|D|) Op(Ad(O^\pm))(D,s,y)
\]
where $a$ is localized at frequency $1$.
On the other hand, arguing in $TT^*$ fashion in order to prove various
$L^2$ estimates involving the operators $Op(Ad(O(t,x,D))$ and
$Op(Ad(O_{<0}(t,x,D)^*))$, we need to consider bounds for similar operators
in the special case when $t = s$.

The kernel of the operator $T_a$ is given by the oscillatory integral
\[
K^aF(t,x) = \int a(\xi) e^{\pm i(t-s)|\xi|} e^{i \xi(x-y)}(O(t,x,\xi)
O^{-1}(s,y,\xi)) F(s,y) (O(t,x,\xi) O^{-1}(s,y,\xi))^{-1} d\xi
\]
Our main estimates for such kernels are as follows:

\begin{proposition}\label{kernel_prop}
  a) Assume that $a$ is a smooth bump on the unit scale. Then the
  kernel $K_a$ satisfies
  \begin{equation} \label{wave-decay} |K_a(t,x;s,y)| \lesssim \langle
    t-s \rangle^{-\frac32} \langle |t-s| - |x-y| \rangle^{-N}
  \end{equation}

  b) Let $a = a_{C}$ be a bump function on a rectangular region $C$ of
  size $2^{k} \times (2^{k+l})^3$ with $k \leq l \leq 0$. Then
  \begin{equation}\label{wave-decay-cube}
    |K_a(t,x;s,y)| \lesssim  2^{4k+3l} \langle 2^{2(k+l)} (t-s) \rangle^{-\frac32}
    \langle 2^k (|t-s| - |x-y|) \rangle^{-N}
  \end{equation}
  If in addition $x-y$ and $C$ have a $2^{k+l}$ angular separation
  then
  \begin{equation}\label{off-angle}
    |K_a(t,x;s,y)| \lesssim  2^{4k+3l} \langle 2^{2(k+l)} |t-s| \rangle^{-N}
    \langle 2^k (|t-s| - |x-y|) \rangle^{-N}
  \end{equation}

\end{proposition}

\begin{proof}
  a) Away from a conic neighborhood of the cone $\{ |t-s| =
  \pm|x-y|\}$ the phase
  \[
  \Psi = \pm (t-s)|\xi| + \xi(x-y)  \]
  is nondegenerate.  Hence applying the symbol bounds \eqref{symbol3} 
  repeated integration by parts with respect to $\xi$ yields
  \[
  | K^a(t,x,s,y) | \lesssim \la(t,x)-(s,y)\ra^{-N}, \qquad N \sim
  \sigma^{-1}
  \]
  Near the cone we need to be more careful. Denoting $T = |t-s|+|x-y|$
  and $R= |t-s|-|x-y|$, in suitable (polar) coordinates the operator $K^a$ takes the
  form
  \[
  K^aF(t,x) = \int (O^{-1}(t,x,\xi') O(s,y,\xi')) F(s,y)  (O^{-1}(t,x,\xi') O(s,y,\xi'))^{-1} e^{i R \xi_1} e^{i T \xi'^2} \tilde a(\xi)
  d\xi
  \]
  In $\xi_1$ (the former radial variable) this is a straight Fourier
  transform, so we get rapid decay in $R$.  
Given the bound \eqref{symbol3}, we can use stationary
  phase in $\xi'$.  While the $\xi$ derivatives of the $O^{-1}(t,x,\xi') O(s,y,\xi')$ part
  of the phase are not bounded, they only bring factors of $T^\sigma$,
  which is small enough not to affect the stationary phase ( this
  works up to $\sigma = \frac12$).  We obtain
  \[
  |K^a(t,x,s,y)| \lesssim T^{-\frac32} (1+ R)^{-N}
  \]

  b) Away from the cone the estimate follows easily as above since the
  phase is nondegenerate.  Near the cone we use again polar
  coordinates to express our oscillatory integral as above,
  \[
  K^CF(t,x) = \int (O^{-1}(t,x,\xi') O(s,y,\xi')) F(s,y)  (O^{-1}(t,x,\xi') O(s,y,\xi'))^{-1}e^{i R \xi_1} e^{i T \xi'^2} \tilde a_C(\xi)
  d\xi
  \]
  where $a_C$ is a bump function in a rectangle on the $2^k$ scale in
  the radial variable $\xi_1$ and on the $2^{k+l}$ scale in the
  angular variable $\xi'$.  Then we can separate variables in
  $(\xi_1,\xi')$. We note that this rectangle need not be centered at
  $\xi'=0$, though this is the worst case.  In $\xi_1$ this is again a
  Fourier transform, so we get the factor
  \[
  2^k \langle 2^k R \rangle^{-N}
  \]
  In $ \xi'$ we can use stationary phase to get the factor
  \[
  2^{3(k+l)} \langle 2^{2(k+l)T} \rangle^{-\frac32}
  \]
  The bound \eqref{wave-decay-cube} follows by multiplying these two
  factors.

  Finally, the estimate \eqref{off-angle} corresponds to the case when
  $a_C$ is supported in $|\xi'| > 2^l$ in the above representation.
  If $T < 2^{-2(k+l)}$ then there are no oscillations in $\xi'$ on the
  $2^{k+l}$ scale, and we just use the brute force estimate. For $T >
  2^{-2(k+l)}$ the phase is nonstationary in $\xi'$, and we obtain the
  factor
  \[
  2^{3(k+l)} (1+2^{2(k+l)} T)^{-N}
  \]
\end{proof}

While the above proposition contains all the oscillatory integral
estimates which are needed, it does not apply directly to the
frequency localized operators $Op(Ad(O))_{<0}(t,x,D)$ and
$Op(Ad(O))_{<0}(D,y,s)$.  For that we need to produce similar
estimates for the kernels $K_{a,<0}$ of the operators
\[
T^a_{<0} = Op(Ad(O))_{<0}(t,x,D) a(D) e^{\pm i(t-s)|D|}Op(Ad(O^{-1}))_{<0}(D,s,y) 
\]
The transition to such operators is made in the next
\begin{proposition}\label{kernel_prop_loc}
  a) Assume that $a$ is a smooth bump on the unit scale. Then the
  kernel $K^a_{<0}$ satisfies
  \begin{equation} \label{wave-decay0} |K^a_{<0}(t,x;s,y)| \lesssim
    \langle t-s \rangle^{-\frac32} \langle |t-s| - |x-y| \rangle^{-N}
  \end{equation}
  In addition, the following fixed time bound holds:
  \begin{equation} \label{almost-ortho} |K^a_{<0}(t,x;t,y)- \check
    a(x-y)| \leq \epsilon |\log \epsilon|
  \end{equation}

  b) Let $a = a_{C}$ be a bump function on a rectangular region $C$ of
  size $2^{k} \times (2^{k+l})^3$ with $k \leq l \leq 0$. Then
  \begin{equation}\label{wave-decay-cube0}
    |K^a_{<0} (t,x;s,y)| \lesssim  2^{4k+3l} \langle 2^{2(k+l)} (t-s) \rangle^{-\frac32}
    \langle 2^k (|t-s| - |x-y|) \rangle^{-N}
  \end{equation}

  c) Let $a = a_{C}$ be a bump function on a rectangular region $C$ of
  size $1 \times (2^{l})^3$ with $ l \leq 0$.  Let $\omega \in \S^3$
  be at angle $l$ from $C$. Then we have the characteristic kernel
  bound
  \begin{equation}\label{off-angle0}
    \begin{split}
      |K^a_{<0}(t,x;s,y)| \lesssim 2^{3l} \langle 2^{2l} |t-s|
      \rangle^{-N} & \langle 2^l |x'-y'| \rangle^{-N}
      \\
      t-s = (x-y)\cdot \omega
    \end{split}
  \end{equation}

\end{proposition}
\begin{proof}
  a) We represent the action of symbol $Op(O)_{<0}$ by
  \begin{equation}
   Op(O)_{<0}F(x) = \int m(z)\int e^{i(x-y)\xi} O(x+z,\xi) F(y) 
 O^{-1}(x+z,\xi) dy d\xi
\,dz
  \end{equation}
  where $m(z)$ is an integrable bump function on the unit scale.
 One proceeds similarly for functions on space-time. 

 This can be expressed in a concise form using the operators $T_z,
 T_w$ to represent translation in the space-time directions $z, w$
 acting on the variables $t, x$. and $s, y$, respectively.

 Using this representation for both operators $Op(O)_{<0}$, $Op(O)_{<0}^*$, and
 denoting $a(z,w) (\xi)= a(\xi) e^{i(\pm |\xi|, \xi) \cdot (z-w)}$,
he kernel $K^a_{<0}$ can be  expressed 
  in terms of the kernels $K^a$ in the
  previous proposition, namely
  \begin{equation}\label{ka0}
    K^a_{<0}F(t,x)=  \int T_z T_w  K^{a(z,w)}F(t,x)   \ m(z) m(w) \ dz dw
  \end{equation}

  To prove the bound \eqref{wave-decay0} we use \eqref{wave-decay},
  together with the additional observation that the implicit constant
  in \eqref{wave-decay} depends on finitely many seminorms of $a$ (at
  most 8, to be precise) which we denote by $||| a|||$. Then
  \[
  ||| a(z,w)||| \lesssim (1+|z|+|w|)^N
  \]
  However, this growth is compensated by the rapid decay of $m$,
  therefore the bound \eqref{wave-decay} for $K^a$ transfers directly
  to $K^a_{<0}$ in \eqref{wave-decay0}.

  To prove \eqref{almost-ortho} we use the same representation as
  above to write
  \[
  K^a_{<0}F(t,x) - \check a*F(t, x) = 
  \int (T_z T_w  K^{a(z,w)} - I)F(t,x)   \ m(z) m(w) \ dz dw
  \]
  By \eqref{symbol2} we have
  \[
  | T_z \psi_{\pm}(t,x,\xi) - T_{w} \psi_{\pm}(t,y,\xi))| \lesssim
  \epsilon \log (1+ |z|+|w|+|x-y|)
  \]
  which yields
  \[
  \begin{split}
    |K^a_{<0}(t,x,t,y) - \check a(x-y) | \lesssim &\ \epsilon \int
    \log (1+ |z|+|w|+|x-y|) |m(z)||m(w)| dz dw \\ \lesssim &\ \epsilon
    \log (2+|x-y|)
  \end{split}
  \]
  This suffices if $\log (2 + |x-y|) \lesssim |\log \epsilon|$. But
  for larger $|x-y|$ we can use \eqref{wave-decay0} directly.

  b) Using the representation \eqref{ka0}, the bound
  \eqref{wave-decay-cube0} follows from \eqref{wave-decay-cube}
  exactly by the same argument as in case (a).

  c) Using the representation \eqref{ka0}, the same argument also
  yields the bound \eqref{wave-decay} provided we have the following
  estimate for $K^a$:
  \[
  |K^a(t,x,s,y)| \lesssim 2^{3l} \langle 2^{2l} |t-s| \rangle^{-N}
  \langle 2^{l} |x'-y'| \rangle^{-N} (1+|(t-s)- (x-y) \cdot
  \omega|)^{10 N}
  \]
  To see that this is true, we consider three cases:

  (i) If $|t-s| \lesssim 2^{-2l}$ then \eqref{wave-decay-cube} applies
  directly.

  (ii) If $|t-s| \gg 2^{-2l}$ but $ ||x-y|-|t-s|| \gtrsim 2^{l}
  |x'-y'|+ 2^{2l}|t-s|$ then \eqref{wave-decay-cube} still suffices.

  (iii) If $|t-s| \gg 2^{-2l}$ and $|(t-s)- (x-y) \cdot \omega|)|
  \gtrsim 2^{l} |x'-y'|+ 2^{2l}|t-s| $ then \eqref{wave-decay-cube}
  also applies.

  (iv) Finally, if $|t-s| \gg 2^{-2l}$, but $ ||x-y|-|t-s|| \ll 2^{l}
  |x'-y'|+ 2^{2l}|t-s| $ and $|(t-s)- (x-y) \cdot \omega|)| \ll 2^{l}
  |x'-y'|+ 2^{2l}|t-s| $ then we must have $\angle (x-y,\omega) \ll
  2^{l}$, which implies that $\angle (x-y,C) \approx 2^{l}$.  Then
  \eqref{off-angle} applies.

\end{proof}

\subsubsection{Fixed-time $L^2$ estimates   for the 
 gauge transformations}\label{s:l2-gauge}

Here we use the previous theorem to prove three $L^2$ estimates which
correspond to the $L^2$-part of \eqref{N_mapping}, \eqref{dtS_mapping}
as well as that of \eqref{S_mapping}. These will also be repeatedly used
later in conjunction with the notion of disposability.

\begin{proposition}
  The following fixed time $L^2$ estimates hold for functions
  localized at frequency $1$, with or without the $<0$ symbol localization:
  \begin{align}
    Op(Ad(O))_{<0}(t,x,D):\quad & L^2 \rightarrow L^2,\, \label{l2}\\
    Op(Ad(O))_{<0}(t,x,D)  a(D) Op(Ad(O^{-1}))_{<0}(D,y,s) - a(D) :
    \quad & L^2\rightarrow \epsilon^{\frac{N-4}{N}} \log \epsilon \
    L^2
    \label{l2-ortho}\\
    \partial_{x,t} Op(Ad(O))_{<0}(t,x,D) : \quad & L^2 \to
    \epsilon L^2 \label{dtl2}
  \end{align}

\end{proposition}
\begin{proof}
  a) By the estimate \eqref{wave-decay} with $s=t$, the $TT^*$ type
  operator
  \[
  Op(Ad(O))(t,x,D) P_0^2Op(Ad(O^{-1}))(D,y,t)
  \]
  has an integrable kernel, so it is $L^2$ bounded.  Therefore $ Op(Ad(O))(t,x,D) P_0 $ 
and its adjoint are $L^2$ bounded. To   accommodate symbol localizations we observe that
  \[
Op(Ad(O))_{<k}   = \int m_k(z) Op(Ad(T_z O)) 
  \,dz
  \]
  where $m(z)$ is an integrable bump function on the $2^{-k}$ scale
  and $T_z$ denotes translation in the direction $z$, with $z$
  representing space-time coordinates. Since the wave equation is
  invariant to translations, the symbol $T_z O $ is of
  the same type as $O$ and its left and right
  quantizations are also $L^2$ bounded.  Thus the bound \eqref{l2}
  follows by integration with respect to $z$.

  b) For the estimate \eqref{l2-ortho} we note that the kernel of
  \[
Op(Ad(O))_{<0}(t,x,D) a(D) Op(Ad(O^{-1}))_{<0}(D,y,s) - a(D) 
\]
 is  given by $K^a_{<0}(t,x,t,y) - \check a(x-y)$. Combining
  \eqref{wave-decay} and \eqref{almost-ortho} we get
  \[
  |K^a_{<0}(t,x,t,y) - \check a(x-y)| \lesssim \min \{ \epsilon |\log
  \epsilon|, |x-y|^{-N}\}
  \]
  The integral of the expression on the right is about
  $\epsilon^{\frac{N-4}{N}} |\log \epsilon|$, therefore the conclusion
  follows.

  c) By translation invariance we discard the $<0$ symbol
  localization, and show that $\partial_{x,t} Op(Ad(O))(t,x,D) P_0$ is $L^2$ bounded.  
We have
  \[
\partial_{x} Ad(O) = ad(O_{;x}) Ad(O)
\]
 By \eqref{Ox_decomp1} we have
  $O_{;x} \in \epsilon DL^\infty(L^\infty)$ therefore we can
  dispose of it and use the $L^2$ boundedness of $  Op(Ad(O)) P_0$.
\end{proof}

\subsection{ High space-time frequencies in $O$}

Although $\Psi_{<k}$ is localized at space-time frequencies $<k$, its renormalization 
counterpart $O_{<k}$ does not share the same property  since it is obtained 
in a nonlinear fashion. Nevertheless,  the following result asserts that the
 high frequency part of $O_{<k}$ does satisfy much better bounds: 

\begin{lemma}
  Assume that $1\leqslant q \leqslant p \leqslant \infty$. Then for
  $k+C \leq l \leq 0$ we have :
  \begin{equation}
    \lp{Op(Ad(O_{<k}))_l(t,x;D)}{L^{p}(L^2)\to L^q(L^2)} 
    \lesssim\ 
    \epsilon 2^{(\frac{1}{p}-\frac{1}{q})k}2^{5(k-l)}
    \ , \label{remainder_est}
  \end{equation}
  This holds for both left and right quantizations.
\end{lemma}
\begin{proof}

  For the symbol we iteratively write:
  \begin{align}
    S_{l} Ad(O_{<k})  \ &= \  2^{-l}S_{l} \partial_{x,t} (Ad(O)_{<k}) = 
2^{-l}S_{l} (ad(O_{;(x,t)} Ad(O)_{<k})
    \ \notag\\
    &= \ \ldots \ = \  2^{-5l}\prod_{j=1}^5
    (S^{(j)}_{l} ad(O_{;(x,t)})  \cdot Ad(O)_{<k} \ ,
    \notag
  \end{align}
  where the product denotes a nested (repeated) application of
  multiplication by $S_{l}\partial_t \psi_{<k}$, for a series of
  frequency cutoffs $S^{(j+1)}_{l}S^{(j)}_{l}=S^{(j)}_{l} \approx
  S_{l}$ with expanding widths.  Disposing of these translation
  invariant cutoffs we see that \eqref{remainder_est} follows directly
  from \eqref{Ox_decomp1}.

\end{proof}


\subsection{Modulation localized estimates}

Our next goal is to show that the fixed time $L^2$ bounds
for $Op(O)$ drastically improve to space-times $L^2(L^2)$
bounds if one selects a fixed ``frequency'' in the symbol.
Precisely, for $k < 0$ we can express the difference
\[
Ad(O_{<0}) - Ad(O_{<k}) = \int_k^0  ad(\Psi_h) Ad(O_{<h}) dh
\]
where the integrand $Ad(O)_{;h}:= ad(\Psi_h) Ad(O_{<h})$, while not exactly 
localized at frequency $2^h$, nevertheless is better behaved 
both at higher and at lower frequencies. The next result asserts that 
the output of $Op(Ad(O)_{;h})(t,x,D)$ is better behaved at modulations less than $2^h$: 

 \begin{proposition}\label{mod_loc_prop}
  For $l\leqslant k'\pm O(1)$ one has the fixed frequency estimate:
  \begin{equation}
    \lp{Q_{l} Op(Ad(O)_{;k'})   Q_{<0}P_0}
    {N^*\to X^{0,\frac{1}{2}}_1} \ \lesssim \  2^{\delta(l-k')}\epsilon
    \ . \label{kk'_mod_loc}
  \end{equation}
  In particular summing over all $(l,k')$ with $l\leqslant k$ and
  $k-O(1)\leqslant k'$ for a fixed $k\leqslant 0$ yields:
  \begin{equation}
    \lp{Q_{<k} (Op({Ad}(O_{<0})) - Op(Ad(O_{<k-C}))Q_{<0}P_0}
    {N^*\to X^{0,\frac{1}{2}}_1} \ \lesssim \ \epsilon \ . \label{mod_loc}
  \end{equation}
\end{proposition}

\begin{proof}[Proof of Proposition \ref{mod_loc_prop}]
  We proceed in a series of steps, where we consider successive modulation scenarios.

\medskip 

  \step{1}{High modulation input} First we estimate the contribution
  of the dyadic piece $Q_{k}Op({Ad}(O)_{;k'})Q_{\geqslant k-C}P_0$ to line
  \eqref{kk'_mod_loc}. Using the $X_\infty^{0,\frac{1}{2}}$ bounds for
  the input, it suffices to prove the estimate:
  \begin{equation}
    \lp{Q_k Op({Ad}(O)_{;k'})P_0}{L^2(L^2)\to L^2(L^2)} \ \lesssim \ 
    2^{\frac{1}{5}(k-k')}\epsilon \ . \notag
  \end{equation}
  By Sobolev estimates in $|\tau|\pm |\xi|$, this reduces to the
  bound:
  \begin{equation}
    \lp{Op(Ad(O)_{;k'})P_0}{L^2(L^2)\to L^\frac{10}{7}(L^2)} \ \lesssim \ 
    2^{-\frac{1}{5}k'}\epsilon \ . \notag
  \end{equation}
  Recalling that $Op(Ad(O)_{;k'})$ has symbol $ad(\Psi_{k'}) Ad(O_{<k'})$,
it suffices to use the $L^2$ boundedness for $Op(O_{<k'})$ 
and the $L^5 L^\infty$ disposability bound for $\Psi_{k'}$.
\medskip

  \step{2}{Main decomposition for low modulation input} Now we
  estimate the expression $Q_{k}Op(Ad(O)_{;k'}) Q_{< k-C}P_0u$.
  First expand the untruncated group elements as follows:
  \begin{align}
  Ad(O)_{;k'}  = & \   ad(\Psi_{k'}) Ad(O_{<k-C}) +  \int_{k-C}^{k'}
    ad(\Psi_{k'}) ad(\Psi_l) Ad(O_{<k-C})  dl\\  & \ +
 \int_{k-C'}^{k'}\int_{l'}^{k'} ad(\Psi_{k'})ad(\Psi_l)ad(\Psi_{l'}) Ad(O_{<l'})dl dl'\notag\\
    = & \ \mathcal{L} + \mathcal{Q} + \mathcal{C} 
    . \notag
  \end{align}
  We will estimate the effect of each of these terms separately.

 \medskip

 \step{3}{Estimating the linear term $\mathcal{L}$} The factor $
 ad(\Psi_{k'}) $ in $\mathcal{L} $ is well localized both in frequency
 and modulation. While not exactly localized, the second 
factor $Ad(O_{<k-C})$ is to the leading order localized at frequency 
and modulation $\leq k - C/2$, with more regular and decaying tails 
at larger frequencies and modulations. The geometry of the 
bilinear wave interactions, on the other hand, requires us to 
estimate differently the contribution of $ ad(\Psi_{k'}) $ depending on 
its modulation relative to $2^k$.  To account for both considerations
above, we split the term $\mathcal L$ as follows:
  \begin{equation}
    \mathcal{L} \ =  ad(\Psi_{k'}) S_{<k-4} Ad(O_{<k-C})
+  ad(\psi_{k'}) S_{>k-4} Ad(O_{<k-C})    
\label{Lsplit}
  \end{equation}
\medskip

  \step{3a}{Estimating the principal linear term in $\mathcal{L}$} For
  the first term on RHS of line \eqref{Lsplit} it suffices to show the
  general estimate:
  \begin{multline}
  \hspace{-.1in}  \lp{Q_{k}Op( ad(\Psi_{k'}) b_{<k-4})
      Q_{<k-C}P_0 }{L^\infty(L^2)\to L^2(L^2)}  \lesssim 
    \epsilon\, 2^{-\frac{1}{2}k+\frac{1}{4}(k-k')} \sup_t
    \lp{B_{<k-4}(t)}{L^2\to L^2}  \label{lin_psi_loc}
  \end{multline}
  for $k' \geqslant k$, and for symbols $b(x,\xi)_{<k-4}$
  with sharp frequency and modulation localization and with either the left
  or right quantization.  The geometry of the 
bilinear wave interactions requires us to 
estimate differently the contribution of $ ad(\Psi_{k'}) $ depending on 
its modulation relative to $2^k$. Thus we will consider three cases:
\medskip

\step{3a(i)}{ The contribution of $ Q_{<k} \Psi_{k'}$}
In this case the modulation of the output
  determines the angle $\theta$ between the spatial frequencies of
  $\Psi_{k'}(x,\xi)$ and the spatial frequency of the input, which is
  $\theta\sim 2^{\frac{1}{2}(k-k')}$. Since this is also the angle with
  $\xi$, we may restrict the symbol of $\Psi_{k'}$ to $\psi_{k'}^{(\theta)}$ for
  which the estimate \eqref{lin_psi_loc} follows immediately from
  \eqref{decomp1} and summing over \eqref{psi_decomp2}.
 \medskip

\step{3a(ii)}{ The contribution of $ Q_{k} \Psi_{k'}$}  In this case one of the 
inputs has the same modulation as the output, so we only get a bound 
from above on the angle $\theta$, namely  $\theta \lesssim 2^{\frac{1}{2}(k-k')}$.
However, instead of \eqref{psi_decomp2}, which looses at small angles,
we can take advantage of the fixed modulation to  use \eqref{psi_decomp2-mod}, 
which gains at small angles.

\medskip
\step{3a(iii)}{ The contribution of $ Q_{>k} \Psi_{k'}$} In this case
one of the inputs has high modulation, say $2^{k'+2j'}$ with $(k-k')/2
< j' \leq 0$ .  This determines the angle $\theta$ to be $\theta
\approx k'+j'$. Then we can use again \eqref{psi_decomp2-mod}.

\medskip

  \step{3b}{Estimating the frequency truncation error in
    $\mathcal{L}$} For the second term on RHS of line \eqref{Lsplit}
  we use \eqref{psi_decomp1} for $\psi_{k'}$ with $p=6$ combined with
  \eqref{remainder_est} with $(p_2,q)=(6,3)$.

  \bigskip

  \step{4}{Estimating the quadratic term $\mathcal{Q}$} We follow a
  similar procedure to \textbf{Step 3} above. First split $S_{< k-4}
  Ad(O_{<k-C}) + S_{> k-4} Ad(O_{<k-C}) $. For the second term one can
  proceed as in \textbf{Step 4b} above using \eqref{psi_decomp1},
  \eqref{remainder_est}, and \eqref{decomp_holder}.

  Therefore we only need to consider the effect of the first term, for
  which we will prove the trilinear bound:
  \begin{multline}
    \lp{Q_{k}\cdot Op( ad(\Psi_{k'}ad(\Psi_{l})b_{<k-4})(t,x;D)\cdot
      Q_{<k-C}P_0 }{L^\infty(L^2)\to L^2(L^2)} \\ \lesssim \
    \epsilon^2\,
    2^{-\frac{1}{2}k}2^{\frac{1}{4}(k-k')}2^{\frac{1}{6}(k-l)} \sup_t
    \lp{B_{<k-4}(t)}{L^2\to L^2} \ , \label{quad_psi_loc}
  \end{multline}
  for $k'\geqslant l\geqslant k$.
We decompose the symbol $ad(\Psi_{k'}) ad(\Psi_{l})$ in terms of the angles,
  \begin{equation}
    \sum_{\theta\gtrsim 2^{\frac{1}{2}(k-k')}} \!\!\!
    ad(\Psi_{k'}^{(\theta)}) ad (\Psi_{l})
    + \sum_{\substack{\theta \gtrsim 2^{\frac{1}{2}(k-k')}\\
        \theta' \ll 2^{\frac{1}{2}(k-l)} }} \!\!\!
    ad(\Psi_{k'}^{(\theta)}) ad(\Psi_{l}^{(\theta')}) +  \sum_{\substack{\theta \ll 2^{\frac{1}{2}(k-k')}\\
        \theta \ll 2^{\frac{1}{2}(k-l)} }} \!\!\!
    ad(\Psi_{k'}^{(\theta)}) ad(\Psi_{l}^{(\theta')})
\ = \ T_1 + T_2 + T_3\ . \notag
  \end{equation}
  For the term $T_1$ put the first factor in $DL^3(L^\infty)$ and the
  second in $DL^6(L^\infty)$. This gives us dyadic terms in
  LHS\eqref{quad_psi_loc}$(T_1)\sim
  2^{-\frac{1}{2}k}2^{\frac{1}{4}(k-l)}2^{\frac{1}{6}(k-l')}$.  For
  the term $T_2$ do the opposite, which yields a similar bound.
  Finally, for the term $T_3$ a frequency modulation analysis shows
  that at least one of the two factors has modulation $\geq k$. Then
  we use \eqref{psi_decomp2-mod} to place that factor in $DL^2(L^\infty)$
and simply bound the remaining factor in $DL^\infty L^\infty$.

 \bigskip

  \step{5}{Estimating the cubic term $\mathcal{C}$} In this case we
  can gain $2^{\frac{1}{6}(k-k')}$ directly through the use of
  \eqref{psi_decomp1} and three $DL^6(L^\infty)$. Further
  details are left to the reader.

\end{proof}

\subsection{ The $N_0$ and $N_0^*$ bounds in 
\eqref{N_mapping},  \eqref{dtS_mapping} and
 \eqref{S_mapping}. }

We are now ready to conclude the proof of the first part of Theorem~\ref{t:para}.

\begin{proof}[Proof of \eqref{N_mapping} for $Z = N_0, N_0^*$]
  By duality it suffices to prove the $N_0^*$ bound for both the left
  and the right calculus.  The $L^\infty L^2$ bound follows from the
  fixed time $L^2$ bound.  The $X^{0,1}_{\infty}$ bound is also
  straightforward when we go from high to low modulation. 
It remains to consider the case of low modulation input and high 
modulation output.  Precisely, we need to show that 
\begin{equation}\label{lo-hi}
\| Q_k Op(Ad(O))Q_{<k-C}P_0\|_{L^\infty L^2 \to L^2} \lesssim \epsilon 2^{-\frac{k}2}
\end{equation}
From here on, we specialize to the left calculus.
By  \eqref{mod_loc}, it remains to estimate
\[
\| Q_{k} Op(Ad(O_{<k-C})) Q_{<k-C} P_0\|_{L^\infty L^2 \to L^2 L^2} \lesssim
\epsilon  2^{-\frac{k}2}   
\]
Here we can harmlessly replace $Ad(O_{<k-C})$ by $S_{>k-4} Ad(O_{<k-C})$.
But then we can conclude using \eqref{remainder_est}.

To prove \eqref{lo-hi} for the right calculus, we use duality to switch to the 
left calculus bound
\begin{equation}\label{lo-hi-left}
\| Q_{<k-C} Op(Ad(O))Q_{<k}P_0\|_{ L^2 \to L^1 L^2} \lesssim \epsilon 2^{-\frac{k}2}
\end{equation}
Then we can conclude the proof in the same manner as before.
\end{proof}

\begin{proof}[Proof of \eqref{dtS_mapping} for $Z = N_0, N_0^*$]
Here we repeat the above analysis with $Ad(O)$ replaced by 
$\partial_t(Ad(O)) = ad(O_{:t}) Ad(O)$. We remark that
\[
\partial_h\partial_t(Ad(O))) = ad(\partial_t(\Psi_h)) Ad(O) + ad(\Psi_h) ad(O_{:t}) Ad(O)
\]
and all terms above are of the same form as above, possibly 
with $ Ad(O))$ harmlessly replaced by $ad(O_{:t}) Ad(O)$. 
\end{proof}

\begin{proof}[Proof of \eqref{S_mapping} for $Z = N_0, N_0^*$]
By duality it suffices to consider the case $Z =N_0^*$.
In view of the $L^2$ bound proved earlier,
it suffices to show that 
\[
\| Q_k Op(Ad((O)) Op(Ad(O))^* Q_{<k-C}\|_{L^\infty L^2 \to L^2} \lesssim \epsilon 2^{-\frac{k}2}
 \]
But this is a consequence of two bounds,
\[
\| Q_k Op(Ad(O))Q_{<k-C}\|_{L^\infty L^2 \to L^2} \lesssim \epsilon 2^{-\frac{k}2}
 \]
and 
\[
\| Q_{> k-C/2}  Op(Ad(O))^* Q_{<k-C}\|_{L^\infty L^2  \to L^2} \lesssim \epsilon 2^{-\frac{k}2}
 \]
both of which follow from \eqref{lo-hi}.
\end{proof}

\subsection{Strichartz and null frame norm estimates} 
\label{s:ssharp}

Here we briefly outline how to prove the bound
\eqref{parametrix_bound}. In fact, the argument for this bound follows
exactly like the proof of (83) in section 11 of \cite{MKG}. One
replaces (114) in \cite{MKG} by the $L^2$-boundedness of the operators
$Op(Ad(O_{\pm})_{<k})(t, x, D)$, the dispersive bounds (108), (110)
in \cite{MKG} by the bounds \eqref{wave-decay0},
\eqref{wave-decay-cube0}, and the bound (118) in \cite{MKG} by
\eqref{mod_loc}.


\section{ Error estimates}
\label{s:error}
Here we again simplify notation by writing $O_{<0} = O$. 
The goal of this section is to consider the conjugation error
\[
E = \Box^p_{A_{<0}}Op(Ad(O))  - Op(Ad(O)) \Box 
\]
and prove the bound \eqref{error_ests} in Theorem~\ref{t:para}.

Commuting $\Box $ we have
\[
\begin{split}
E = &  2 Op(ad(A_{<0,\alpha}) Ad(O)) \partial^\alpha + 2  Op( ad(A_{<0,\alpha}) ad(O^{;\alpha})
Ad(O)) + 2 Op(ad(O_{;\alpha}) Ad(O))\partial^\alpha  
\\ & + Op(ad(\partial^\alpha O_{;\alpha}) Ad(O))
+ Op(ad(O_{;\alpha}) ad(O^{;\alpha}) Ad(O))
\\
 = & \ 2 Op(ad(A_{<0,\alpha} + \Psi_{<0,\alpha}) Ad(O)) \partial^\alpha 
\\ & + 2 Op(ad(O_{;\alpha} - \Psi_\alpha) Ad(O))\partial^\alpha  
\\ & \   + 2  Op( ad(A_{<0,\alpha}) ad(O^{;\alpha}) Ad(O)) 
+ Op(ad(O_{;\alpha}) ad(O^{;\alpha}) Ad(O))
\\ & \ + Op(ad(\partial^\alpha O_{;\alpha}) Ad(O))
\\ = & \ E_1 + E_2 + E_3 + E_4
\end{split}
\]
Here the main difficulty is to estimate the term $E_1$, which not only contains
the input of $A_{j,<0}^{pert,\pm}$ but also the full input from $A_0$.
We carry out a good portion of the analysis in Section~\ref{s:pert+}, modulo 
a single interaction scenario which is more extensive
and requires more than the $S$ norm of $A$ ; this is relegated to the
last section~\ref{s:tri}.  The remaining terms $E_2$, $E_3$ and $E_4$ 
are dealt with in Section~\ref{s:error-good}. These are more in line with previous estimates,
and only require the $S^1$ norm of $A_x$.

\subsection{The estimate for $E_1$}
\label{s:pert+}
We recall that 
\[
\Psi_{k,+}(t,x,\xi) \ = \
   L^\omega_- \Delta^{-1}_{\omega^\perp} (A_{j,k}^{main} \cdot \omega_j), 
\qquad
A_{j,k}^{main} = \Pi^{\omega}_{>\delta k}  \Pi^\omega_{cone} 
A_{j,k} 
\]
where 
\[
L^\omega_- = \partial_t- \omega \nabla_x
\]
Replacing the operator $D_t$ by $- |D_x|$  we produce a first error, namely
\[
Op ( ad(A_0 + \partial_0 \Psi) Ad (O)) (D_t +|D_x|) 
\]
which is easily dealt with using $DL^2 L^\infty$ disposability bounds 
for $A_0$ and $\partial_t \Psi$. We are left with 
\[
E_{1} = Op (ad( A_j \cdot \omega+A_0 + L^\omega_{+} \Psi^+) Ad(O))   
\]
Now we use 
\[
L^\omega_{+} L^\omega_- \Delta^{-1}_{\omega^\perp} = \Box \Delta^{-1}_{\omega^\perp}
- 1
\]
to write
\[
G:= A_j \cdot \omega+A_0 + L^\omega_{+} \Psi^+ = G_{cone} + G_{null} + G_{out}
\]
where
\[
\begin{split}
G_{cone} = & \ \Box  \Delta^{-1}_{\omega^\perp} \Pi^{\omega}_{>\delta k} A_{j,cone} \omega_j
+ \Pi^{\omega}_{< \delta k} A_{j,cone} \omega_j
+ A_{0,cone} 
\\
G_{null} = & \ A_{j,null} \omega_j + A_{0,null}
\\
G_{out}  = & \ A_{j,out} \omega_j + A_{0,out}
\end{split}
\]
We seek to prove that $Op(ad(G)Ad(O)): N^* \to N$. We do this in two stages.
First we will show that we can dispense with $O$, and simply prove that
\begin{equation}\label{no-O}
Op(ad(G)) : S_0 \to N
\end{equation}
Since $Op(Ad(O))$ is bounded from $S_0^\sharp$ into $S^0$, in order to
achieve this it suffices to show that
\begin{equation}\label{kill-O}
Op(ad(G)Ad(O)) - Op(ad(G)) Op(Ad(O)): N^* \to N
\end{equation}
This latter bound will not follow immediately from pdo calculus, since
$G$ is not smooth with respect to $\xi$ on the unit scale. Instead,
our strategy will be to first peel off a contribution which is bad from the perspective of 
pdo calculus but has a good decomposable structure. For this we consider 
the pieces $G_h^{(\theta)}$ of $G$, which are localized at frequency $2^h$ and 
angle $\theta$ with respect to $\omega$. In view of the  bounds \eqref{A_decomp2} and \eqref{psi_decomp2}
they satisfy
\[
\| G_h^{(\theta)}\|_{DL^2 L^\infty} \lesssim \theta^\frac32 2^{\frac{h}2}
\]
These symbols are smooth in $\xi$ on the $\theta$ scale, so it is natural to match them
against symbols which are smooth in $x$ on the $\theta^{-1}$ scale. Thus, let $h_\theta$
be defined by $2^{h_\theta} = \theta$. Then we decompose the above difference as
\[
\begin{split}
D_h^{(\theta)} = & \  Op(ad(G_h^{(\theta)})Ad(O)) - Op(ad(G_h^{(\theta)})) Op(Ad(O))
\\
= & \ \int_{h_\theta}^0 Op(ad(G_h^{(\theta)}) ad (\Psi_k) Ad(O_{<k})) dk
\\
& \ - \int_{h_\theta}^0 Op(ad(G_h^{(\theta)}) Op( ad(\Psi_k) Ad(O_{<k})) dk
\\
+ & \ Op(ad(G_h^{(\theta)})Ad(O_{<h_\theta})) - Op(ad(G_h^{(\theta)})) 
Op(Ad(O_{<h_\theta}))
\end{split}
\]
For the first term, decomposable estimates show
\[
\| Op(ad(G_h^{(\theta)}) ad (\Psi_k) Ad(O_{<k})) \|_{L^\infty L^2 \to L^1 L^2}
\lesssim \| G_h^{(\theta)}\|_{DL^2 L^\infty} \| \Psi_k\|_{DL^2 L^\infty} \lesssim  \theta^\frac32 2^{\frac{h}2} 2^{(-\frac12-\delta)k}
\]
which is favorable in view of the range $ \theta < 2^k < 1$.
A similar argument applies for the second term. For the third term, instead, we can use the
pdo calculus. For $|\alpha| \geq 1$ we have 
\[
\|\partial_\xi^\alpha G_h^{(\theta)}\|_{DL^2 L^\infty}  \leq c_\alpha  \theta^{-|\alpha|}
\theta^\frac32 2^{\frac{h}2}
\]
while (using Lemma~\ref{lem:Obounds1})
\[
\| \partial_x^\alpha Ad(O_{<h_\theta}) \|_{L^\infty L^2 \to L^2 L^2}  \lesssim \theta^{|\alpha|-\frac12-\delta} 
\]
It follows that
\[
\| Op(ad(G_h^{(\theta)})Ad(O_{<h_\theta})) - Op(ad(G_h^{(\theta)})) 
Op(Ad(O_{<h_\theta}))\|_{L^\infty L^2 \to L^1 L^2} \lesssim 
\theta^\frac12 2^{\frac{h}2} \theta^{\frac12-\delta}
\]
which again suffices. Thus the bound \eqref{kill-O} is proved. We now
return to \eqref{no-O}.

Corresponding to the partition of $G$  into three parts we will also partition 
\[
E_1 = E_{1,cone} + E_{1, null} + E_{1,out}
\]
In this section we will estimate  $ E_{1,cone}$ and $E_{1,out}$. We will postpone the bound
for $E_{1,null}$ for the next section.

\bigskip

{\bf The bound for $E_{1,cone}$.} The redeeming feature of $E_{1,cone}$
is that the modulation localization and the angle are mismatched and
that forces a large modulation on either the input or the
output. Precisely, consider the $G_{cone}$ component
$G_{cone,k}^{(\theta)}$ at frequency $k$ and angle $\theta$.
Then $G_{cone,k}^{(\theta)}$ has modulation at most $2^k \theta^2$, whereas 
either the input or the output must have modulation at least $2^k \theta^2$.
Hence we can use the $L^2$ norm for either the input or the output,
therefore it suffices  to have $L^2 L^\infty$ disposability for the terms in 
$G_{cone}$. Precisely, we obtain 
\[
\| E_{1,cone}\|_{N^* \to N} \lesssim \sum_{k < 0} \sum_{\theta<1}
\theta^{-1} 2^{-k/2} \| G_{cone,k}^{(\theta)}\|_{DL^2L^\infty}
\]
The nontrivial business is to insure summation.
In the second term in $G_{cone}$ we gain from the angle, and thus also in $k$. 
In the first term we use disposability derived from the $L^2$ bound for $\Box A_k$
therefore we gain in angle, and $\ell^1$ summation in $k$. Same for the third term.

\bigskip 
{\bf The bound for $E_{1,out}$.}. Again the modulation localization
and the angle are mismatched and that forces a large modulation on
either the input or the output. Precisely, consider the $G_{cone}$
component $Q_{k+2j} G_{out,k}^{(\theta)}$ at frequency $k$ and angle $\theta$.
Then $G_{cone,k}^{(\theta)}$ has modulation  $2^{k+2j} \geq 2^k \theta^2$,
whereas either the input or the output must have modulation at least
comparable.  Hence we can again use the $L^2$ norm for either the input
or the output, therefore it suffices to have $L^2 L^\infty$
disposability for the terms in $G_{cone}$. We obtain
\[
\begin{split}
\| E_{1,cone}\|_{N^* \to N} \lesssim & \sum_{k < 0}  \sum_{j < 0} \sum_{\theta < 2^j}
 2^{-(k+2j)/2} \| Q_{k+2j} G_{out,k}^{(\theta)}\|_{DL^2L^\infty}
\\
\lesssim & 
\sum_{k < 0}  \sum_{j < 0} \sum_{\theta < 2^j} 2^{-\frac{(k+2j)}{2}} \theta \theta^{\frac32} 2^{2k}\big\|P_kQ_{k+2j}A_x\big\|_{L^2L^2}
+ 2^{-(k+2j)/2} \theta^{\frac32} 2^{2k}\big\|P_kA_0\big\|_{L^2L^2}
\end{split}
\]
The first term comes from $A_j$ and the second from $A_0$. The latter has $\ell^1$ 
dyadic summation, while for the former we use Proposition~\ref{prop:refinedbox}. 

\bigskip 

{\bf The bound for $E_{1,null}$.}
We can dispense with the case when either the input or the output have high 
modulation ($\gtrsim  2^k \theta^2$, where  $k$, $\theta$ stand for the   
frequency, respectively the angle of $A$) as in the case of $E_{1,cone}$.
We are then left with the expression
\[
\H^* Op(ad (A_{\alpha,<0})) \partial^\alpha C
\]
The bound for this expression is stated in the following lemma, whose proof
is relegated to the next section:

\begin{lemma}\label{l:ttri}
  Suppose that $A$ has $S^1$ norm at most $\epsilon$ and solves the YM-CG
  equation in a time interval $I$.  Extend $A_x$ to a free wave
  outside $I$, and $A_0$ by $0$.  Then for $C$ at frequency $1$ we have the estimate
\begin{equation}
\|\H^* Op(ad (A_{\alpha,<0})) \partial^\alpha C\|_{N} \lesssim \epsilon    \|C\|_{S}
\end{equation}
\end{lemma}

\bigskip

\subsection{The estimates for $E_2$, $E_3$ and $E_4$}
\label{s:error-good}
For these terms we can directly use the decomposability bounds bounds
on $\Psi$ and $O$ in the previous sections.  We consider them
successively.

\subsubsection{The $E_2$ term.}
For the second term in the error we recall that 
\[
\partial_h O_{;\alpha} = \Psi_{h,\alpha} + [\Psi_{h},  O_{;\alpha}]
\]
Thus, repeatedly expanding the symbol $ad(O_{;\alpha} - \Psi_{\alpha})
Ad(O)$ (by means of \eqref{eq:intident}) with respect to $h$, we are left with an integral with respect
to decreasing $h$'s of expressions of the form
\[
ad(\Psi_{h_1}) ad(\partial_\alpha\Psi_{h_2}) Ad(O_{<h_2}), \quad 
\cdots \quad ad(\Psi_{h_1}) \cdots ad(\Psi_{h_5}) ad(\partial_\alpha\Psi_{h_6}) Ad(O_{<h_6})
\]
plus a final remainder term
\[
ad(\Psi_{h_1}) \cdots ad(\Psi_{h_5}) ad(O_{<h_6;\alpha})
 Ad(O_{<h_6})
\]
with possibly changed order of factors.

 For the sixth-linear terms we use $DL^6L^\infty$ bounds for all factors
(in particular we need this for $O_{<h_6;\alpha}$ with no loss; $DL^\infty L^\infty$
would also do by reiterating once more).

For the lower order expressions we are in the same situation as in the MKG
case, with the critical difference that the $\Psi$'s may now have nonzero modulations.
We discuss the second order term, as all higher order terms are similar.
\[
ad(\Psi_{h_1}) ad(\partial_\alpha\Psi_{h_2}) Ad(O_{<h_2}) \partial^\alpha
\]
Replacing $\partial^0$ by $- |\xi|$ (with a better error)
this becomes
\[
ad(\Psi_{h_1}) ad(L^+ \Psi_{h_2}) Ad(O_{<h_2}) |\xi|
\]
and doing the symbol computation, this has the form
\[
D_2 = ad(\Psi_{h_1}) ad(A_{j,h_2}^{main} \omega_j) Ad(O_{<h_2}) |\xi|
\]
Now we do an angle/modulation analysis. We begin with angles, and denote 
by $\theta_1$, $\theta_2$ the two angles. Then by \eqref{psi_decomp2}
and \eqref{A_decomp2} we can first estimate 
\[
\|D_2\|_{L^\infty L^2 \to L^1 L^2} \lesssim \|  \Psi_{h_1}^{(\theta^1)} \|_{DL^2 L^\infty}
\|  A_{j,h_2}^{main,(\theta^2)} \cdot \omega \|_{DL^2 L^\infty} \lesssim 2^{(h_2-h_1)/2} \theta_1^{-\frac12}
\theta_{2}^\frac32
\]
This is favorable if $2^{h_2} \theta_2^2 \gtrsim 2^{h_1} \theta_1^2$. If this is not the 
case, then either one of the factors or the input or the output must have modulation
at least as large as $ 2^{h_1} \theta_1^2$. This cannot be the case for 
$A_{j,h_2}^{main,(\theta^1)}  $ by definition, so we have three scenarios to consider:

a) High modulation input.
Then by \eqref{psi_decomp2}
and \eqref{A_decomp2} we have
\[
\|D_2 B\|_{L^1 L^2} \lesssim 2^{-h_2/2} \theta_2^{-1} 
\|  \Psi_{h_1}^{(\theta^1)} \|_{DL^6 L^\infty} \|  A_{j,h_2}^{main,(\theta^2)} \cdot \omega  \|_{DL^3 L^\infty}
\| B\|_{N^*} \lesssim 2^{(h_2-h_1)/6} \theta_2^{2-\frac16}
\]
which suffices.

b) High modulation output where we have exactly the same bound.

c) High modulation on $\Psi_1$. Then we can use \eqref{psi_decomp2-mod}
for its $DL^2L^\infty$ bound.

 \subsubsection{The term $E_3$}

In this term we have high frequencies to spare.
 For $[ O_{;\alpha}, [O_;^\alpha, Op(O)\cdot ] ]$ we need some mild 
$L^2 L^\infty$ disposability estimate for $O_{;\alpha}$.
Similarly, for $A_\alpha$ we can use an $L^2 L^\infty$ bound.

\subsubsection{The term $E_4$}

For $ad(\partial^\alpha  O_{;\alpha}) Ad(O)$ we expand in $h$ 
\[
ad(\partial^\alpha  O_{;\alpha}) Ad(O) =  \int_{-\infty}^0 (ad(\Box \Psi_{h}) 
+  \partial^\alpha [ \Psi_{h},O_{<h;\alpha}]  + 
ad(\partial^\alpha  O_{<h;\alpha}) ad(\Psi_h)) Ad(O_{<h}) dh
\]
For the second and third term we use $L^2 L^\infty$ disposability for
$\Psi_h$ and $O_{<h;\alpha}$, with room to spare.  For the first term,
Consider the component $ad(\Box \Psi_{h}^{(\theta)})$ at angle $\theta$
and reexpand with $2^{h_\theta} = \theta^2 2^h$:
\[
ad(\Box \Psi_{h}^{(\theta)})  Ad(O_{<h}) =  
ad(\Box \Psi_{h}^{(\theta)})  Ad(O_{<h_\theta-C})+
\int_{h_{\theta} -C}^h
ad(\Box \Psi_{h}^{(\theta)})  ad(\Psi_{h_1}) Ad(O_{<h_1}) dh_1
\]
For the integrand we can use two $DL^2L^\infty$ bounds to estimate
\[
\|\Box \Psi_{h}^{(\theta)}\|_{DL^2 L^\infty} \| \Psi_{h_1}\|_{DL^2 L^\infty} \lesssim
\theta^\frac32 2^{\frac{3h}2} 2^{-(\frac12+\delta) h_1} 
\]
which is favorable due to the range of $h_1$. For the leading term, using \eqref{remainder_est}, we replace $Ad(O_{<h_\theta-C})$ by $S_{< h_\theta-4}
Ad(O_{<h_\theta-C}) $.  At this stage we are left with the operator
\[
Op( ad(\Box \Psi_{h}^{(\theta)}) S_{< h_\theta-4}
Ad(O_{<h_\theta-C}) ) 
\]
Given the frequency localization of $\Psi_{h}^{(\theta)}$, the space-time frequency 
interaction analysis shows that either the input or the output must have modulations 
at least $2^h \theta^2$. Then we can conclude using the $DL^2L^\infty$ disposability 
of $\Box \Psi_{h}^{(\theta)}$ in \eqref{boxpsi-disp}.

\end{proof}

\section{Trilinear forms and the second null structure}
\label{s:tri}

Here we prove Lemma~\ref{l:ttri} and Lemma~\ref{l:tri}, which we restate for convenience:

\begin{lemma}\label{l:ttri-re}
a)  Suppose that $A$ has $S^1$ norm at most $\epsilon$ and solves the YM-CG
  equation in a time interval $I$.  Extend $A_x$ to a free wave
  outside $I$, and $A_0$ by $0$.  Then for $C_k$ at frequency $2^k$ we have the estimate
\begin{equation}
\|\H^* [A_{\alpha,<k},\partial^\alpha C_k]\|_{N} \lesssim \epsilon    \|C_k\|_{S}
\end{equation}

b)  Suppose in addition that $B \in S^s$ solves the linearized equation
  \eqref{ym-cg-lin} in a time interval $I$. Extend $B_j$ outside $I$ as free waves,
and $B_0$ by zero. Then for $s < 1$, close to $1$ we have the global estimate
\begin{equation}
\|\H^* [B_{\alpha,<k}, \partial^\alpha C_k]\|_{N^{s-1}} \lesssim \epsilon \|B\|_{S^s}    \|C_k\|_{S^1}
\end{equation}
\end{lemma}

The proofs for the two parts are quite similar, and hinge on a double null structure
in the main trilinear expression arising when one replaces the first factor in the 
expressions above with the solutions of the corresponding $\Box$ equation for $A_x$ 
and $B_x$, respectively the $\Delta$ equation for $A_0$ and $B_0$. 

\begin{proof}[Proof of Lemma~\ref{l:ttri-re}]
a)  To better frame the question, denote by $2^h$, $\theta$ the
  frequency, respectively the angle of $A$.  Then the $\H^*$ operator 
selects the cases where both the input and the output are at modulation 
less than $2^{h} \theta^2$. 

Our first tool here is to use the $Z$ norm bounds \eqref{zsn}, \eqref{zsn-ell}.
To bound (most of) $A_x$ and $A_0$ we use their  equations \eqref{ym-cg}, respectively \eqref{a0}.
We claim that the following hold:
\begin{equation}\label{z-good}
\begin{split}
\| \Box A_j - \H \P [A_i,\chi_I \partial_j A_i]\|_{\ell^1 \Box Z} \lesssim & \ \epsilon^2  
\\
\| \Delta A_0 - \H[A_i,\chi_I \partial_0 A_i]\|_{\ell^1 \Box^\frac12 \Delta^\frac12 Z} \lesssim & \ \epsilon^2  
\end{split}
\end{equation}
For this we consider all other terms in the equations for $A_j$ and $A_0$, which we 
recall here:
\[
\Box A_j  = \P\left(  [A^\aa,  \p_j A_\aa] - 2 [ A^\alpha, \partial_\alpha A_j] 
- [ \partial_0 A_0, A_j]  - [A^\alpha,[A_\alpha,A_j]] \right)  
\]
\[
\Delta A_0 = [A_j,\partial_0 A_j] - 2 [ A_j, \partial_j A_0] - [A_j,[A_j,A_0]]
\]
 Here we seemingly pay a price for working in an interval $I$, as both right
hand sides need to be multiplied by the characteristic function $\chi_I$ of $I$.
However, this turns out to be harmless, because we can always place $\chi_I$ 
on the differentiated factor, and still retain the use of the $S$ norm.

{\em (i) Cubic terms $A^3$.} These are placed in $\ell^2 L^1 L^2$
which suffices by \eqref{l12-to-z}. (we do need to gain $\ell^1$ summability
in $k$).

{\em (ii) $[A_j, \partial_j A_0]$ and $[\partial_0 A_0, A_j]$.} are also in $\ell^1 L^1L^2$
by using $L^2 \dot H^\frac12$ for $\nabla A_0$ and $L^2 L^6$ for $A_j$.

{\em (iii) The term $[A_0, \partial_0 A_j]$.} The low-high case is the worst, 
but even then we can use Strichartz to produce $L^1 L^\infty$.

{\em (iv) High-low interactions in the quadratic terms $A_j \nabla A_k$.}
This is where we use \eqref{hl-to-z}.

{\em (v) High-high interactions in $[A_j,\partial_j A_k]$.}
 Here we can take the derivative out and estimate as in the high-low case
via \eqref{hl-to-z}..
 
{\em (vi) High-high interactions in $[A_j,\partial_\alpha A_j].$ with at least one 
high modulation} Here by estimating one factor in $L^2$ we can gain 
in terms of high frequencies, see \eqref{hh-to-z}.

This concludes the proof of \eqref{z-good}. In view of \eqref{zsn}, \eqref{zsn-ell},
this leaves us with one remaining case:

\bigskip

{\em (Final case) High-high interactions in $[A_j,\partial_\alpha A_j]$ with two low modulations.}
Here we need to combine the $\Box^{-1} A_j$ and $\Delta^{-1} A_0$ contributions
in order to gain an additional cancellation.
Omitting the frequency and modulation localizations, the expression  is as follows: 
\[
\begin{split}
L = & (\Box^{-1} \P [A_j, \partial_k A_j] \partial_k F + \Delta^{-1} [A_j, \partial_0 A_j])  |D| F
\\
= & \ \Box^{-1} [A_j, \partial_k A_j] \p_k F - \frac{\partial_k\partial_i} {\Box \Delta} 
[A_j, \partial_i A_j] \partial_k F - \Delta^{-1} [A_j, \partial_0 A_j]  \partial_0 F
\\
= & \  \ \Box^{-1} [A_j, \partial_\alpha A_j] \partial^\alpha F  -  
\frac{\partial_k\partial_i} {\Box \Delta} [A_j, \partial_i A_j] \partial_k F 
+ \frac{\partial_0^2}{\Box \Delta}  [A_j, \partial_0 A_j]  \partial_0 F
\\
= & \  \ \Box^{-1} [A_j, \partial_\alpha A_j] \partial^\alpha F  -  
\frac{\partial_\alpha\partial_i} {\Box \Delta} [A_j, \partial_i A_j] \partial^\alpha F
-  \frac{\partial_0\partial_i} {\Box \Delta} [A_j, \partial_i A_j] \partial_0 F 
+ \frac{\partial_0^2}{\Box \Delta}  [A_j, \partial_0 A_j]  \partial_0 F
\\
= & \  \ \Box^{-1} [A_j, \partial_\alpha A_j] \partial^\alpha F  -  
\frac{\partial_\alpha\partial_i} {\Box \Delta} [A_j, \partial_i A_j] \partial^\alpha F
-  \frac{\partial_0\partial_\alpha} {\Box \Delta} [A_j, \partial^\alpha A_j] \partial_0 F 
\end{split}
\]
The estimate for this term is exactly the trilinear bound in \cite{MKG}, see (136) - (138) in Theorem 12.1 there. 

\bigskip

b) This is similar to the  proof in part (a), with two differences:

i) There is an additional gain in the  low frequency input, which eliminates any need to 
control $\ell^1$ norms.

ii) There is a small additional loss in high-high interactions in $\Box
A_x$ and $\Delta A_0$. However, this is harmless as in all cases we have a small
high frequency gain (including, notably, the trilinear case).

\end{proof}

\bibliography{wm}
\bibliographystyle{plain}

\end{document}